\documentclass[11pt,a4paper]{amsart}
\usepackage{amssymb,amsthm,enumitem,amsmath,mathrsfs,comment}
\usepackage[noadjust]{cite}
\usepackage{hyperref}
\usepackage[capitalise]{cleveref}
\usepackage{color}
\usepackage{kbordermatrix}

\theoremstyle{plain}
\newtheorem{theorem}{Theorem}[section]
\newtheorem{lemma}[theorem]{Lemma}
\newtheorem{corollary}[theorem]{Corollary}
\newtheorem{proposition}[theorem]{Proposition}
\newtheorem{conjecture}[theorem]{Conjecture}
\newtheorem{claim}{}[theorem] 
\newenvironment{subproof}{\begin{proof}[Subproof.]}{\end{proof}}

\theoremstyle{definition}

\Crefname{theorem}{Theorem}{Theorems}
\crefrangeformat{claim}{#3#1#4--#5#2#6}
\crefmultiformat{claim}{#2#1#3}{ and~#2#1#3}{, #2#1#3}{ and~#2#1#3}
\crefrangemultiformat{claim}{#3#1#4--#5#2#6}{ and~#3#1#4--#5#2#6}{, #3#1#4--#5#2#6}{ and~#3#1#4--#5#2#6}

\DeclareMathOperator{\cl}{cl}
\DeclareMathOperator{\fcl}{fcl}
\DeclareMathOperator{\si}{si}
\DeclareMathOperator{\co}{co}
\newcommand{\del}{\backslash}
\newcommand{\ba}{\backslash}
\newcommand{\cocl}{\cl^*}

\newcommand{\lc}{\sqcap}

\newcommand{\dY}{$\Delta$\nobreakdash-$Y$}
\newcommand{\Yd}{$Y$\nobreakdash-$\Delta$}

\newcommand{\TQ}{\mathit{TQ}}
\newcommand{\AG}{\mathit{AG}}
\newcommand{\GF}{\textrm{GF}}
\newcommand{\seq}[1]{[#1]}
\newcommand{\utfutf}{\{U_{2,5},U_{3,5}\}}

\setenumerate{label=\rm(\roman*),midpenalty=2}

\usepackage[labelformat=simple]{subcaption}
\usepackage{tikz}
\usepackage{tkz-graph}
\tikzstyle{VertexStyle} = [shape = circle, draw, fill]
\tikzset{pre/.style={-}}    
\tikzstyle{every node}=[circle, inner sep=0pt, minimum width=4pt]
\usetikzlibrary{backgrounds}
\title[The excluded minors for $2$-regular matroids]{The excluded minors for $2$- and $3$-regular matroids}

\author[N.\ Brettell]{Nick Brettell}
\address{School of Mathematics and Statistics, Victoria University of Wellington, New Zealand}
\email{nick.brettell@vuw.ac.nz}
\author[J.\ Oxley]{James Oxley}
\address{Department of Mathematics, Louisiana State University, Baton Rouge, Louisiana, USA}
\email{oxley@math.lsu.edu}
\author[C.\ Semple]{Charles Semple}
\address{School of Mathematics and Statistics, University of Canterbury, New Zealand}
\email{charles.semple@canterbury.ac.nz}
\author[G.\ Whittle]{Geoff Whittle}
\address{School of Mathematics and Statistics, Victoria University of Wellington, New Zealand}
\email{geoff.whittle@vuw.ac.nz}

\thanks{The first author was supported by a Rutherford Foundation postdoctoral fellowship.  The first, third and fourth authors were supported by the New Zealand Marsden Fund.}

\date{\today}


\begin{document}

\begin{abstract}
  The class of \emph{$2$-regular} matroids is a natural generalisation of regular and near-regular matroids.
  We prove an excluded-minor characterisation for the class of $2$-regular matroids.
  The class of \emph{$3$-regular} matroids coincides with the class of matroids representable over the Hydra-$5$ partial field,
  and the $3$-connected matroids in the class with a $U_{2,5}$- or $U_{3,5}$-minor are precisely those with six inequivalent representations over $\GF(5)$.
  We also prove that an excluded minor for this class has at most $15$ elements.
\end{abstract}

\maketitle

\section{Introduction}

Let $\mathbb Q(\alpha_1,\ldots,\alpha_n)$ denote the field obtained by extending the rational numbers by $n$ independent transcendentals $\alpha_1,\ldots,\alpha_n$. For a non-negative integer~$k$, a matroid is
\emph{$k$-regular} if it has a representation by a matrix $A$ over $\mathbb Q(\alpha_1,\ldots,\alpha_k)$
in which every non-zero subdeterminant of $A$ is in the set that consists of all integer powers of 
differences of distinct members of $\{0,1,\alpha_1,\ldots,\alpha_k\}$. 
The class of $0$-regular matroids coincides with the class of \emph{regular} matroids. 
The class of $1$-regular matroids is the class
of {\em near-regular} matroids, and is the class of matroids representable over all fields of size at
least~3~\cite{Whittle97}. Excluded minor characterisations for regular matroids are given by 
Tutte \cite{Tutte58} and for near-regular matroids by Hall, Mayhew and van Zwam~\cite{HMvZ11}. 

This paper focuses on $2$-regular and $3$-regular matroids. We prove the following:

\begin{theorem}
\label{intro1}
An excluded minor for either the class of $2$-regular or $3$-regular matroids has at most 15 elements.
\end{theorem}

This bound enables a computer search to be undertaken to find all excluded minors 
for $2$-regular matroids. This search is undertaken in \cite{BP20} and we are able to give the following
excluded-minor characterisation of $2$-regular matroids. All matroids mentioned below are described
in the appendix.

\begin{theorem}
\label{intro2}
  A matroid $M$ is $2$-regular if and only if $M$ has no minor isomorphic to
$U_{2,6}$, $U_{3,6}$, $U_{4,6}$, $P_6$,
$F_7$, $F_7^*$, $F_7^-$, $(F_7^-)^*$, $F_7^=$, $(F_7^=)^*$,
$\AG(2,3)\ba e$, $(\AG(2,3)\ba e)^*$, $(\AG(2,3)\ba e)^{\Delta Y}$, $P_8$, $P_8^-$, $P_8^=$, and $\TQ_8$.
\end{theorem}

It is natural to hope for an analogous result for $3$-regular matroids. A search in \cite{Brettell22}
uncovers all excluded minors for this class up to size 13. We conjecture that the list is complete. 
Note that $\Delta^{(*)}(U_{2,7})$ is a family of six matroids obtained from $U_{2,7}$ via $\Delta$-$Y$ exchange and dualising.

\begin{conjecture}
  \label{intro3}
  A matroid is $3$-regular if and only if it has no minor isomorphic to one of the following $33$ matroids:
  $$F_7, F_7^-, F_7^=, H_7,M(K_4)+e, \mathcal{W}^3+e,\Lambda_3, Q_6+e, P_6+e, U_{3,7},$$ and their duals; a matroid in $\Delta^{(*)}(U_{2,7})$; and
$\AG(2,3)\ba e$, $(\AG(2,3)\ba e)^*$, $(\AG(2,3)\ba e)^{\Delta Y}$, $P_8$, $P_8^-$, $P_8^=$, and $\TQ_8$.
\end{conjecture}

To resolve \cref{intro3}, we need to eliminate the possibility of excluded minors with 
14 or 15 elements. One strategy would be to do further work along the lines of that contained in this 
paper, with the goal of reducing the bound of 15. Another strategy would be to narrow the space
for a computer search by exploiting known properties of the structure of excluded minors.

The results of this paper are motivated by two problems. The first is to understand the class of 
matroids representable over all fields of size at least~4. The second is to find an explicit excluded-minor
characterisation for the class of $\GF(5)$-representable matroids. In the remainder of this introduction
we discuss the connection between our results and these problems, before outlining the approach taken to prove \cref{intro1}. 

\subsection*{Matroids representable over all fields of size at least 4} The
class of $0$-regular matroids is the class of matroids representable over all
fields, and the class of $1$-regular matroids is the class of matroids
representable over all fields of size at least~$3$ \cite{Whittle97}.  Let $\mathcal M_4$
denote the class of matroids representable over all fields of size at least~4.  One might
hope that the class of $2$-regular matroids is $\mathcal M_4$, but this is not
the case.  Rather, the class of $2$-regular matroids is properly contained in
$\mathcal M_4$.  It follows from Theorem~\ref{intro2}  that, up to duality, 
there are four excluded minors for $2$-regular matroids
that belong to $\mathcal M_4$: these are $P_8^-$, $\TQ_8$, $U_{3,6}$ and
$F_7^=$.  Nonetheless, a start has been made. Up to duality, any member of $\mathcal M_4$ that is not
$2$-regular
must contain one of these four matroids as a minor. It is possible that members of $\mathcal M_4$
containing either $P_8^-$, $\TQ_8$ or $U_{3,6}$ as minors form classes of bounded branch
width that can be described structurally, but it turns out that there are members of 
$\mathcal M_4$ of unbounded branch 
width containing $F_7^=$ as a minor.

For $n\geq 4$, let $x$ and $y$ be elements of $M(K_n)$ that are not contained in a triangle, 
that is, they correspond
to a matching in the underlying graph~$K_n$. Extend by adding a point freely to the line spanned by
$\{x,y\}$. Denote the resulting matroid by $M_n$. 
It is readily verified that $M_4\cong F_7^=$. It is also routine to see that $M_n\in \mathcal M_4$ for 
all $n\geq 4$. We now have a rich class of matroids contained in $\mathcal M_4$ that are
not $2$-regular.

All up, the news is not particularly optimistic. If one seeks an excluded-minor characterisation of 
$\mathcal M_4$, then, using current techniques, for each $N\in\{P_8^-,\TQ_8,U_{3,6},F_7^=\}$
you will need to perform
an exercise similar to the one that finds the excluded minors for 
$2$-regular matroids except with $U_{2,5}$ replaced by $N$. This will require an 
understanding of a certain class of $N$-fragile matroids. That, in itself, is likely to be 
a challenge. Even with that issue resolved, the bound obtained on the size of an excluded minor
is likely to be too large to enable a computer search to find them.

The structural approach may be more promising. If the classes containing  $P_8^-$,
$\TQ_8$ or $U_{3,6}$ truly are thin, then perhaps they can be explicitly described. 
It is possible that the class containing $F_7^=$ is not too difficult either. 
Perhaps there is a 
structure theorem analogous to the regular matroid decomposition theorem where 
the class obtained by taking minors of $M_n$ plays a role analogous to that of graphic matroids in
regular matroids.

In any case, the goal of obtaining a better understanding of $\mathcal M_4$ is no doubt a 
worthy one, and at least there are some plausible lines of attack.

\subsection*{Excluded minors for \texorpdfstring{$\GF(5)$}{GF(5)}} We prove in Lemma~\ref{3reglemma} that the class of
$3$-regular matroids coincides with the class of matroids representable over the 
Hydra-5 partial field $\mathbb H_5$. This partial field was introduced in \cite{PvZ10b}, where it is
shown that 3-connected $\mathbb H_5$-representable matroids with a $U_{2,5}$-minor have
exactly six inequivalent $\GF(5)$-representations. Resolving \cref{intro3}
would achieve the first step of a program for finding excluded minors for $\GF(5)$-representability
adumbrated in \cite{PvZ10b}. In essence, the program is as follows. 
There is a sequence of partial fields $\mathbb H_1, \mathbb H_2, \mathbb H_3, \mathbb H_4, \mathbb H_5$, where $\mathbb H_1 = \GF(5)$, and each other partial field has a homomorphism into $\GF(5)$.
For a $3$-connected matroid with a $U_{2,5}$-minor, the matroid is representable over $\mathbb H_2$ if it has at least two inequivalent representations over $\GF(5)$.
If such a matroid is representable over $\mathbb H_3$, or $\mathbb H_4$, then it has at least three, or at least four, inequivalent representations over $\GF(5)$, respectively.
If such a matroid is representable over $\mathbb H_5$, then it has six inequivalent $\GF(5)$-representations (and a matroid can have at most six inequivalent representations over $\GF(5)$~\cite{OVW96}).

Each
$\GF(5)$-representable excluded minor~$N$ for representability over $\mathbb H_5$ will be
a strong stabiliser for a  member $\mathbb P$ in $\{\mathbb H_1, \mathbb H_2, \mathbb H_3, \mathbb H_4\}$.
One then has to find the excluded minors for the $\mathbb P$-representable matroids that
contain an $N$-minor. This is a task similar to the one undertaken in this paper except that
one has to understand the $N$-fragile $\mathbb P$-representable matroids. 
This will typically be significantly more difficult than understanding the $2$- or $3$-regular 
$U_{2,5}$-fragile matroids.

It turns out that, modulo the truth of \cref{intro3}, there are, up to duality, ten excluded minors
for $\mathbb H_5$-representable matroids that are $\GF(5)$-representable. Hence the exercise
described above has to be repeated ten times. Moreover the process repeats, so that some of the
excluded minors found may well be $\GF(5)$-representable, although they will be placed further up 
in the hierarchy.

That is a massive undertaking. It is significantly more difficult than a related, but as yet unresolved,
problem. That problem is to find the excluded minors for the class of {\em dyadic} matroids.
It is known the class of dyadic matroids is the class of matroids
representable over $\GF(3)$ and $\GF(5)$. The excluded minors for
dyadic matroids can easily be deduced from a knowledge of the excluded minors of 
$\GF(5)$-representable matroids. But one would expect it to be significantly easier to find the 
excluded minors for dyadic matroids directly and, indeed, knowing the dyadic excluded minors
would in itself be a step towards  $\GF(5)$. So how difficult is this problem?

We know that any excluded minor for dyadic matroids that we do not understand must contain
either the non-Fano matroid $F_7^-$, its dual, or $P_8$, as a minor. Moreover these matroids are
strong stabilisers for the  class of dyadic matroids.  There is some hope of understanding
the dyadic $P_8$-fragile matroids, but understanding the dyadic $F_7^-$-fragile matroids 
seems much more difficult. Moreover, it is known that excluded minors for dyadic matroids
can be quite large. The computer search in \cite{BP20} has uncovered one with 16 elements.

The upshot is that current techniques are most probably inadequate for obtaining an excluded-minor
characterisation of dyadic matroids and certainly inadequate for $\GF(5)$. More optimistically,
one can pursue techniques that do not commit one to solving an $N$-fragility problem.
It is possible that real progress in the future could occur by developing such techniques.
Success in obtaining an excluded-minor characterisation of dyadic matroids would be a major
breakthrough.

\subsection*{Overview of the proof} The high-level approach of the proof Theorem~\ref{intro1}
is as follows. Let $M$ be a large excluded minor for the class of $2$-regular matroids such that $M$ 
has an $N$-minor, where $N$ is $U_{2,5}$  or $U_{3,5}$. 
By results in \cite{BWW22}, the matroid $M$ has a pair of elements $a,b$ such that $M\backslash a,b$ 
is $3$-connected and has an $N$-minor. Now, by results in \cite{BCOSW20}, 
$M\backslash a,b$ is close to being $N$-fragile. Clark et al.\  \cite{CMvZW16}
described the structure of large $2$-regular $\{U_{2,5},U_{3,5}\}$-fragile
matroids; in particular, they have path width three. 
In \cref{utfutfsec}, we recap and expand on the properties of these matroids that we require. 
It still remains to prove that $M\backslash a,b$ is itself  $\{U_{2,5},U_{3,5}\}$-fragile; 
we do this in \cref{utfutffragilesec}, using results from \cite{paper1} (presented in \cref{almostfragsec}) and \cref{utfutfsec}. Then, in \cref{fragilepropssec}, we prove 
some more properties of such a $\{U_{2,5},U_{3,5}\}$-fragile matroid $M\backslash a,b$. 
In \cref{ndtsec}, we show that if $M\backslash a,b$ is large enough, then we can use the path structure
to find a triple $\{a',b',c'\}$ such that $M\backslash a',b',c'$ is $3$-connected with an $N$-minor, 
and therefore $\{U_{2,5},U_{3,5}\}$-fragile. When there is a triple $a',b',c'$ such that 
$M\backslash a',b',c'$ is $\{U_{2,5},U_{3,5}\}$-fragile, each of the matroids $M\backslash a',b'$, 
$M\backslash a',c'$, and $M\backslash b',c'$ has path width three, and we show, in \cref{dtsec}, that we can 
use this to bound the size of $M$. Combining these results in \cref{concsec}, we complete the proof of \cref{intro1}.

\section{Preliminaries}

All unspecified terminology and notation follows Oxley \cite{Oxley11}.
When there is no ambiguity, we often denote a singleton set $\{q\}$ by $q$.
For a positive integer $n$, we denote the set $\{1,2,\dotsc,n\}$ by $[n]$.
For sets $X$ and $Y$, we say $X$ \emph{meets} $Y$ if $X \cap Y \neq \emptyset$, and $X$ \emph{avoids} $Y$ if $X \cap Y = \emptyset$.
A parallel class or series class is \emph{non-trivial} if it has size at least~$2$.
For a partition $\{X_1,X_2,\dotsc,X_m\}$ or an ordered partition $(X_1,X_2,\dotsc,X_m)$ we require that each cell $X_i$ is non-empty.

\subsection*{Segments, cosegments, and fans}

Let $M$ be a matroid.
A subset~$S$ of $E(M)$ with $|S| \ge 3$ is a \emph{segment} if every $3$-element subset of $S$ is a triangle.
We say that a segment~$S$ is an $\ell$-segment if $|S|=\ell$.
A \emph{cosegment} is a segment of $M^*$, and an $\ell$-cosegment~$S^*$ is a cosegment with $|S^*| = \ell$.
A subset~$F$ of $E(M)$ with $|F| \ge 3$ is a \emph{fan} if there is an ordering $(f_1, f_2, \dotsc, f_\ell)$ of $F$ such that
\begin{enumerate}[label=\rm(\alph*)]
  \item $\{f_1,f_2,f_3\}$ is either a triangle or a triad, and
  \item for all $i \in \seq{\ell-3}$, if $\{f_i, f_{i+1}, f_{i+2}\}$ is a triangle, then $\{f_{i+1}, f_{i+2}, f_{i+3}\}$ is a triad, whereas if $\{f_i, f_{i+1}, f_{i+2}\}$ is a triad, then $\{f_{i+1}, f_{i+2}, f_{i+3}\}$ is a triangle.
\end{enumerate}
When there is no ambiguity, we also say that the ordering $(f_1,f_2,\dotsc,f_\ell)$ is a fan.
If $F$ has a fan ordering $(f_1, f_2, \dotsc, f_\ell)$ where $\ell \geq 4$, then $f_1$ and $f_\ell$ are the \emph{ends} of $F$, and $f_2, f_3, \dotsc, f_{\ell-1}$ are the \emph{internal elements} of $F$.
We also say such a fan has \emph{size} $\ell$.
%
We say that a fan $F$ is \emph{maximal} if there is no fan that properly contains $F$.

For a rank-$r$ wheel $M(\mathcal{W}_r)$, there is a natural partition of the ground set into \emph{spoke} elements, each of which is contained in two triangles, and \emph{rim} elements, each of which is contained in two triads.
There is an analogous notion for elements in fans.
Let $F$ be a fan with ordering $(f_1, f_2, \dotsc, f_\ell)$ where $\ell \geq 4$, and let $i \in \seq{\ell}$ if $\ell \geq 5$, or $i \in \{1,4\}$ if $\ell=4$.
An element $f_i$ is a \emph{spoke element} of $F$ if $\{f_1, f_2, f_3\}$ is a triangle and $i$ is odd, or if $\{f_1, f_2, f_3\}$ is a triad and $i$ is even; otherwise $f_i$ is a \emph{rim element} of $F$.

\begin{lemma}[{\cite[Lemma~3.4]{OW00}}]
  \label{fanunique}
  Let $M$ be a $3$-connected matroid that is not isomorphic to $M(\mathcal{W}_3)$, and let $F$ be a $5$-element fan of $M$ with ordering $(f_1,f_2,\dotsc,f_5)$, where $\{f_1,f_2,f_3\}$ is a triangle.
  Then $\{f_1,f_2,f_3\}$ and $\{f_3,f_4,f_5\}$ are the only triangles of $M$ containing $f_3$, and $\{f_2,f_3,f_4\}$ is the unique triad of $M$ containing $f_3$.
\end{lemma}

\subsection*{Connectivity}

Let $M$ be a matroid and let $X \subseteq E(M)$.
The set $X$ or the partition $(X,E(M)-X)$ is \emph{$k$-separating} if $\lambda_M(X) < k$, where $\lambda_M(X) = r(X) + r(E(M)-X) - r(M)$.
A $k$-separating set $X$ or partition $(X,E(M)-X)$ is \emph{exact} if $\lambda_M(X) = k-1$.
If $X$ is $k$-separating and $|X|,|E(M)-X| \ge k$, then $X$ is a \emph{$k$-separation}.
The matroid $M$ is \emph{$k'$-connected} if $M$ has no $k$-separations for $k < k'$.
If $M$ is $2$-connected, we simply say it is \emph{connected}.
Suppose $M$ is connected.
If for every $2$-separation $(X,Y)$ of $M$ either $X$ or $Y$ is a parallel pair (or parallel class), then $M$ is \emph{$3$-connected up to parallel pairs} (or \emph{parallel classes}, respectively).
Dually, if for every $2$-separation $(X,Y)$ of $M$ either $X$ or $Y$ is a series pair (or series class), then $M$ is \emph{$3$-connected up to series pairs} (or \emph{series classes}, respectively).

We say $Z \subseteq E(M)$ is in the \emph{guts} of a $k$-separation $(X,Y)$ if $Z \subseteq \cl(X-Z) \cap \cl(Y-Z)$, and we say $Z$ is in the \emph{coguts} of $(X,Y)$ if $Z$ is in the guts of $(X,Y)$ in $M^*$.
We also say $z$ is in the guts (or the coguts) of a $k$-separation $(X,Y)$ if $\{z\}$ is in the guts (or the coguts, respectively) of $(X,Y)$.
Note that if $z$ is in the guts of $(X,Y)$, then $z \notin \cocl(X)$ and $z \notin \cocl(Y)$.

We say that a partition $(X_1,X_2,\dotsc,X_m)$ of $E(M)$ is a \emph{path of $3$-separations} if $(X_1 \cup \dotsm \cup X_i, X_{i+1} \cup \dotsm \cup X_m)$ is exactly $3$-separating for each $i \in \seq{m-1}$.
Note that $|X_1|,|X_m| \ge 2$ (and $|X_i| \ge 1$ for all $i \in [m]$). 
If $X_i$ is in the guts (or coguts) of $(X_1 \cup \dotsm \cup X_i, X_{i+1} \cup \dotsm \cup X_m)$, then we say $X_i$ is a \emph{guts set} (or \emph{coguts set}, respectively) and, for each $x \in X_i$, we say $x$ is a \emph{guts element} (or \emph{coguts element}, respectively).

A $3$-separation $(X,Y)$ of $M$ is a \textit{vertical $3$-separation} if $\min\{r(X),r(Y)\}\geq 3$. We also say that a partition $(X,\{z\},Y)$ is a \textit{vertical $3$-separation} of $M$ when both $(X\cup z,Y)$ and $(X,Y\cup z)$ are vertical $3$-separations with $z$ in the guts. We will write $(X,z,Y)$ for $(X,\{z\},Y)$.
If $(X,z,Y)$ is a vertical $3$-separation of $M$, then we say that $(X,z,Y)$ is a \emph{cyclic $3$-separation} of $M^*$.

Suppose $e \in E(M)$, and $(X,Y)$ is a partition of $M \ba e$ with $\lambda_{M \ba e}(X)=k$.
We say that $e$ \emph{blocks} $X$ if $\lambda_{M}(X) > k$.
In particular, we say $e$ blocks a series pair (or triad) $X$ of $M \ba e$ if $X$ is not a series pair (or triad, respectively) of $M$.
In any case, if $e$ blocks $X$, then $e \notin \cl_M(Y)$, so $e \in \cocl_M(X)$.
We say that $e$ \emph{fully blocks} $(X,Y)$ if both $\lambda_M(X) > k$ and $\lambda_M(X \cup e) > k$.  It is easy to see that $e$ fully blocks $(X,Y)$ if and only if $e \notin \cl_M(X) \cup \cl_M(Y)$.

A set $X \subseteq E(M)$ is \emph{fully closed} if $X$ is closed in both $M$ and $M^*$.
The \emph{full closure of $X$}, denoted $\fcl_M(X)$, is the intersection of all fully closed sets containing $X$.
The full closure can be obtained by repeatedly taking closures and coclosures until no new elements are added.

The following lemma is routine, but helpful to show a matroid is $3$-connected up to series and parallel classes.

\begin{lemma}
  \label{fclnontrivialsep}
  Let $M$ be a simple and cosimple matroid, and let $(X,Y)$ be a $2$-separation of $M$.
  Then $\fcl(X) \neq E(M)$ and $\fcl(Y) \neq E(M)$.
  Moreover, $(\fcl(X),Y-\fcl(X))$ is also a $2$-separation of $M$.
\end{lemma}

We use the following well-known results about the existence of elements that preserve connectivity.  The first we refer to as Bixby's Lemma.

\begin{lemma}[Bixby's Lemma~\cite{Bixby82}]
\label{bixby}
Let $M$ be a $3$-connected matroid, and let $e\in E(M)$. Then $M/e$ is $3$-connected up to parallel pairs or $M\del e$ is $3$-connected up to series pairs.
\end{lemma}

\begin{lemma}[see {\cite[Lemma~8.8.2]{Oxley11}}, for example]
  \label{lineconn}
  Let $M$ be a $3$-connected matroid and let $L$ be a segment of $M$ with $|L| \ge 4$. If $\ell \in L$, then $M \ba \ell$ is $3$-connected.
\end{lemma}

\begin{lemma}[{\cite[Lemma~3.5]{Whittle99}}]
  \label{vert3sep}
  Let $M$ be a $3$-connected matroid and let $z \in E(M)$.
  The following are equivalent:
  \begin{enumerate}
    \item $M$ has a vertical $3$-separation $(X, z, Y)$.
    \item $\si(M/z)$ is not 3-connected.
  \end{enumerate}
\end{lemma}

\begin{lemma}[see {\cite[Lemma 2.12]{BS14}}, for example]
\label{fanends}
Let $M$ be a $3$-connected matroid with $|E(M)| \ge 7$.  
Suppose that $M$ has a fan $F$ of size at least~$4$, and let $f$ be an end of $F$.
\begin{enumerate}
 \item If $f$ is a spoke element, then $\co(M\del f)$ is $3$-connected and $\si(M/f)$ is not $3$-connected.
 \item If $f$ is a rim element, then $\si(M/f)$ is $3$-connected and $\co(M\del f)$ is not $3$-connected.
\end{enumerate}
\end{lemma}

\begin{lemma}[{\cite[Lemma~1.5]{OW00}}]
\label{fanendsstrong}
  Let $M$ be $3$-connected matroid that is not a wheel or a whirl.
  Suppose $M$ has a maximal fan~$F$ of size at least~$4$, and let $f$ be an end of $F$.
  \begin{enumerate}
    \item If $f$ is a spoke element, then $M\ba f$ is $3$-connected.
    \item If $f$ is a rim element, then $M/f$ is $3$-connected.
  \end{enumerate}
\end{lemma}

\begin{lemma}
  \label{fanmiddle}
  Let $M$ be a $3$-connected matroid, 
  and let $F$ be a $5$-element fan of $M$ with ordering $(f_1,f_2,\dotsc,f_5)$, where $\{f_2,f_3,f_4\}$ is a triangle.
  Either $\si(M/f_3)$ is $3$-connected, or there exists some $f_6 \in E(M)-F$ such that $M^* | \{f_1,f_2,f_3,f_4,f_5,f_6\} \cong M(K_4)$.
\end{lemma}
\begin{proof}
  We may assume that $M \not\cong M(\mathcal{W}_3)$, for otherwise $\si(M/f_3)$ is $3$-connected.
  Thus, by \cref{fanunique}, $\si(M/f_3) \cong M/f_3\ba f_2$.
  Suppose that $\si(M/f_3)$ is not $3$-connected, so $M/f_3\ba f_2$ has a $2$-separation $(U,V)$.
  To begin with, assume that $\si(M/f_3)$ is cosimple.
  Note that $\{f_1,f_2,f_4,f_5\}$ is a cocircuit of $M$, so $\{f_1,f_4,f_5\}$ is a triad of $M/f_3\ba f_2$.
  Without loss of generality, $|U \cap \{f_1,f_4,f_5\}| \ge 2$, and $U$ is fully closed by \cref{fclnontrivialsep}.
  So $\{f_1,f_4,f_5\} \subseteq U$.
  Now $f_2 \in \cl_{M/f_3}(U)$, so $(U \cup f_2,V)$ is a $2$-separation in $M/f_3$.
  Moreover, $f_3 \in \cocl_{M}(U \cup f_2)$, so $(U \cup \{f_2,f_3\},V)$ is a $2$-separation in $M$, contradicting that $M$ is $3$-connected.

  Now we may assume that $M /f_3 \ba f_2$ is not cosimple, so $f_2$ is in a triad~$T^*$ of $M$ that avoids $f_3$.
  By orthogonality with the triangle $\{f_2,f_3,f_4\}$, we have $f_4 \in T^*$.
  Since $M$ is $3$-connected, it follows that $T^* = \{f_2,f_4,f_6\}$ for some $f_6 \in E(M)-F$.
  Now, by submodularity,
  \begin{align*}
      r^*(\{f_1,f_5,f_6\}) &\le r^*(\{f_1,f_2,f_3,f_4,f_5,f_6\}) + r^*(E(M)-\{f_2,f_3,f_4\}) - r(M^*) \\
  &= 3 + (r(M^*)-1) - r(M^*) = 2,
  \end{align*}
  so $\{f_1,f_5,f_6\}$ is also a triad of $M$.
  It now follows that $M^* | \{f_1,f_2,f_3,f_4,f_5,f_6\} \cong M(K_4)$, as required.
\end{proof}

We also require the following lemma.
\begin{lemma}[{\cite[Lemma~2.11]{BS14}}]
  \label{gutsandcoguts}
  Let $(X, Y)$ be a $3$-separation of a $3$-connected matroid $M$. If $X \cap \cl(Y) \neq \emptyset$ and $X \cap \cocl(Y) \neq \emptyset$, then $|X \cap \cl(Y)| = 1$ and $|X \cap \cocl(Y)| = 1$.
\end{lemma}

\subsection*{Local connectivity}

Let $M$ be a matroid with $X,Y \subseteq E(M)$.
The \emph{local connectivity} between $X$ and $Y$, denoted $\lc_M(X,Y)$, is defined to be $\lc_M(X,Y)= r(X)+r(Y)-r(X \cup Y)$. Evidently, $\lc_M(Y,X)= \lc_M(X,Y)$. Note that if $\{X,Y\}$ is a partition of $E(M)$, then $\lc_M(X,Y)= \lambda_M(X)$.
We write $\lc$ instead of $\lc_M$ when $M$ is clear from context, and we write $\lc^*(X,Y)$ for $\lc_{M^*}(X,Y)$.
We now recall some elementary properties. 

\begin{lemma}[see {\cite[Lemma~8.2.3]{Oxley11}}, for example]
  \label{growpi}
  Let $X_1$, $X_2$, $Y_1$, and $Y_2$ be subsets of the ground set of a matroid. If $X_1 \subseteq Y_1$ and $X_2 \subseteq Y_2$, then $\lc(X_1,X_2) \le \lc(Y_1,Y_2)$.
\end{lemma}

\begin{lemma}[{\cite[Lemma~2.4(iv)]{OSW04}}]
  \label{pihelper}
  If $\{X,Y,Z\}$ is a partition of the ground set of a matroid, then $\lambda(X) + \lc(Y,Z)= \lambda(Z)+ \lc(X,Y)$.
\end{lemma}

The next lemma is elementary.
\begin{lemma}
  \label{picircuits}
  For a matroid $M$, let $L$ and $R$ be disjoint subsets of $E(M)$.
  If $\lc(L,R) = 0$, and $C$ is a circuit contained in $L \cup R$, then either $C \subseteq L$ or $C \subseteq R$.
\end{lemma}

\begin{lemma}
  \label{pflancoguts}
  Let $M$ be a $3$-connected matroid with a path of $3$-separations $(X,Z,Y)$ such that $Z$ is a coguts set.
  Then $\lc(X,Y) \le 1$.
  Moreover, $\lc(X,Y)=1$ if and only if $|Z| = 1$.
\end{lemma}
\begin{proof}
  Since each $z \in Z$ is a coguts element, $r(X \cup Z) = r(X) + |Z|$ and $|Z|=r(Z)$.
  So $\lc(X,Z) = 0$. Now, by \cref{pihelper},
  \begin{align*}
    \lambda(Z) &= \lambda(Y) + \lc(X,Z) - \lc(X,Y) \\
    &= 2 - \lc(X,Y).
  \end{align*}
  If $\lc(X,Y) = 2$, then $\lambda(Z)=0$, a contradiction.
  So $\lc(X,Y) \le 1$.
  Now if $\lc(X,Y) = 1$, then, as $M$ is $3$-connected, $|Z|=1$.
  On the other hand, if $|Z| = 1$, then $\lambda(Z) = 1$, so $\lc(X,Y) = 1$, as required.
\end{proof}

\begin{lemma}
  \label{pflantriad}
  Let $M$ be a $3$-connected matroid with a path of $3$-separations $(X,\{z_1\},\{z_2\},\{z_3\},Y)$ such that $z_1$ and $z_3$ are coguts elements, and $z_2$ is a guts element.
  Then $\lc(X,Y) \le 1$.
  Moreover, $\lc(X,Y)=1$ if and only if $\{z_1,z_2,z_3\}$ is a triad.
\end{lemma}
\begin{proof}
  Let $Z = \{z_1,z_2,z_3\}$.
  If $r(Z)=2$, then $z_1 \in \cl(Y \cup \{z_2,z_3\})$, so $z_1 \notin \cocl(X)$, contradicting that $z_1$ is a coguts element.
  So $r(Z)=3$.
  Moreover, since $z_1$ and $z_3$ are coguts elements whereas $z_2$ is a guts element, $r(Y \cup Z) = r(Y)+2$.
  So $\lc(Y,Z) = 1$. 
  Now, by \cref{pihelper},
  \begin{align*}
    \lambda(Z) &= \lambda(X) + \lc(Y,Z) - \lc(X,Y) \\
    &= 3-\lc(X,Y).
  \end{align*}
  Since $|Z|=3$, we have $\lambda(Z) \ge 2$, implying $\lc(X,Y) \le 1$.
  Now if $Z$ is a triad, then $\lambda(Z) = 2$ so $\lc(X,Y) = 1$.
  On the other hand, if $\lc(X,Y) = 1$, then $\lambda(Z)=2$ in which case,
  since $r(Z)=3$, we deduce that $Z$ is a triad.
\end{proof}

\subsection*{Minors and fragility}

Let $M$ be a matroid, let $\mathcal{N}$ be a set of matroids, and let $x$ be an element of $M$.
For a matroid $N$, we say that \emph{$M$ has an $N$-minor} if $M$ has a minor isomorphic to $N$.
We say $M$ has an $\mathcal{N}$-minor if $M$ has an $N$-minor for some $N \in \mathcal{N}$.
If $M\ba x$ has an $\mathcal{N}$-minor, then $x$ is $\mathcal{N}$-\textit{deletable}.
If $M/x$ has an $\mathcal{N}$-minor, then $x$ is $\mathcal{N}$-\textit{contractible}.
If neither $M\ba x$ nor $M/x$ has an $\mathcal{N}$-minor, then $x$ is $\mathcal{N}$-\textit{essential}.
If $x$ is both $\mathcal{N}$-deletable and $\mathcal{N}$-contractible, then we say that $x$ is \textit{$\mathcal{N}$-flexible}.
A matroid $M$ is \textit{$\mathcal{N}$-fragile} if $M$ has an $\mathcal{N}$-minor, and no element of $M$ is $\mathcal{N}$-flexible
(note that sometimes this is referred to in the literature as ``strictly $\mathcal{N}$-fragile'').
For $X \subseteq E(M)$, we also say that $X$ is \emph{$\mathcal{N}$-deletable} (or \emph{$\mathcal{N}$-contractible}) when $M \del X$ (or $M / X$, respectively) has an $\mathcal{N}$-minor.
When $\mathcal{N} = \{N\}$, we use the prefix ``$N$-'' for these terms, rather than ``$\{N\}$-''.

The next lemma is well known, and the subsequent lemma is a straightforward corollary.

\begin{lemma}[see {\cite[Corollary~8.2.5]{Oxley11}}, for example]
  \label{minor3conn}
  Let $M$ be a matroid with a $2$-separation $(X,Y)$, and let $N$ be a $3$-connected minor of $M$.
  Then $|Z \cap E(N)| \le 1$ for some $Z \in \{X,Y\}$.
\end{lemma}

\begin{lemma}
  \label{niceVertSep}
  Let $(X, z, Y)$ be a vertical $3$-separation of a $3$-connected matroid $M$, and let $N$ be a $3$-connected minor of $M/z$. Then there exists a vertical $3$-separation $(X', z, Y')$ of $M$ such that $|X' \cap E(N)| \le 1$ and $Y' \cup z$ is closed in $M$.
\end{lemma}

The following is proved in \cite{BS14,Clark15}.

\begin{lemma}
  \label{CPL}
  Let $N$ be a $3$-connected minor of a $3$-connected matroid $M$. Let $(X, \{z\}, Y)$ be a vertical $3$-separation of $M$ such that $M / z$ has an $N$-minor with $|X \cap E(N)| \le 1$. Let $X' = X-\cl(Y)$
  and $Y' = \cl(Y) - z$.
  Then
  \begin{enumerate}
    \item each element of $X'$ is $N$-contractible; and
    \item at most one element of $\cl(X)-z$ is not $N$-deletable, and if such an element~$x$ exists, then $x \in X' \cap \cocl(Y')$ and $z \in \cl(X' - x)$.
  \end{enumerate}
\end{lemma}

We also use the following well-known property of fragile matroids.
\begin{lemma}[see {\cite[Proposition~4.4]{MvZW10}}, for example]
  \label{genfragileconn}
  Let $\mathcal{N}$ be a non-empty set of $3$-connected matroids with $|E(N)| \ge 4$ for each $N \in \mathcal{N}$.
  If $M$ is $\mathcal{N}$-fragile, then $M$ is $3$-connected up to series and parallel classes.
\end{lemma}

We require two more lemmas, about fans in fragile matroids.
Recall that a maximal $4$-element fan has one rim element and one spoke element at the two ends: the internal elements are not considered to be rim or spoke elements.

\begin{lemma}
  \label{fragilefanelements}
  Let $\mathcal{N}$ be a non-empty set of $3$-connected matroids, 
  let $M$ be a $\mathcal{N}$-fragile matroid, and let $F$ be a fan of $M$ of size at least~$4$.
  If $s$ is a spoke element of $F$, then $s$ is not $\mathcal{N}$-contractible,
  whereas if $t$ is a rim element of $F$, then $t$ is not $\mathcal{N}$-deletable.
\end{lemma}
\begin{proof}
  Let $(f_1,f_2,f_3,s)$ be a (not necessarily maximal) fan where $\{f_1,f_2,f_3\}$ is a triad and $\{f_2,f_3,s\}$ is a triangle, so $s$ is a spoke element, and suppose that $s$ is $\mathcal{N}$-contractible.
  Since each $N \in \mathcal{N}$ is $3$-connected, $\si(M/s)$ has an $\mathcal{N}$-minor.
  So $f_2$ and $f_3$ are $\mathcal{N}$-deletable.
  Similarly, $\co(M \ba f_2)$ has an $\mathcal{N}$-minor, so $f_3$ is $\mathcal{N}$-contractible, due to the triad $\{f_1,f_2,f_3\}$ of $M$.
  But now $f_3$ is $\mathcal{N}$-flexible, a contradiction.
  A similar argument applies if $(f_1,f_2,s,f_4,f_5)$ is a fan where $\{f_1,f_2,s\}$ and $\{s,f_4,f_5\}$ are triangles and $\{f_2,s,f_4\}$ is a triad.
  The result then follows by duality.
\end{proof}

\begin{lemma}
  \label{fragilefans}
  Let $\mathcal{N}$ be a non-empty set of $3$-connected matroids, each of which has no $4$-element fans. Let $M$ be a $\mathcal{N}$-fragile matroid, and let $F$ be a fan of $M$.
  \begin{enumerate}
    \item If $|F| \ge 5$ and $e$ is an end of $F$, then $e$ is not $\mathcal{N}$-essential.
    \item If $|F| \ge 6$ and $e \in F$, then $e$ is not $\mathcal{N}$-essential.
  \end{enumerate}
\end{lemma}
\begin{proof}
  Suppose $|F| = 5$ and let $(f_1,f_2,f_3,f_4,f_5)$ be a fan ordering of $F$.
  By duality, we may assume $f_1$ is a spoke element, so $\{f_1,f_2,f_3\}$ is a triangle.
  By \cref{fragilefanelements}, $f_1$ is not $\mathcal{N}$-contractible.
  Suppose $f_1$ is not $\mathcal{N}$-deletable.
  Since each matroid in $\mathcal{N}$ is $3$-connected, a matroid $M'$ has an $\mathcal{N}$-minor if and only if $\si(M')$ has an $\mathcal{N}$-minor (and the same holds when ``$\si(M')$'' is replaced with ``$\co(M')$'').
  If $f_2$ is $\mathcal{N}$-contractible, then, as $\{f_1,f_3\}$ is a parallel pair in $M / f_2$, the element $f_1$ is $\mathcal{N}$-deletable.
  So $f_2$ is not $\mathcal{N}$-contractible due to the triangle $\{f_1,f_2,f_3\}$.
  Similarly, due to the triad $\{f_2,f_3,f_4\}$, the element $f_3$ is not $\mathcal{N}$-deletable.
  Finally, due to the triangle $\{f_3,f_4,f_5\}$, the element $f_4$ is not $\mathcal{N}$-contractible.
  By \cref{fragilefanelements}, the elements $\{f_1,f_2,f_3,f_4\}$ are $\mathcal{N}$-essential.
  We have that $M /C \ba D \cong N$ for some $N \in \mathcal{N}$ and disjoint $C,D \subseteq E(M)$. Let $N' = M/C\ba D$.
  Now, $r_{N'}(\{f_1,f_2,f_3\}) \le 2$ and $r^*_{N'}(\{f_2,f_3,f_4\}) \le 2$, a contradiction. 

  Now suppose $|F|=6$ and let $(f_1,f_2,\dotsc,f_6)$ be a fan ordering of $F$.
  By the foregoing, $f_1$, $f_2$, $f_5$, and $f_6$ are not 
  $\mathcal{N}$-essential.
  Up to duality, we may assume that $f_1$ is a spoke element.
  Then $\{f_1,f_2,f_3\}$ is a triangle and $f_2$ is a rim element, so $f_2$ is 
  $\mathcal{N}$-contractible
  by \cref{fragilefanelements}.
  Since $\{f_1,f_3\}$ is a parallel pair in $M / f_2$, it follows that $f_3$ is 
  $\mathcal{N}$-deletable.
  By a symmetric argument, $f_4$ is 
  $\mathcal{N}$-contractible.
  The result follows.
\end{proof}

\subsection*{Path width three}

A matroid $M$ has \emph{path width at most $k$} if there exists an ordering $(e_1,e_2,\dotsc,e_{n})$ of $E(M)$ such that $\{e_1,\dotsc,e_t\}$ is $k$-separating for all $t \in \seq{n-1}$.
For a $3$-connected matroid $M$ with $|E(M)| \ge 4$ and path width at most three, $M$ does not have path width at most two, so we simply say that $M$ has \emph{path width three}.
When $M$ has path width three with respect to the ordering $(e_1,e_2,\dotsc,e_{n})$, then we say $(e_1,e_2,\dotsc,e_{n})$ is a \emph{sequential ordering} of $M$.
The next lemma is well known, and it implies that, relative to such a sequential ordering, each $e_i \in \{e_3,e_4,\dotsc,e_{n-2}\}$ is unambiguously a guts or a coguts element.

\begin{lemma}
  \label{gutsstayguts}
  Let $M$ be a $3$-connected matroid, and let $(X, e, Y)$ be a partition of $E$ such that $X$ is exactly $3$-separating. Then
  \begin{enumerate}
    \item $X \cup e$ is $3$-separating if and only if $e \in \cl(X)$ or $e \in \cocl(X)$, and
    \item $X \cup e$ is exactly $3$-separating if and only if either $e \in \cl(X) \cap \cl(Y)$ or $e \in \cocl(X) \cap \cocl(Y)$.
  \end{enumerate}
\end{lemma}

We say that a set $X$ in a matroid~$M$ is \emph{path generating} if $X$ is $3$-separating and $\fcl_M(X)= E(M)$.
In particular, if $M$ has path width three and $(e_1,e_2,\dotsc,e_{n})$ is a sequential ordering of $M$, then $\{e_1,e_2\}$ and $\{e_{n-1},e_n\}$ are path-generating sets.

Let $M$ be a $3$-connected matroid of path width three that has rank and corank at least~$3$ and is not a wheel or a whirl.
Let $\sigma=(e_1, e_2, \dotsc, e_n)$ be a sequential ordering of $M$.
Then $\{e_1, e_2, e_3\}$ is a triangle or a triad.
If this set is not in a larger segment, cosegment, or fan of $M$, then let $L(\sigma) = \{e_1, e_2, e_3\}$ and call $L(\sigma)$ a \emph{triangle end} or a \emph{triad end} of $M$, respectively.
If $\{e_1, e_2, e_3\}$ is contained in a $4$-segment or $4$-cosegment, then let $L(\sigma)$ be the maximal segment or cosegment (respectively) containing $\{e_1,e_2,e_3\}$, and call $L(\sigma)$ a \emph{segment end} or a \emph{cosegment end} of $M$, respectively.
Finally, if $\{e_1, e_2, e_3\}$ is contained in a fan of size at least~$4$, then take a maximal such fan $F$, let $L(\sigma)$ be the set of internal elements of the fan $F$, and call $L(\sigma)$ a \emph{fan end} of $\sigma$.
We define $R(\sigma)$ analogously.

Loosely speaking, the next lemma shows that, up to reversal, any sequential ordering of a matroid of path width three has the same pair of ends.

\begin{lemma}[{\cite[Theorem 1.3]{HOS07}}]
  \label{welldefinedends}
  Let $M$ be a $3$-connected matroid of path width three that has rank and corank at least~$3$ and is not a wheel or a whirl.
  Then there are distinct subsets $L(M)$ and $R(M)$ of $E(M)$ such that $\{L(M), R(M)\} = \{L(\sigma), R(\sigma)\}$ for every sequential ordering $\sigma$ of $E(M)$.
\end{lemma}

\begin{lemma}[{\cite[Theorem 1.4]{HOS07}}]
  \label{endslipperiness}
  Let $M$ be a $3$-connected matroid of path width three that has rank and corank at least~$3$ and is not a wheel or a whirl.
  Let $\sigma$ and $\sigma'$ be sequential orderings of $M$ such that $L(\sigma)=L(\sigma')$ and $R(\sigma)=R(\sigma')$.
  Then
  \begin{enumerate}
    \item if $L(\sigma)$ is a triangle or a triad end of $M$, then the first three elements of $\sigma'$ are in $L(\sigma)$;\label{esi}
    \item if $L(\sigma)$ is a segment or cosegment end of $M$, then the first $|L(\sigma)| - 1$ elements of $\sigma'$ are in $L(\sigma)$; and\label{esii}
    \item if $L(\sigma)$ is a fan end of $M$, then either the first $|L(\sigma)|$ elements of $\sigma'$ are in $L(\sigma)$, or there is a maximal fan~$F$ of $M$ having $L(\sigma)$ as its set of internal elements such that the first $|L(\sigma)| +1$ elements of $\sigma'$ include $L(\sigma)$ and are contained in $F$.\label{esiii}
  \end{enumerate}
\end{lemma}

Let $\mathbf{P} = (P_1, P_2, \dotsc, P_n)$ be an ordered partition of a set $S$.
Then the ordered partition $\mathbf{Q} = (Q_1, Q_2, \dotsc, Q_m)$ is a \emph{concatenation} of $\mathbf{P}$ if there are indices $0 = k_0 < k_1 < \dotsm < k_m = n$ such that $Q_i = P_{k_{i-1}+1} \cup \dotsm \cup P_{k_i}$ for $i \in \{1, \dotsc,m\}$.
If $\mathbf{Q}$ is a concatenation of $\mathbf{P}$, then $\mathbf{P}$ is a \emph{refinement} of $\mathbf{Q}$.

Let $\mathbf{P} = (P_1,P_2,\dotsc,P_m)$ be an ordered partition of the ground set of a matroid~$M$ with path width three.
We say that $\mathbf{P}$ is a \emph{guts-coguts path} if $\mathbf{P}$ is a path of $3$-separations such that, for each $i \in \{2,3,\dotsc,m-1\}$, the set $P_i$ is in the guts or coguts of the $3$-separation $(P_1 \cup \dotsm \cup P_{i},P_{i+1} \cup \dotsm \cup P_m)$, and, for each $i \in \{2,3,\dotsc,m-2\}$, if $P_i$ is in the guts (respectively, the coguts), then $P_{i+1}$ is in the coguts (respectively, the guts).

  Let $\sigma=(e_1,e_2,\dotsc,e_n)$ be a sequential ordering of a $3$-connected matroid~$M$ with path width three.
  We treat $\sigma$ as a partition into singletons, in which case any concatenation of $\sigma$ is a path of $3$-separations.
  For $X \subseteq E(M)$, we say that $X$ is an \emph{initial segment} of $\sigma$ if $X = \{e_i : i \in [j]\}$ for some $j \in [n]$, and 
  $X$ is a \emph{terminal segment} of $\sigma$ if $X = \{e_i : j \le i \le n\}$ for some $j \in [n]$.
  For an initial segment $P$ and a terminal segment $P'$ of $\sigma$, where $P$ and $P'$ are disjoint and each have size at least~$2$, there is a unique concatenation $(P_1,P_2,\dotsc,P_m)$ of $\sigma$ that is a guts-coguts path with $P=P_1$ and $P' = P_m$ (where uniqueness follows from \cref{gutsstayguts}).
  We call $(P_1,P_2,\dotsc,P_m)$ the \emph{guts-coguts concatenation of $\sigma$ with ends $P$ and $P'$}.
  We also call $P$ the \emph{left end}, and $P'$ the \emph{right end}.

\subsection*{Representation theory}

A \textit{partial field} is a pair $(R, G)$, where $R$ is a commutative ring with unity, and $G$ is a subgroup of the group of units of $R$ such that $-1 \in G$. If $\mathbb{P}=(R,G)$ is a partial field, then we write $p\in \mathbb{P}$ whenever $p\in G\cup \{0\}$.

Let $\mathbb{P}=(R,G)$ be a partial field, and let $A$ be an $X\times Y$ matrix with entries from $\mathbb{P}$. Then $A$ is a $\mathbb{P}$-\textit{matrix} if every non-zero subdeterminant of $A$ is in $G$. If $X'\subseteq X$ and $Y'\subseteq Y$, then we write $A[X',Y']$ to denote the submatrix of $A$ induced by $X'$ and $Y'$.
When $X$ and $Y$ are disjoint, and $Z\subseteq X\cup Y$, we denote by $A[Z]$ the submatrix induced by $X\cap Z$ and $Y\cap Z$, and we denote by $A-Z$ the submatrix induced by $X-Z$ and $Y-Z$.
 
\begin{theorem}[{\cite[Theorem 2.8]{PvZ10b}}]
\label{pmatroid}
Let $\mathbb{P}$ be a partial field, and let $A$ be an $X\times Y$ $\mathbb{P}$-matrix, where $X$ and $Y$ are disjoint. Let
\begin{equation*}
\mathcal{B}=\{X\}\cup \{X\triangle Z : |X\cap Z|=|Y\cap Z|, \det(A[Z])\neq 0\}. 
\end{equation*}
 Then $\mathcal{B}$ is the set of bases of a matroid on $X\cup Y$.
\end{theorem}

We say that the matroid in \cref{pmatroid} is $\mathbb{P}$-\textit{representable}, and that $A$ is a $\mathbb{P}$-\textit{representation} of $M$. We write $M=M[I|A]$ if $A$ is a $\mathbb{P}$-matrix, and $M$ is the matroid whose bases are described in \cref{pmatroid}. 

Let $A$ be an $X\times Y$ $\mathbb{P}$-matrix, with $X \cap Y = \emptyset$, and let $x\in X$ and $y\in Y$ such that $A_{xy}\neq 0$. Then we define $A^{xy}$ to be the $(X\triangle \{x,y\})\times (Y\triangle \{x,y\})$ $\mathbb{P}$-matrix given by 
\begin{displaymath}
  (A^{xy})_{uv} =
\begin{cases}
    A_{xy}^{-1} \quad & \textrm{if } uv = yx\\
    A_{xy}^{-1} A_{xv} & \textrm{if } u = y, v\neq x\\
    -A_{xy}^{-1} A_{uy} & \textrm{if } v = x, u \neq y\\
    A_{uv} - A_{xy}^{-1} A_{uy} A_{xv} & \textrm{otherwise.}
\end{cases}
\end{displaymath}
We say that $A^{xy}$ is obtained from $A$ by \textit{pivoting} on $xy$.

Two $\mathbb{P}$-matrices are \textit{scaling equivalent} if one can be obtained from the other by repeatedly scaling rows and columns by non-zero elements of $\mathbb{P}$. Two $\mathbb{P}$-matrices are \textit{geometrically equivalent} if one can be obtained from the other by a sequence of the following operations: scaling rows and columns by non-zero entries of $\mathbb{P}$, permuting rows, permuting columns, and pivoting.
 
Let $\mathbb{P}$ be a partial field, and let $M$ and $N$ be matroids such that $N$ is a minor of $M$. Suppose that the ground set of $N$ is $X'\cup Y'$, where $X'$ is a basis of $N$. We say that $M$ is $\mathbb{P}$-\textit{stabilized by $N$} if, whenever $A_1$ and $A_2$ are $X\times Y$ $\mathbb{P}$-matrices, with $X'\subseteq X$ and $Y'\subseteq Y$, such that
\begin{enumerate}
 \item[(i)] $M=M[I|A_1]=M[I|A_2]$,
 \item[(ii)] $A_1[X',Y']$ is scaling equivalent to $A_2[X',Y']$, and
 \item[(iii)] $N=M[I|A_1[X',Y']]=M[I|A_2[X',Y']],$ 
\end{enumerate}
then $A_1$ is scaling equivalent to $A_2$.
If $M$ is $\mathbb{P}$-stabilized by $N$, and every $\mathbb{P}$-representation of $N$ extends to a $\mathbb{P}$-representation of $M$, then we say $M$ is \emph{strongly $\mathbb{P}$-stabilized} by $N$.

Let $\mathcal{M}$ be a class of matroids. We say that $N$ is a \textit{$\mathbb{P}$-stabilizer for $\mathcal{M}$} if, for every $3$-connected $\mathbb{P}$-representable matroid $M\in \mathcal{M}$ with an $N$-minor, $M$ is $\mathbb{P}$-stabilized by $N$.
We say that $N$ is a \textit{strong $\mathbb{P}$-stabilizer for $\mathcal{M}$} if, for every $3$-connected $\mathbb{P}$-representable matroid $M\in \mathcal{M}$ with an $N$-minor, $M$ is strongly $\mathbb{P}$-stabilized by $N$.
Here we will be primarily interested in the case where $\mathcal{M}$ is the class of $\mathbb{P}$-representable matroids for some partial field~$\mathbb{P}$, in which case, when there is no ambiguity, we simply say ``$N$ is a strong $\mathbb{P}$-stabilizer''. 

\subsection*{\texorpdfstring{$2$}{2}-regular, \texorpdfstring{$3$}{3}-regular, and \texorpdfstring{$\mathbb{H}_5$}{H5}-representable matroids}

The \emph{$2$-regular} partial field is
$$\mathbb{U}_2 = (\mathbb{Q}(\alpha_1, \alpha_2),\left<-1,\alpha_1, \alpha_2, 1-\alpha_1, 1-\alpha_2,\alpha_1-\alpha_2\right>),$$
where $\alpha_1$ and $\alpha_2$ are indeterminates.
Recall that we say a matroid is \emph{$2$-regular} if it is $\mathbb{U}_2$-representable.
Note that $\mathbb{U}_2$ is the universal partial field of $U_{2,5}$ \cite[Theorem 3.3.24]{vanZwam09}; intuitively, this means that $\mathbb{U}_2$ is the most general partial field that $U_{2,5}$ is representable over (for a formal definition, refer to \cite{PvZ10b}).
If a matroid is $2$-regular, then it is $\mathbb{F}$-representable for every field $\mathbb{F}$ of size at least~$4$~\cite[Corollary 3.1.3]{Semple98}.

The \emph{$3$-regular} partial field is:
\begin{multline*}
\mathbb{U}_3 = (\mathbb{Q}(\alpha_1,\alpha_2,\alpha_3), \\
\left<-1,\alpha_1,\alpha_2,\alpha_3,\alpha_1-1,\alpha_2-1,\alpha_3-1,\alpha_1-\alpha_2,\alpha_1-\alpha_3,\alpha_2-\alpha_3\right>),
\end{multline*}
where $\alpha_1,\alpha_2,\alpha_3$ are indeterminates, and recall that we say a matroid is \emph{$3$-regular} if it is $\mathbb{U}_3$-representable.

The \emph{Hydra-5}- partial field is
\begin{multline*}
\mathbb{H}_5 = (\mathbb{Q}(\alpha, \beta, \gamma), \\
\left<-1,\alpha, \beta, \gamma, 1-\alpha, 1-\beta, 1-\gamma, \alpha-\gamma, \gamma-\alpha\beta, 1-\gamma - (1-\alpha)\beta \right>),
\end{multline*}
where $\alpha$, $\beta$, and $\gamma$ are indeterminates.
A $3$-connected matroid with a $\utfutf$-minor is $\mathbb{H}_5$-representable if and only if it has six inequivalent representations over $\GF(5)$ \cite[Lemma~5.17]{PvZ10b}.

We next prove that the partial fields $\mathbb{U}_3$ and $\mathbb{H}_5$ are isomorphic; in particular, a matroid is $\mathbb{H}_5$-representable if and only if it is $3$-regular.
For partial fields $\mathbb{P}_1$ and $\mathbb{P}_2$, a function
$\phi : \mathbb{P}_1 \rightarrow \mathbb{P}_2$ is a \emph{homomorphism} if
\begin{enumerate}
  \item $\phi(1) = 1$,
  \item $\phi(pq) = \phi(p)\phi(q)$ for all $p, q \in \mathbb{P}_1$, and
  \item $\phi(p) +\phi(q) = \phi(p +q)$ for all $p, q \in \mathbb{P}_1$ such that $p +q \in \mathbb{P}_1$.
\end{enumerate}
The existence of a homomorphism from $\mathbb{P}_1$ to $\mathbb{P}_2$ certifies that each $\mathbb{P}_1$-representable matroid is also $\mathbb{P}_2$-representable \cite[Corollary 2.9]{PvZ10b}.

\begin{lemma}
  \label{3reglemma}
  The partial fields $\mathbb{H}_5$ and $\mathbb{U}_3$ are isomorphic.
  In particular, a matroid is $3$-regular if and only if it is $\mathbb{H}_5$-representable.
\end{lemma}
\begin{proof}
  It is easy, but tedious, to check that $\phi : \mathbb{H}_5 \rightarrow \mathbb{U}_3$ determined by
  $$\phi(\alpha) = \alpha_1,$$
  $$\phi(\beta) = \frac{\alpha_3-\alpha_2}{\alpha_1-\alpha_2},$$
  $$\phi(\gamma) = \alpha_3$$
  is well defined, and is a homomorphism.
  In particular, observe that $\phi(\gamma-\alpha\beta)=\frac{\alpha_2(\alpha_1-\alpha_3)}{\alpha_1-\alpha_2}$, and $\phi((1-\gamma)-(1-\alpha)\beta)=\frac{(\alpha_3-\alpha_1)(\alpha_2-1)}{\alpha_1-\alpha_2}$.
  Moreover, it is also easily checked that $\phi' : \mathbb{U}_3 \rightarrow \mathbb{H}_5$ determined by
  $$\phi'(\alpha_1) = \alpha,$$
  $$\phi'(\alpha_2) = \frac{\gamma-\alpha\beta}{1-\beta},$$
  $$\phi'(\alpha_3) = \gamma$$
  is well defined, and a homomorphism.
  Clearly $\phi'(\phi(\alpha)) = \alpha$ and $\phi'(\phi(\gamma)) = \gamma$.
  Furthermore, $$\phi'(\phi(\beta)) = \frac{\gamma - \frac{\gamma-\alpha\beta}{1-\beta}}{\alpha - \frac{\gamma-\alpha\beta}{1-\beta}} = \frac{\frac{\alpha\beta-\gamma\beta}{1-\beta}}{\frac{\alpha-\gamma}{1-\beta}} =\beta.$$
  It now follows that $\phi'(\phi(x))=x$ for any $x \in \mathbb{H}_5$.
  Similarly $\phi(\phi'(x))=x$ for any $x \in \mathbb{U}_3$.
  Hence $\phi$ is a bijection with inverse $\phi'$, so the partial fields $\mathbb{H}_5$ and $\mathbb{U}_3$ are isomorphic.
\end{proof}

The next lemma is a consequence of \cite[Lemmas~5.7 and~5.25]{OSV00} when $\mathbb{P} = \mathbb{U}_2$, and is proved for $\mathbb{P} = \mathbb{H}_5$ in \cite[Lemma~7.3.16]{vanZwam09}.

\begin{lemma}[{\cite{OSV00,vanZwam09}}]
  \label{u2stabs}
  For $\mathbb{P} \in \{\mathbb{H}_5,\mathbb{U}_2\}$,
  the matroids $U_{2,5}$ and $U_{3,5}$ are
  non-binary, $3$-connected,
  strong $\mathbb{P}$-stabilizers.
\end{lemma}

\subsection*{Certifying non-representability}

Let $\mathbb{P}$ be a partial field.
Let $M$ be a matroid and let $E(M)=X \cup Y$ where $X$ and $Y$ are disjoint.
Let $A$ be an $X \times Y$ matrix with entries in $\mathbb{P}$ such that, for some distinct $a, b \in Y$, both $A-a$ and $A-b$ are $\mathbb{P}$-matrices, $M \del a=M[I|A-a]$, and $M \del b=M[I|A-b]$. 
Then we say $A$ is an $X \times Y$ \emph{companion $\mathbb{P}$-matrix} for $M$.

Let $B$ be a basis of $M$.
We write $B^*$ to denote $E(M)-B$.
Let $A$ be a $B\times B^*$ matrix with entries in $\mathbb{P}$.
A subset~$Z$ of $E(M)$ \textit{incriminates} the pair $(M, A)$ if $A[Z]$ is square and one of the following holds: 
\begin{enumerate}
 \item $\det(A[Z])\notin \mathbb{P}$,
 \item $\det(A[Z])=0$ but $B\triangle Z$ is a basis of $M$, or
 \item $\det(A[Z])\neq 0$ but $B\triangle Z$ is dependent in $M$.
\end{enumerate}
The next lemma follows immediately. 

\begin{lemma}
  Let $M$ be a matroid, let $A$ be an $X\times Y$ matrix with entries in $\mathbb{P}$, where $X$ and $Y$ are disjoint, and $X\cup Y=E(M)$. Exactly one of the following statements is true:
\begin{enumerate}
 \item $A$ is a $\mathbb{P}$-matrix and $M=M[I | A]$, or
 \item there is some $Z\subseteq X\cup Y$ that incriminates $(M, A)$.
\end{enumerate}
\end{lemma}

Let $M$ be an excluded minor for the class of $\mathbb{P}$-representable matroids.  We will obtain a $B \times B^*$ companion $\mathbb{P}$-matrix $A$ for $M$ such that $\{x,y,a,b\}$ incriminates $(M,A)$ for some distinct $x,y \in B$ and $a,b \in B^*$.
In this setting, for $p \in B$ and $q \in B^*$ where $A_{pq} \neq 0$, we say that the pivot $A^{pq}$ is \emph{allowable} if $\{p,q\} \cap \{x,y,a,b\} \neq \emptyset$ and $\{x,y,a,b\} \triangle \{p,q\}$ incriminates $(M,A^{pq})$, or $\{p,q\} \cap \{x,y,a,b\} = \emptyset$ and $\{x,y,a,b\}$ incriminates $(M,A^{pq})$.
The next two lemmas describe situations where a pivot is allowable.

\begin{lemma}[{\cite[Lemma 5.10]{MvZW10}}]
  \label{allowablexyrow2}
  Let $A$ be a $B\times B^{*}$ companion $\mathbb{P}$-matrix for $M$. Suppose that $\{x,y,a,b\}$ incriminates $(M,A)$, for pairs $\{x,y\}\subseteq B$ and $\{a,b\}\subseteq B^{*}$. If $p\in \{x,y\}$, $q\in B^{*}-\{a,b\}$, and $A_{pq}\neq 0$, then 
  $A^{pq}$ is an allowable pivot.
\end{lemma}

\begin{lemma}[{\cite[Lemma 5.11]{MvZW10}}]
  \label{allowablenonxy2}
  Let $A$ be a $B\times B^{*}$ companion $\mathbb{P}$-matrix for $M$. Suppose that $\{x,y,a,b\}$ incriminates $(M,A)$, for pairs $\{x,y\}\subseteq B$ and $\{a,b\}\subseteq B^{*}$. If $p\in B-\{x,y\}$, $q\in B^{*}-\{a,b\}$, $A_{pq}\neq 0$, and either $A_{pa}=A_{pb}=0$ or $A_{xq}=A_{yq}=0$, then 
  $A^{pq}$ is an allowable pivot.
\end{lemma}

\subsection*{Delta-wye exchange}

Let $M$ be a matroid with a triangle $T=\{a,b,c\}$.
Consider a copy of $M(K_4)$ having $T$ as a triangle with $\{a',b',c'\}$ as the complementary triad labelled such that $\{a,b',c'\}$, $\{a',b,c'\}$ and $\{a',b',c\}$ are triangles.
Let $P_{T}(M,M(K_4))$ denote the generalised parallel connection of $M$ with this copy of $M(K_4)$ along the triangle $T$.
Let $M'$ be the matroid $P_{T}(M,M(K_4))\backslash T$ where the elements $a'$, $b'$ and $c'$ are relabelled as $a$, $b$ and $c$ respectively.
The matroid $M'$ is said to be obtained from $M$ by a \emph{\dY\ exchange} on the triangle~$T$, and is denoted $\Delta_T(M)$.
Dually, $M''$ is obtained from $M$ by a \emph{\Yd\ exchange} on the triad $T^*=\{a,b,c\}$ if $(M'')^*$ is obtained from $M^*$ by a \dY\ exchange on $T^*$.  The matroid $M''$ is denoted $\nabla_{T^*}(M)$.

We say that a matroid $M_1$ is \emph{\dY-equivalent} to a matroid $M_0$ if $M_1$ can be obtained from $M_0$ by a sequence of \dY\ and \Yd\ exchanges on coindependent triangles and independent triads, respectively.
We let $\Delta^*(M)$ denote the set of matroids that are \dY-equivalent to $M$ or $M^*$.

Oxley, Semple, and Vertigan proved that the set of excluded minors for $\mathbb{P}$-representability is closed under \dY\ exchange.
\begin{proposition}[{\cite[Theorem~1.1]{OSV00}}]
  \label{osvdelta}
  Let $\mathbb{P}$ be a partial field, and let $M$ be an excluded minor for the class of $\mathbb{P}$-representable matroids.
  If $M' \in \Delta^*(M)$, then $M'$ is an excluded minor for the class of $\mathbb{P}$-representable matroids.
\end{proposition}

\subsection*{Robust and strong elements, and bolstered bases}

Let $M$ be a $3$-connected matroid, let $B$ be a basis of $M$, and let $N$ be a $3$-connected minor of $M$.
Recall that we write $B^*$ to denote $E(M)-B$.
An element $e \in E(M)$ is \textit{$(N,B)$-robust} if either
\begin{enumerate}
 \item $e\in B$ and $M/e$ has an $N$-minor, or
 \item $e\in B^*$ and $M\del e$ has an $N$-minor.
\end{enumerate}

\noindent
Note that an $N$-flexible element of $M$ is clearly $(N,B)$-robust for any basis~$B$ of $M$. 
An element $e \in E(M)$ is \textit{$(N,B)$-strong} if either
\begin{enumerate}
 \item $e\in B$ and $\si(M/e)$ is $3$-connected and has an $N$-minor, or
 \item $e\in B^*$ and $\co(M\del e)$ is $3$-connected and has an $N$-minor.
\end{enumerate}

Now let $\{a,b\}$ be a pair of elements of $M$ such that $M \ba a,b$ is $3$-connected with an $N$-minor.
Let $B$ be a basis of a matroid $M \ba a,b$, and let $A$ be a $B\times B^{*}$ companion $\mathbb{P}$-matrix of $M$ such that $\{x,y,a,b\}$ incriminates $(M,A)$, for some $\{x,y\}\subseteq B$.
If either
  \begin{itemize}
    \item [(i)] $M \ba a,b$ has exactly one $(N,B)$-strong element $u$ outside of $\{x,y\}$, and $\{u,x,y\}$ is a triad of $M \ba a,b$; or
    \item [(ii)] $M \ba a,b$ has no $(N,B')$-strong elements outside of $\{x',y'\}$ for every choice of basis~$B'$ with a $B'\times (B')^{*}$ companion $\mathbb{P}$-matrix $A'$ of $M$ such that $\{x',y',a,b\}$ incriminates $(M,A')$, for some $\{x',y'\}\subseteq B'$;
  \end{itemize}
  then $B$ is a \emph{strengthened} basis.

In other words, a basis~$B$ is strengthened if $B$ is chosen such that either there is one $(N,B)$-strong element $u$ of $M \ba a,b$ outside of $\{x,y\}$, and $\{u,x,y\}$ is a triad; or there are no $(N,B)$-strong elements outside of $\{x,y\}$, and, moreover, there are no $(N,B')$-strong elements outside of $\{x',y'\}$ for any choice of basis~$B'$ with an incriminating set $\{x',y',a,b\}$ where $\{x',y'\} \subseteq B'$.

In particular, for a strengthened basis~$B$ with no $(N,B)$-strong elements, an allowable pivot cannot introduce an $(N,B)$-strong element.

Now suppose $B$ is strengthened.
We say that $B$ is \emph{bolstered} if
\begin{enumerate}
  \item when $M \ba a,b$ has no $(N,B)$-strong elements outside of $\{x,y\}$, then
    for any $B_1\times B_1^{*}$ companion $\mathbb{P}$-matrix~$A_1$ where $\{x_1,y_1,a,b\}$ incriminates $(M,A_1)$, with $\{x_1,y_1\}\subseteq B_1$ and $\{a,b\}\subseteq B_1^{*}$, the number of $(N,B)$-robust elements of $M \ba a,b$ outside of $\{x,y\}$ is at least the number of $(N,B_1)$-robust elements of $M \ba a,b$ outside of $\{x_1,y_1\}$; or
  \item when $M \ba a,b$ has an $(N,B)$-strong element $u$ of $M \ba a,b$ outside of $\{x,y\}$, then for any $B_1\times B_1^{*}$ companion $\mathbb{P}$-matrix~$A_1$ such that 
    \begin{enumerate}[label=\rm(\Roman*)]
      \item $\{x,y,a,b\}$ incriminates $(M,A_1)$, with $\{x,y\}\subseteq B_1$ and $\{a,b\}\subseteq B_1^{*}$, and
      \item $u$ is the only $(N,B_1)$-strong element of $M \ba a,b$, with $u \in B_1^*$,
    \end{enumerate}
    the number of $(N,B)$-robust elements of $M \ba a,b$ is at least the number of $(N,B_1)$-robust elements of $M \ba a,b$. 
\end{enumerate}
Loosely speaking, a strengthened basis~$B$ is bolstered if no allowable pivot increases the number of elements that are 
robust but not strong.

\section{Excluded minors are almost fragile}
\label{almostfragsec}

We now recap results that we require from \cite{BCOSW20,paper1}.
All of these results, appearing in the remainder of this section, are under the following hypotheses:
Let $\mathbb{P}$ be a partial field.
Let $M$ be an excluded minor for the class of $\mathbb{P}$-representable matroids, and let $N$ be a non-binary $3$-connected strong stabilizer for the class of $\mathbb{P}$-representable matroids, where $M$ has an $N$-minor.

The first result implies that we can essentially restrict attention to an excluded minor with no triads.  A proof appears in \cite{paper1}, but it is essentially a consequence of the main theorem proved in \cite{BWW20,BWW21,BWW22}.

\begin{lemma}[{\cite[Lemma~3.1]{paper1}}]
  \label{notriads2}
  Suppose that $|E(M)| \ge |E(N)| + 10$.
  Then there exists a matroid~$M_1 \in \Delta^*(M)$ such that $M_1$ has a pair of elements $\{a,b\}$ for which $M_1 \ba a,b$ is $3$-connected and has a $\Delta^*(N)$-minor, and $M_1$ has no triads.
\end{lemma}

The following theorem addresses the case where $M \ba a,b$ is not $N$-fragile, and is extracted from {\cite[Theorem~6.7]{BCOSW20}}.
For item~\ref{easywin2}, the fact that the triangle is closed is established as \cite[Lemma~3.4]{paper1}.

\begin{theorem}[{\cite[Theorem~6.7(ii)(b)]{BCOSW20}}]
  \label{bcosw-thm}
  Let $a,b \in E(M)$ be a pair of elements for which $M \ba a,b$ is $3$-connected with an $N$-minor.
  Suppose $|E(M)| \ge |E(N)| + 10$, and $M \ba a,b$ is not $N$-fragile.
  Then
  \begin{enumerate}
    \item $M$ has a bolstered basis~$B$, and a $B\times B^{*}$ companion $\mathbb{P}$-matrix $A$ for which $\{x,y,a,b\}$ incriminates $(M,A)$, where $\{x,y\}\subseteq B$ and $\{a,b\}\subseteq B^{*}$, and there is an element $u \in B^*-\{a,b\}$ that is $(N,B)$-strong in $M \ba a,b$;
    \item either\label{bcoswii}
      \begin{enumerate}[label=\rm(\Roman*)]
        \item the $N$-flexible, and $(N,B)$-robust, elements of $M\del a,b$ are contained in $\{u,x,y\}$, or\label{bcoswii1}
        \item the $N$-flexible, and $(N,B)$-robust, elements of $M\del a,b$ are contained in $\{u,x,y,z\}$, where $z \in B$, and $(z,u,x,y)$ is a maximal fan of $M \del a,b$, or\label{bcoswii2}
        \item the $N$-flexible, and $(N,B)$-robust, elements of $M\del a,b$ are contained in $\{u,x,y,z,w\}$, where $z \in B$, $w \in B^*$, and $(w,z,x,u,y)$ is a maximal fan of $M \del a,b$;\label{bcoswii3}
      \end{enumerate}
    \item the unique triad in $M \ba a,b$ containing $u$ is $\{u,x,y\}$;\label{bcoswiii}
    \item $M$ has a cocircuit $\{x,y,u,a,b\}$; and
    \item for some $d\in \{a,b\}$, the set $\{d,x,y\}$ is a closed triangle.\label{easywin2}
  \end{enumerate}
\end{theorem}

A consequence of this theorem is that $M \ba a,b$ has no $M(K_4)$ restriction or co-restriction, as shown below.

\begin{lemma}
  \label{gadgetnotMK4}
  Let $a,b \in E(M)$ be a pair of elements for which $M \ba a,b$ is $3$-connected with an $N$-minor.
  Suppose $|E(M)| \ge |E(N)| + 10$, and $M \ba a,b$ is not $N$-fragile, so \cref{bcosw-thm} holds.
  Let $Z$ be the set of $N$-flexible elements of $M \ba a,b$.
  Then there is no set $Z'$ containing $Z$ such that either $(M\ba a,b)|Z' \cong M(K_4)$ or $(M\ba a,b)^*|Z' \cong M(K_4)$.
\end{lemma}
\begin{proof}
  Towards a contradiction, let $Z'$ be a set containing $Z$ such that $M'|Z' \cong M(K_4)$ for some $M' \in \{M\ba a,b,(M\ba a,b)^*\}$.
  Note that every element of $Z'$ is in at least $2$ triangles $T_1$ and $T_2$ of $M'$, where $r_{M'}(T_1 \cup T_2)=3$.
  If $u$ is not $N$-flexible, then $u$ is not $N$-contractible.  But this implies that $x$ and $y$ are not $N$-deletable, and it follows that $M \ba a,b$ has no $N$-flexible elements, contradicting that $M \ba a,b$ is not $N$-fragile.  So $u$ is $N$-flexible.
  Since $u \in Z \subseteq Z'$ and $\{u,x,y\}$ is the unique triad containing $u$, we have $M' = M \ba a,b$.
  Now $u$ is in two triangles of $(M \ba a,b)|Z'$ that are not both contained in a common segment.
  By orthogonality, one of these two triangles contains $\{u,x\}$, and the other contains $\{u,y\}$ (and, in particular, $x,y \in Z'$).
  Since $\{u,y\}$ is contained in a triangle, neither \ref{bcoswii2} nor \ref{bcoswii3} of \cref{bcosw-thm}\ref{bcoswii} holds.
  So \cref{bcosw-thm}\ref{bcoswii}\ref{bcoswii1} holds.
  Then there exist elements $e,f,g \in E(M \ba a,b) - \{x,y,u\}$ such that $\{u,x,f\}$, $\{u,y,g\}$, $\{e,f,g\}$ and $\{x,y,e\}$ are triangles.
  But the triangle $\{x,y,e\}$ contradicts \cref{bcosw-thm}\ref{easywin2}.
\end{proof}

The next three results are from \cite{paper1}.
\begin{lemma}[{\cite[Lemma~3.3]{paper1}}]
  \label{wmatype1again}
  Suppose $M$ has a pair of elements $\{a,b\}$ such that $M\ba a,b$ is $3$-connected with an $N$-minor, $|E(M)| \ge |E(N)| + 10$, and $M \ba a,b$ is not $N$-fragile, so \cref{bcosw-thm} holds.
  Assume \cref{bcosw-thm}\ref{easywin2} holds with $d=b$.
  Then, either
  \begin{enumerate}
    \item the $N$-flexible elements of $M \ba a,b$ are contained in $\{u,x,y\}$, or
    \item $M \ba a,x$ is $3$-connected with an $N$-minor, but is not $N$-fragile, and there are at most three $N$-flexible elements in $M \ba a,x$.
  \end{enumerate}
\end{lemma}

\begin{theorem}[{\cite[Theorem~3.5]{paper1}}]
  \label{thegrandfantasy}
  Suppose $M$ has a pair of elements $\{a,b\}$ such that $M\ba a,b$ is $3$-connected with an $N$-minor, $M$ has no triads, $|E(M)| \ge |E(N)| + 11$, and $M \ba a,b$ is not $N$-fragile, so \cref{bcosw-thm} holds.
  If the $N$-flexible elements of $M \ba a,b$ are contained in $\{u,x,y\}$, then, for every $b' \in B-\{x,y\}$, the element $b'$ is $N$-essential in at least one of $M \ba a,b\ba u$ and $M \ba a,b/u$.
\end{theorem}

\begin{lemma}[{\cite[Lemma~4.1]{paper1}}]
  \label{subfrag3conn}
  Suppose $M$ has a pair of elements $\{a,b\}$ such that $M\ba a,b$ is $3$-connected with an $N$-minor, $|E(M)| \ge |E(N)| + 10$, and $M \ba a,b$ is not $N$-fragile, so \cref{bcosw-thm} holds.
  Suppose the $N$-flexible elements of $M \ba a,b$ are contained in $\{u,x,y\}$,
  and $M \ba a,b,u/x$ is not $N$-fragile.
  Then either $M \ba a,b,u/x/y$ or $M \ba a,b,u/x \ba y$ is $3$-connected and $N$-fragile.
\end{lemma}

The next theorem is the counterpart to \cref{bcosw-thm} that addresses the case where $M \ba a,b$ is $N$-fragile.
Item~\ref{nostronginbasis} was established as \cite[Lemma~3.1]{BCOSW20}.

\begin{theorem}[{\cite[Theorem~6.7(ii)(a)]{BCOSW20}}]
  \label{fragilecase}
  Let $a,b \in E(M)$ be a pair of elements for which $M \ba a,b$ is $3$-connected with an $N$-minor.
  Suppose $|E(M)| \ge |E(N)| + 10$, and $M \ba a,b$ is $N$-fragile.
  Then
  \begin{enumerate}
    \item $M$ has a bolstered basis~$B$, and a $B\times B^{*}$ companion $\mathbb{P}$-matrix $A$ for which $\{x,y,a,b\}$ incriminates $(M,A)$, where $\{x,y\}\subseteq B$ and $\{a,b\}\subseteq B^{*}$; and
    \item $M \ba a,b$ has at most one $(N,B)$-robust element outside of $\{x,y\}$, where if such an element $u$ exists, then $u \in B^*-\{a,b\}$ is an $(N,B)$-strong element of $M \ba a,b$, and $\{u,x,y\}$ is a coclosed triad of $M\ba a,b$.
    \item if $v$ is an $(N,B_1)$-strong element of $M \ba a,b$, for some basis $B_1$ such that there exists a $B_1 \times B_1^*$ companion $\mathbb{P}$-matrix~$A_1$ of $M$ where $\{x_1, y_1, a, b\}$ incriminates $(M, A_1)$, and $\{x_1,y_1\} \subseteq B_1$ and $\{a,b\} \subseteq B_1^*$, then $v \notin B_1 - \{x_1,y_1\}$.\label{nostronginbasis}
  \end{enumerate}
\end{theorem}

The last theorem implies that $M \ba a,b$ cannot have arbitrarily large fans, as proved below.

\begin{corollary}
  \label{fragilefanscase}
  Assume that $N$ has no $4$-element fans, and let $a,b \in E(M)$ be a pair of elements for which $M \ba a,b$ is $3$-connected with an $N$-minor.
  Suppose that $|E(M)| \ge |E(N)| + 10$ and $M \ba a,b$ is $N$-fragile, so \cref{fragilecase} holds.
  Then $M \ba a,b$ has no fan with more than five elements.
  Moreover, if $(f_1,f_2,f_3,f_4,f_5)$ is a fan 
  in $M \ba a,b$, then either 
  \begin{enumerate}
    \item there is a triad $\{u,x,y\} \in \{\{f_1,f_2,f_3\},\{f_3,f_4,f_5\}\}$, where $u$ is the unique $(N,B)$-robust element outside of $\{x,y\}$, and $u \in \{f_2,f_4\}$,\label{ffcc1}
    \item $\{f_1,f_2,f_3\}$ is a triad, and $\{f_2,f_4\} = \{x,y\}$,\label{ffcc2}
    \item each element in $\{f_2,f_3,f_4\}$ is $N$-essential, or\label{ffcc3}
    \item $\{f_1,f_2,f_3\}$ is a triad, and $\si(M/f_3)$ is not $3$-connected.\label{ffcc4}
  \end{enumerate}
\end{corollary}
\begin{proof}
  By \cref{fragilecase}, $M \ba a,b$ has at most one $(N,B)$-robust element outside of $\{x,y\}$ and if such an element $u$ exists, then $u$ is an $(N,B)$-strong element of $M \ba a,b$ that is in $B^* - \{a,b\}$, and $\{u,x,y\}$ is a coclosed triad of $M \ba a,b$.
  Let $F$ be a maximal fan of $M \ba a,b$ of size at least~$5$.
  By \cref{fragilefanelements}, if $s$ is a spoke element of $F$ that is not $N$-essential, then $M \ba a,b \ba s$ has an $N$-minor;
  whereas if $t$ is a rim element of $F$ that is not $N$-essential, then $M \ba a,b / t$ has an $N$-minor.
  The only $(N,B)$-robust elements are in $\{x,y\}$ or in $\{x,y,u\}$ for an element $u \in B^*-\{a,b\}$.
  We deduce that $s \in B \cup u$ for any such spoke $s$ of $F$, and
  $t \in B^* \cup \{x,y\}$ for any such rim $t$ of $F$.

  \begin{claim}
    If $M \ba a,b$ has a triangle $\{t_1,t_2,t_3\}$ where $\{t_1,t_3\} \subseteq B$, the element $t_2$ is $N$-contractible in $M \ba a,b$, and $\si(M/t_2)$ is $3$-connected, then $\{t_1,t_3\} = \{x,y\}$.
    \label{ffcclaim}
  \end{claim}

  \begin{subproof}
    Assume $M \ba a,b$ has such a triangle $\{t_1,t_2,t_3\}$.
    As $M \ba a,b$ is $N$-fragile, $t_2$ is not $N$-deletable and, in particular, $t_2 \neq u$.
    Moreover, $t_2 \in B^*$.
    Suppose $\{t_1,t_3\}$ avoids $\{x,y\}$.
    Then, by \cref{allowablenonxy2}, a pivot on $A_{t_1t_2}$ is allowable.
    Let $B' = B \triangle \{t_1,t_2\}$.
    Now $t_2 \in B'-\{x,y\}$ and $t_2$ is an $(N, B')$-strong element, contradicting \cref{fragilecase}\ref{nostronginbasis}.
    Next suppose that $t_1 \in \{x,y\}$ but $t_3 \notin \{x,y\}$.
    Without loss of generality, let $t_1 = x$.
    First, observe that if the element $u$ exists, then by orthogonality between the triad $\{u,x,y\}$ and the triangle $\{t_1,t_2,t_3\}$, we have $t_2=u$, a contradiction.
    So $M \ba a,b$ has no $(N,B)$-strong elements.
    By \cref{allowablexyrow2}, the pivot on $A_{xt_2}$ is allowable.
    Let $B' = B \triangle \{x,t_2\}$.
    Since $t_2$ is $N$-contractible, and $\{x,t_3\}$ is a parallel pair in $M \ba a,b / t_2$, the element $x$ is $N$-deletable.
    Then $x$ is $(N, B')$-robust, whereas $t_2$ is not $(N,B)$-robust, so the number of $(N,B')$-robust elements outside of $\{t_2,y\}$ is greater than the number of $(N,B)$-robust elements outside of $\{x,y\}$, contradicting that $B$ is a bolstered basis.
    We deduce that $\{t_1,t_3\} = \{x,y\}$, as required.
  \end{subproof}

  Let $(f_1,f_2,\dotsc,f_\ell)$ be a fan ordering of $F$.
  Suppose first that $M \ba a,b$ has an $(N,B)$-robust element $u$, where $u \in F$.
  If $u$ is a rim element of $F$, then $u$ is not $N$-deletable by \cref{fragilefanelements}, contradicting that $u$ is $(N,B)$-robust.  So we may
  assume that $u$ is a spoke element $f_i$ of $F$.
  Suppose $3 \le i \le \ell-2$.
  Then $f_{i-2}$ and $f_{i+2}$ are spokes, so they are $N$-deletable by \cref{fragilefanelements,fragilefans}.
  Since $u$ is the only $(N,B)$-robust element of $M \ba a,b$ in $B^*$, we have $\{f_{i-2},f_{i+2}\} \subseteq B$.
  Let $F' = \{f_{i-2},f_{i-1},f_{i},f_{i+1},f_{i+2}\}$.
  Observe that $\{x,y\} \neq \{f_{i-1},f_{i+1}\}$, for otherwise the rank-$3$ fan $F'$ contains four elements of the basis~$B$.
  By orthogonality between the triad $\{u,x,y\}$ and the triangles $\{f_{i-2},f_{i-1},u\}$ and $\{u,f_{i+1},f_{i+2}\}$, we have $\{x,y\} \subseteq F'$.
  Now $F'$ contains distinct triads $\{f_{i-1},u,f_{i+1}\}$ and $\{u,x,y\}$, so $r^*_{M \ba a,b}(F') \le 3$.
  But $r(F') = 3$, so $\lambda_{M \ba a,b}(F') \le 1$, a contradiction.

  Next, let $i = 2$.
  Suppose $\{f_1,f_2,f_3\} = \{u,x,y\}$.
  In the case that $|F| \ge 6$, the set $\{f_4,f_5,f_6\}$ is a triangle.
  Then, by \cref{fragilefanelements,fragilefans}, $f_5$ is $N$-contractible, so $f_4$ and $f_6$ are $N$-deletable.
  Moreover, $\si(M/f_5)$ is $3$-connected by \cref{fanends}.
  So $\{f_4,f_6\} = \{x,y\}$ by \cref{ffcclaim}, a contradiction.
  Thus $|F| = 5$, and \ref{ffcc1} holds in this case.
  Now we may assume that $\{f_1,f_2,f_3\} \neq \{u,x,y\}$.
  Observe that $\{f_1,f_3\}$ does not meet $\{x,y\}$,
  for otherwise $\{u,x,y\}$ is not coclosed in $M \ba a,b$. 
  In particular, $f_3 \notin \{x,y\}$.
  So $f_4 \in \{x,y\}$, by orthogonality with the triangle $\{u,f_3,f_4\}$.
  Since $u$ is $N$-deletable in $M\ba a,b$, and $\{f_1,f_3\}$ is a series pair in $M \ba a,b \ba u$, the element $f_3$ is $N$-contractible.
  So $f_3 \in B^*$.
  Then $f_1 \in B$, since the triad $\{f_1,f_2,f_3\}$ cannot be contained in $B^*$.
  But then $f_1$ is $(N,B)$-robust by \cref{fragilefanelements,fragilefans}, and $f_1 \notin \{x,y\}$, a contradiction.

  Now let $i=1$, so $u$ is a spoke end of $F$.
  By orthogonality, $\{x,y\}$ meets $\{f_2,f_3\}$.
  If $f_3 \in \{x,y\}$, then $f_3$ is in distinct triads $\{u,x,y\}$ and $\{f_2,f_3,f_4\}$, which contradicts \cref{fanunique}.
  So $f_3 \notin \{x,y\}$.
  Without loss of generality, say $f_2 =x$.
  If $y \notin F$, then $F \cup y$ is a fan, contradicting that $F$ is maximal.
  So $y \in F$.
  Then, by orthogonality, $y=f_\ell$ is a rim end.
  Moreover, $u \in \cl(F-u)$ due to the triangle $\{u,x,f_3\}$, and $u \in \cocl_{M \ba a,b}(F-u)$ due to the triad $\{u,x,y\}$.
  But $\lambda_{M \ba a,b}(F-u) \le 2$ and so $\lambda_{M \ba a,b}(F) \le 1$, implying that $|E(M \ba a,b)-F| \le 1$.
  If $F=E(M \ba a,b)$, then it follows that $M \ba a,b$ is a rank-$\frac{\ell}{2}$ wheel or whirl.
  But $f_i \in B$ for each $i \in \{3,5,\dotsc,\ell-1\}$, and $x,y \in B$, so $|B| = \ell/2 + 1$, a contradiction.
  Now $|E(M \ba a,b)-F| = 1$, so let $E(M \ba a,b)-F = \{q\}$.
  Then $(\{u,x\}, f_3, f_4, \dotsc, f_{\ell-1},\{y,q\})$ is a path of $3$-separations, so $\{f_{\ell-1},y,q\}$ is a triangle or a triad.  But it is not a triangle, by orthogonality with the triad $\{u,x,y\}$, and it is not a triad, since $F$ is maximal, a contradiction.

  Finally, we may assume that there are no $(N,B)$-robust elements contained in $F - \{x,y\}$.
  Suppose $F$ contains a $5$-element fan $F'$ with fan ordering $(f_1',f_2',\dotsc,f_5')$ such that $\{f'_1,f'_2,f'_3\}$ is a triangle.
  Then $\{f'_3,f'_4,f'_5\}$ is also a triangle, and $f'_1$, $f'_3$, and $f'_5$ are the spoke elements.
  \Cref{fragilefanelements,fragilefans} imply that $f'_1$ and $f'_5$ are $N$-deletable in $M \ba a,b$.
  Moreover, $\si(M/f'_2)$ and $\si(M/f'_4)$ are $3$-connected by \cref{fanends}.
  If $f'_3$ is $N$-deletable, then $\{f'_1,f'_3\} = \{x,y\} = \{f'_3,f'_5\}$ by \cref{ffcclaim}, a contradiction.
  Thus $f'_3$ is $N$-essential and, in particular, $|F|=5$, by \cref{fragilefans}.
  Due to the triangles $\{f'_1,f'_2,f'_3\}$ and $\{f'_3,f'_4,f'_5\}$, and by \cref{fragilefanelements}, it follows that $f'_2$ and $f'_4$ are also $N$-essential.
  So \ref{ffcc3} holds in this case.
  We may now assume that $|F| = 5$, and when $(f_1,f_2,\dotsc,f_5)$ is a fan ordering of $F$, the set $\{f_2,f_3,f_4\}$ is a triangle.

  Next we claim that if $\{f_2,f_3,f_4\}$ contains an element that is not $N$-essential, then no element of $F$ is $N$-essential.
  Suppose $f_3$ is not $N$-essential.
  Then $f_3$ is $N$-contractible, by \cref{fragilefanelements,fragilefans}.
  Since $\{f_2,f_4\}$ is a parallel pair in $M \ba a,b / f_3$, the elements $f_2$ and $f_4$ are $N$-deletable, so no element of $F$ is $N$-essential.
  Similarly, if $f_2$ (or $f_4$) is not $N$-essential, then no element in $F$ is $N$-essential.
  This proves the claim.

  We may now assume $\{f_2,f_3,f_4\}$ contains an element that is not $N$-essential, otherwise \ref{ffcc3} holds.
  Then, by the foregoing and \cref{fragilefanelements}, $f_2$ and $f_4$ are $N$-deletable, and $f_3$ is $N$-contractible.
  So $\{f_2,f_4\} \subseteq B$, and hence $f_3 \in B^*$.
  If $\si(M/f_3)$ is not $3$-connected, then \ref{ffcc4} holds.
  Otherwise, $\si(M/f_3)$ is $3$-connected, in which case $\{f_2,f_4\} = \{x,y\}$, by \cref{ffcclaim}. So \ref{ffcc2} holds.
\end{proof}

\section{\texorpdfstring{$\utfutf$}{\{U(2,5),U(3,5)\}}-fragile matroids}
\label{utfutfsec}

In this section, we recap some known properties of $\{U_{2,5},U_{3,5}\}$-fragile matroids \cite{CMvZW16}, and prove some further structural properties of this class that have not previously been explicitly stated.
Recall that, by definition, when we say a matroid is $\utfutf$-fragile, it has an $\utfutf$-minor.

Throughout this section, we focus on $\utfutf$-fragile matroids, rather than $U_{2,5}$-fragile or $U_{3,5}$-fragile matroids.
\Cref{utfutfequiv}, which follows from the following well-known lemma, connects these classes of fragile matroids.

\begin{lemma}[see {\cite[Proposition~12.2.15]{Oxley11}}, for example]
  \label{utfutfprop}
  Let $M$ be a $3$-connected matroid with rank and corank at least~$3$.
  Then $M$ has a $U_{2,5}$-minor if and only if $M$ has a $U_{3,5}$-minor.
\end{lemma}

\begin{corollary}
  \label{utfutfequiv}
  Let $M$ be a $3$-connected matroid with rank and corank at least~$3$, and $|E(M)| \ge 7$.
  Then $M$ is $U_{2,5}$-fragile and $U_{3,5}$-fragile if and only if $M$ is $\utfutf$-fragile.
\end{corollary}
\begin{proof}
  Suppose $M$ is $\utfutf$-fragile.
  Then $M$ has a $\utfutf$-minor, so, by \cref{utfutfprop}, $M$ has both a $U_{2,5}$- and a $U_{3,5}$-minor.
  So clearly $M$ is $U_{2,5}$-fragile and $U_{3,5}$-fragile.

  Now let $M$ be $U_{2,5}$-fragile and $U_{3,5}$-fragile and, towards a contradiction, suppose $M$ is not $\utfutf$-fragile.
  Clearly $M$ has a $\utfutf$-minor.
  So, by duality, we may assume that, for some $e \in E(M)$, the matroid $M \ba e$ has a $U_{3,5}$-minor (but $M /e$ does not), and $M/e$ has a $U_{2,5}$-minor (but $M \ba e$ does not).
  Since $U_{2,5}$ and $U_{3,5}$ are $3$-connected, $\co(M \ba e)$ has a $U_{3,5}$-minor and $\si(M/e)$ has a $U_{2,5}$-minor.
  In particular, $\co(M \ba e)$ has rank at least~$3$, and $\si(M/e)$ has corank at least~$3$.
  Since $|E(M)| \ge 7$, the rank or corank of $M$ is at least~$4$.

  Assume without loss of generality that $M$ has corank at least~$4$.
  Then $M \ba e$ and $\co(M \ba e)$ have corank at least~$3$.
  Since $M \ba e$ has a $U_{3,5}$-minor, it is $U_{3,5}$-fragile.
  As $M$ is $3$-connected, and by \cref{genfragileconn}, $M \ba e$ is $3$-connected up to series classes.
  Now $\co(M \ba e)$ is $3$-connected and has rank and corank at least~$3$.
  Thus $\co(M \ba e)$, and hence $M \ba e$, has a $U_{2,5}$-minor by \cref{utfutfprop}.
  But $M / e$ has a $U_{2,5}$-minor, so $e$ is $U_{2,5}$-flexible in $M$, and hence $M$ is not $U_{2,5}$-fragile, a contradiction.
  We deduce that $M$ is $\utfutf$-fragile, as required.
\end{proof}

Let $(x_1,x_2,x_3)$ be an ordered subset of elements of a matroid $M$ in which $\{x_1,x_2,x_3\}$ is a triangle~$T$.
Let $W$ be a copy of the rank-$r$ wheel $M(\mathcal{W}_r)$ having a triangle $\{x_1,x_2,x_3\}$ where $x_1$ and $x_3$ are spoke elements.
Let $X \subseteq \{x_1,x_2,x_3\}$ such that $x_2 \in X$.
We say that \emph{gluing an $r$-wheel onto $(x_1,x_2,x_3)$} (with \emph{remove set} $X$) is the operation by which we obtain the matroid $P_T(M,W) \ba X$, where $P_T(M,W)$ is the generalized parallel connection of $M$ and $W$ across the triangle $T$.

The following was proved by Chun et al.~\cite{CCCMWvZ13,CCMvZ15}.
Geometric representations of the matroids 
$M_{9,9}$,
$X_8$ and $Y_8$ are given in \cref{pathdescminors-fig}.
Note that $X_8$ is self-dual.

\begin{theorem}[{\cite[Theorem~1.3 and Corollary~1.4]{CCCMWvZ13}}]
  \label{ccmwvz-result}
  Let $\mathbb{P} \in \{\mathbb{H}_5, \mathbb{U}_2\}$, and
  let $M$ be a $3$-connected $\{U_{2,5},U_{3,5}\}$-fragile $\mathbb{P}$-representable matroid.
  Then either
  \begin{enumerate}
    \item $M$ has an $\{X_8,Y_8,Y_8^*\}$-minor;\label{pathlikecase}
    \item $M$ is isomorphic to a matroid in $\{U_{2,6},U_{4,6},P_6,M_{9,9},M_{9,9}^*\}$;
    \item $M$ or $M^*$ can be obtained from $Y_8 \ba 4$ 
      by gluing a wheel to $(1,5,7)$;
    \item $M$ or $M^*$ can be obtained from $U_{2,5}$, with $E(U_{2,5})=\{e_1,e_2,e_3,e_4,e_5\}$, by gluing up to two wheels to $(e_1,e_2,e_3)$ and $(e_3,e_4,e_5)$; or
    \item $M$ or $M^*$ can be obtained from $U_{2,5}$, with $E(U_{2,5})=\{e_1,e_2,e_3,e_4,e_5\}$, by gluing up to three wheels to $(e_1,e_3,e_2)$, $(e_1,e_4,e_2)$, and $(e_1, e_5, e_2)$.
  \end{enumerate}
\end{theorem}

In the case that \ref{pathlikecase} holds, Clark et al.~\cite{CMvZW16} proved the following:

\begin{theorem}[\cite{CMvZW16,Clark15}]
  \label{nicepathdescription}
  Let $\mathbb{P} \in \{\mathbb{H}_5, \mathbb{U}_2\}$, and
  let $M$ be a $3$-connected $\{U_{2,5},U_{3,5}\}$-fragile $\mathbb{P}$-representable matroid with an $\{X_8,Y_8,Y_8^*\}$-minor.
  Then $M$ has path width three.
  Moreover, $M$ has a guts-coguts path $(P_1,P_2,\dotsc,P_m)$ such that
\begin{enumerate}
  \item for $i \in \{1,m\}$, the set $P_i$ is path generating, and is either a triangle, triad, $4$-segment, $4$-cosegment, or fan of size at least~$4$;\label{npd1}
  \item for $\{i,i'\} = \{1,m\}$, the set $P_i$ is maximal in the sense that there is no $P'$ with $P_i \subsetneqq P' \subseteq E(M) - P_{i'}$ such that $P'$ is a segment, cosegment or fan;
  \item for $i \in \{1,m\}$, if $P_i$ is not a fan of size at least~$4$, then either $P_i$ is a segment containing an element that is not $\utfutf$-deletable, or $P_i$ is a cosegment containing an element that is not $\utfutf$-contractible; and\label{nicepathdesciii}
  \item $|P_i| \le 3$ for each $i \in \{2,3,\dotsc,m-1\}$.
\end{enumerate}
\end{theorem}

\noindent
Note that the result stated here is essentially a stronger version of \cite[Lemma~4.1 and Theorem 4.2]{CMvZW16} that follows from \cite[Lemmas~2.21 and 2.22]{CMvZW16} (see also \cite[Lemma~3.3.1]{Clark15}).

We say that a guts-coguts path $(P_1,P_2,\dotsc,P_m)$ as described in \cref{nicepathdescription} is a \emph{nice path description} for $M$.

Note that a nice path description is not necessarily unique, even up to reversal.
However, a nice path description $(P_1,P_2,\dotsc,P_m)$ can be refined to a sequential ordering~$\sigma$.
By \cref{welldefinedends}, $M$ has a well-defined pair of ends $\{L(M),R(M)\} = \{L(\sigma),R(\sigma)\}$.
If both ends of $M$ are triangle or triad ends, then, by \cref{endslipperiness}\ref{esi}, $\{L(M),R(M)\} = \{P_1,P_m\}$.
If $M$ has a segment or cosegment end, $L(M)$ say, then, by \cref{endslipperiness}\ref{esii}, $P_i \subseteq L(M)$ and $|P_i| \in \{3,4\}$ for some $i \in \{1,m\}$.
In the case that $M$ has a fan end, the outcome from \cref{endslipperiness}\ref{esiii} is more complicated, partly due to the fact an $M(K_4)$ restriction has three distinct maximal $5$-element fans (see \cite[Theorem 1.6]{OW00});
however, we will see, as \cref{noMK4}, that a $3$-connected $\utfutf$-fragile matroid has no $M(K_4)$ restriction or co-restriction.

\begin{figure}[bht]
  \begin{subfigure}{0.42\textwidth}
    \centering
    \begin{tikzpicture}[rotate=180,xscale=1.1,yscale=0.8,line width=1pt]
      \tikzset{VertexStyle/.append style = {minimum height=5,minimum width=5}}
      \clip (-0.5,-4.5) rectangle (2.5,-.5);
      \draw (0,-1) -- (2,-2.5) -- (0,-4);
      \draw (0.67,-1.5) -- (0,-4);
      \draw (0.67,-3.5) -- (0,-1);


      \Vertex[L=$4$,Lpos=180,LabelOut=true,x=2,y=-2.5]{a2}
      \Vertex[L=$3$,Lpos=225,LabelOut=true,x=1.33,y=-2.0]{p1}
      \Vertex[L=$2$,Lpos=225,LabelOut=true,x=.67,y=-1.5]{a5}
      \Vertex[L=$1$,Lpos=225,LabelOut=true,x=0,y=-1]{e}
      \Vertex[L=$6$,Lpos=135,LabelOut=true,x=1.33,y=-3.0]{a4}
      \Vertex[L=$7$,Lpos=135,LabelOut=true,x=.67,y=-3.5]{p2}
      \Vertex[L=$8$,Lpos=135,LabelOut=true,x=0,y=-4]{f}

      \Vertex[L=$5$,Lpos=0,LabelOut=true,x=0.40,y=-2.5]{a3}
    \end{tikzpicture}
    \caption{$Y_8$.}
  \end{subfigure}
  \begin{subfigure}{0.42\textwidth}
    \centering
    \begin{tikzpicture}[rotate=0,yscale=0.4,xscale=0.8,line width=1pt]
      \draw (0,0) -- (4.5,0) -- (5.5,4) -- (1,4) -- (0,0);
      \draw (4.5,0) -- (5.5,-4) -- (1,-4) -- (0,0);
      \draw (1,0) -- (4,-3) -- (3,0);
      \draw (3,-2) -- (4,0);
      \draw (3,0) -- (3,3.5) -- (4,0);
      \draw (2,0) -- (3.5,2);

      \Vertex[L=$4$,Lpos=225,LabelOut=true,x=1,y=0]{a4}
      \Vertex[L=$3$,Lpos=225,LabelOut=true,x=2,y=-1]{a3}
      \Vertex[L=$2$,Lpos=225,LabelOut=true,x=3,y=-2]{a2}
      \Vertex[L=$1$,Lpos=225,LabelOut=true,x=4,y=-3]{a1}
      \Vertex[L=$5$,LabelOut=true,x=3.4,y=-1.2]{a5}
      \Vertex[L=$6$,LabelOut=true,x=3,y=3.5]{a6}
      \Vertex[L=$7$,LabelOut=true,x=3.45,y=1.9]{a7}
      \Vertex[L=$8$,Lpos=180,LabelOut=true,x=3,y=1.35]{a8}
    \end{tikzpicture}
    \caption{$X_8$.}
  \end{subfigure}
  \begin{subfigure}{0.42\textwidth}
    \centering
    \begin{tikzpicture}[rotate=0,yscale=0.4,xscale=0.8,line width=1pt]
      \draw (0,0) -- (4.5,0) -- (5.5,4) -- (1,4) -- (0,0);
      \draw (4.5,0) -- (5.5,-4) -- (1,-4) -- (0,0);
      \draw (1,0) -- (4,-3) -- (3,0);
      \draw (3,-2) -- (4,0);
      \draw (4,0) -- (2.25,3.4) -- (3,0);
      \draw (3,2) -- (1,0);

      \SetVertexNoLabel
      \Vertex[L=$3$,Lpos=225,LabelOut=true,x=2,y=-1]{a3}
      \Vertex[L=$2$,Lpos=225,LabelOut=true,x=3,y=-2]{a2}
      \Vertex[L=$1$,Lpos=225,LabelOut=true,x=4,y=-3]{a1}
      \Vertex[L=$4$,LabelOut=true,x=3.4,y=-1.2]{a5}
      \Vertex[L=$6$,Lpos=175,LabelOut=true,x=2.64,y=1.65]{b3}
      \Vertex[L=$9$,Lpos=180,LabelOut=true,x=2.25,y=3.4]{b3}
      \Vertex[L=$8$,Lpos=45,LabelOut=true,x=3,y=2]{b2}
      \Vertex[L=$7$,LabelOut=true,x=3.4,y=1.2]{b5}

      \Vertex[L=$5$,Lpos=225,LabelOut=true,x=3,y=0]{a4}
    \end{tikzpicture}
    \caption{$M_{9,9}$.}
  \end{subfigure}
  \caption{
  Geometric representations of matroids appearing in \cref{ccmwvz-result}.}
  \label{pathdescminors-fig}
\end{figure}
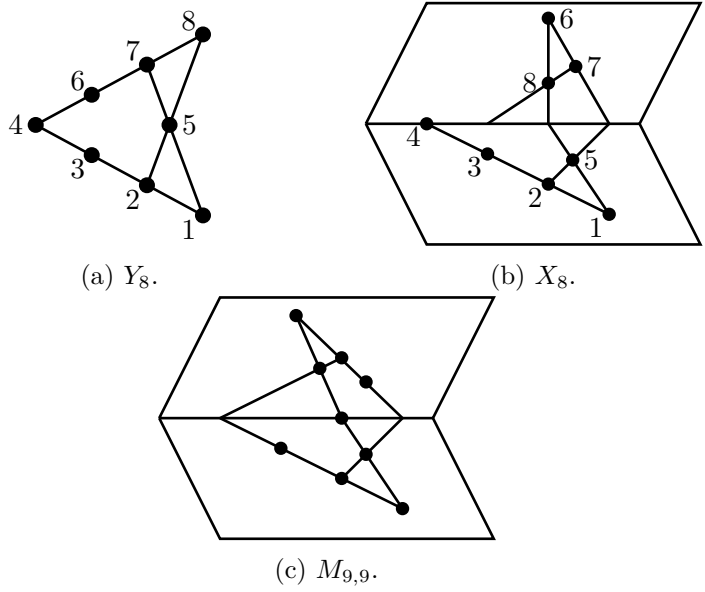

We now prove some more properties of $\utfutf$-fragile matroids with nice path descriptions.

\begin{lemma}
  \label{noessential}
  Let $M$ be a matroid with an $\{X_8,Y_8,Y_8^*\}$-minor.
  Then $M$ has no $\{U_{2,5},U_{3,5}\}$-essential elements.
\end{lemma}
\begin{proof}
  Observe that, as $M/C \ba D$ is isomorphic to a matroid in $\{X_8,Y_8,Y_8^*\}$ for some disjoint sets $C,D \subseteq E(M)$, it suffices to show that at least one of $N \ba z$ and $N / z$ has a $\utfutf$-minor for all $N \in \{X_8,Y_8,Y_8^*\}$ and all $z \in E(N)$.
  We show this for $N \in \{X_8,Y_8\}$; the result then follows by duality.

  Using the labelling given in \cref{pathdescminors-fig}, observe that $Y_8/3\ba \{y_1,y_2\} \cong U_{2,5}$ for every $2$-element subset $\{y_1,y_2\}$ of $\{1,2,4\}$.
  Since $Y_8 / 5 \ba \{7,8\} \cong U_{2,5}$, it follows, by symmetry, that for all $z \in E(Y_8)$, at least one of $Y_8 \ba z$ and $Y_8/z$ has a $\utfutf$-minor.

  Now $X_8 / \{5,7\} \ba y \cong U_{2,5}$ for $y \in \{2,6\}$ and $X_8 / \{5,8\} \ba 1 \cong U_{2,5}$.
  Also $X_8 / 3 \ba \{y_1,y_2\} \cong U_{3,5}$ for every $2$-element subset $\{y_1,y_2\}$ of $\{1,2,4\}$.
  Thus $X_8 \ba z$ or $X_8 / z$ has a $\utfutf$-minor, for all $z \in E(X_8)$.
\end{proof}

\begin{lemma}
  \label{pathdesctris}
  Let $M$ be a $3$-connected $\{U_{2,5},U_{3,5}\}$-fragile $\mathbb{P}$-representable matroid with $|E(M)| \ge 10$, for $\mathbb{P} \in \{\mathbb{H}_5, \mathbb{U}_2\}$, such that $M$ has an $\{X_8,Y_8,Y_8^*\}$-minor.
  \begin{enumerate}
    \item If $T$ is a triangle of $M$, then at least two elements of $T$ are $\{U_{2,5},U_{3,5}\}$-deletable.
    \item If $T^*$ is a triad of $M$, then at least two elements of $T^*$ are $\{U_{2,5},U_{3,5}\}$-contractible.
  \end{enumerate}
\end{lemma}
\begin{proof}
  Let $T=\{a,b,c\}$ be a triangle of $M$.
  By \cref{noessential}, $M$ has no $\utfutf$-essential elements.
  If $c \in T$ is $\utfutf$-contractible, say, then $a$ and $b$ are $\utfutf$-deletable, since $\{a,b\}$ is a parallel pair in $M/c$, and $U_{2,5}$ and $U_{3,5}$ are $3$-connected.
  Since $M$ is $\utfutf$-fragile, neither $a$ nor $b$ is $\utfutf$-contractible.
  So $T$ contains at most one $\utfutf$-contractible element.  The result follows by duality.
\end{proof}

\begin{lemma}
  \label{pathdescprops}
  Let $M$ be a $3$-connected $\{U_{2,5},U_{3,5}\}$-fragile $\mathbb{P}$-representable matroid with $|E(M)| \ge 10$, for $\mathbb{P} \in \{\mathbb{H}_5, \mathbb{U}_2\}$, such that $M$ has an $\{X_8,Y_8,Y_8^*\}$-minor and a nice path description $(P_1,P_2,\dotsc,P_m)$.
  For each $i \in \{2,3,\dotsc,m-1\}$,
  \begin{enumerate}
    \item if $P_i$ is a guts set and $e \in P_i$, then $e$ is 
      $\{U_{2,5},U_{3,5}\}$-deletable, and
      $\co(M \ba e)$ is $3$-connected; and
    \item if $P_i$ is a coguts set and $e \in P_i$, then $e$ is 
      $\{U_{2,5},U_{3,5}\}$-contractible, and
      $\si(M / e)$ is $3$-connected.
  \end{enumerate}
\end{lemma}
\begin{proof}
  For some such $i$, let $e \in P_i$.
  Suppose $P_i$ is a guts set.
  Note that if $i=2$, then $P_1$ is not a triangle or $4$-segment.
  It follows that $r(P_1 \cup \dotsm \cup P_{i-1}) \ge 3$.
  By symmetry, $r(P_{i+1} \cup \dotsm \cup P_{m}) \ge 3$.
  If $e \in P_i$ is $N$-contractible, for $N \in \{U_{2,5},U_{3,5}\}$,
  then it follows from \cref{CPL} that $M$ has an element that is $N$-flexible, a contradiction.
  So each $e \in P_i$ is not $\{U_{2,5},U_{3,5}\}$-contractible.
  By \cref{noessential}, each $e \in P_i$ is $\{U_{2,5},U_{3,5}\}$-deletable.
  Moreover, $\si(M/e)$ is not $3$-connected, by \cref{vert3sep}, so $\co(M\ba e)$ is $3$-connected by Bixby's Lemma. 

  By a dual argument, if $P_i$ is a coguts set then each $e \in P_i$ is $\{U_{2,5},U_{3,5}\}$-contractible, and $\si(M/e)$ is $3$-connected.
\end{proof}

We next consider fans appearing in $\utfutf$-fragile matroids.

\begin{lemma}[{\cite[Lemma~2.22]{CMvZW16}}]
  \label{cmvzw-fans}
  Let $\mathbb{P} \in \{\mathbb{U}_2, \mathbb{H}_5\}$, and let $M$ be a $\utfutf$-fragile $\mathbb{P}$-representable matroid.
  Let $A = \{a, b, c\}$ be a coindependent triangle of $M$ such that $b$ is not $\utfutf$-deletable.
  Let $M'$ be obtained from $M$ by gluing an $r$-wheel $W$ onto $(a, b, c)$ with remove set $X \subseteq \{a, b, c\}$ such that $b \in X$.
  If $M'$ is $3$-connected, then $M'$ is a $\utfutf$-fragile $\mathbb{P}$-representable matroid.
  Moreover, $F = E(W) - X$ is a fan.
\end{lemma}

\noindent For simplicity, when gluing a wheel $W$ with remove set $X$ as in the last lemma, we refer to $F = E(W)-X$ as the \emph{resulting fan}.

We now strengthen \cref{fragilefans} in the case that $M$ is a $\utfutf$-fragile matroid (that is, when $\mathcal{N} = \utfutf$).

\begin{lemma}
  \label{fragilefans2}
  Let $M$ be a $3$-connected $\utfutf$-fragile matroid, and let $F$ be a fan of $M$.
  \begin{enumerate}
    \item If $|F| = 4$ and $e$ is an end of $F$, then $e$ is not $\utfutf$-essential.
    \item If $|F| \ge 5$ and $e \in F$, then $e$ is not $\utfutf$-essential.
  \end{enumerate}
\end{lemma}
\begin{proof}
  If $|F| \ge 6$, or $|F|=5$ and $e$ is an end of $F$, then the result follows from \cref{fragilefans}.

  Suppose $|F| = 4$ and let $e$ be an end of $F$.
  Since $M$ is $3$-connected and has a $\utfutf$-minor, $r(M) \ge 3$ and $r^*(M) \ge 3$.
  By \cref{utfutfprop}, $M$ has a $U_{2,5}$-minor and a $U_{3,5}$-minor.
  Let $(f_1,f_2,f_3,e)$ be a fan ordering of $F$.
  Suppose $e$ is a spoke of $F$, so $\{f_2,f_3,e\}$ is a triangle.
  By \cref{fragilefanelements}, $e$ is not $U_{3,5}$-contractible.
  Suppose $e$ is not $U_{3,5}$-deletable, so $e$ is $U_{3,5}$-essential.
  Also by \cref{fragilefanelements}, $f_2$ is not $U_{3,5}$-contractible and $f_3$ is not $U_{3,5}$-deletable.
  Moreover, $f_3$ is not $U_{3,5}$-contractible, for otherwise $e$ is $U_{3,5}$-deletable; and $f_2$ is not $U_{3,5}$-deletable, for otherwise $f_3$ is $U_{3,5}$-contractible.
  So all elements in the triangle $\{f_2,f_3,e\}$ are $U_{3,5}$-essential.
  Let $C,D \subseteq E(M)$ such that $M / C \ba D \cong U_{3,5}$.
  By the foregoing, $\{f_2,f_3,e\} \cap (C \cup D) = \emptyset$.
  But $r_{M/C \ba D}(\{f_2,f_3,e\}) \le 2$, a contradiction.
  We deduce that $e$ is $U_{3,5}$-deletable, so it is $\utfutf$-deletable.
  By a dual argument, if $e$ is a rim of $F$, then e is not $U_{2,5}$-contractible, so it is $\utfutf$-contractible.
  This proves that for an end $e$ of $F$, the element~$e$ is not $\utfutf$-essential.

  Finally, suppose $F$ is a maximal fan with $|F|=5$ where $(f_1,f_2,f_3,f_4,f_5)$ is a fan ordering of $F$.
  We use a similar argument to show that $f_3$ is not $\utfutf$-essential.
  By \cref{utfutfprop}, $M$ has a $U_{2,5}$-minor and a $U_{3,5}$-minor.
  By duality, we may assume that $\{f_2,f_3,f_4\}$ is a triad.
  By \cref{fragilefanelements}, $f_3$ is not $U_{2,5}$-contractible, and $f_2$ and $f_4$ are not $U_{2,5}$-deletable.
  Suppose $f_3$ is not $U_{2,5}$-deletable.
  Then $f_2$ and $f_4$ are not $U_{2,5}$-contractible, so they are $U_{2,5}$-essential.
  Let $C,D \subseteq E(M)$ such that $M / C \ba D \cong U_{2,5}$.
  Now $r^*_{M/C \ba D}(\{f_2,f_3,f_4\}) \le 2$, a contradiction.
  This proves that $f_3$ is not $\utfutf$-essential.
\end{proof}

\begin{lemma}
  \label{noMK4}
  Let $M$ be a $3$-connected $\utfutf$-fragile matroid.
  Then there is no set $X \subseteq E(M)$ such that $M|X \cong M(K_4)$ or $M^*|X \cong M(K_4)$.
\end{lemma}
\begin{proof}
  Suppose that $M|X \cong M(K_4)$ for some $X \subseteq E(M)$.
  Then $M$ has three $5$-element fans $F_1=(f_1,f_2,f_3,f_4,f_5)$, $F_2=(g,f_2,f_4,f_3,f_5)$, and $F_3=(g,f_4,f_2,f_3,f_1)$, where $X = \{f_1,f_2,f_3,f_4,f_5,g\}$.
  By \cref{fragilefans2}, no element in $X$ is $\utfutf$-essential.
  By \cref{fragilefanelements}, each spoke of $F_1$, $F_2$, or $F_3$ is $\utfutf$-deletable, and each rim of $F_1$, $F_2$, or $F_3$ is $\utfutf$-contractible.
  But $f_4$ is a rim of $F_1$, and a spoke of $F_2$, so it is $\utfutf$-flexible, a contradiction.
\end{proof}

For a fan with at least six elements, in a $3$-connected matroid, \cref{fanends} tells us that we can retain $3$-connectivity up to series pairs or parallel pairs when a spoke is deleted or a rim is contracted, whereas we lose $3$-connectivity if a spoke is contracted or a rim is deleted.
For the middle element of a $5$-element fan, no such guarantee can be made in general.
However, we can guarantee this for $5$-element fans appearing in $\utfutf$-fragile matroids, as shown in the next lemma. 

\begin{lemma}
  \label{noMK4conn}
  Let $M$ be a $3$-connected $\utfutf$-fragile matroid, and suppose $M$ has a $5$-element fan with ordering $(f_1,f_2,\dotsc,f_5)$, where $\{f_2,f_3,f_4\}$ is a triangle. Then $\si(M/f_3)$ is $3$-connected.
\end{lemma}
\begin{proof}
  Suppose that $\si(M/f_3)$ is not $3$-connected.
  Then, by \cref{fanmiddle}, there exists some element $f_6$ such that $M^*|\{f_1,f_2,\dotsc,f_6\} \cong M(K_4)$, contradicting \cref{noMK4}.
\end{proof}

We return to $\utfutf$-fragile matroids with nice path descriptions: we next consider properties of the ends.

\begin{lemma}
  \label{pathdescendconn}
  Let $M$ be a $3$-connected $\{U_{2,5},U_{3,5}\}$-fragile $\mathbb{P}$-representable matroid with $|E(M)| \ge 10$, for $\mathbb{P} \in \{\mathbb{H}_5, \mathbb{U}_2\}$, having a nice path description $(P_1,P_2,\dotsc,P_m)$.
  Let $i \in \{1,m\}$.
  \begin{enumerate}
    \item If $P_i$ is a triangle or $4$-segment, then $\si(M / e)$ is $3$-connected for each $e \in P_i$.
    \item If $P_i$ is a triad or $4$-cosegment, then $\co(M \ba e)$ is $3$-connected for each $e \in P_i$.
  \end{enumerate}
\end{lemma}
\begin{proof}
  Suppose $P_1$ is a triad or $4$-cosegment, and $\co(M\ba e)$ is not $3$-connected for some $e \in P_1$.
  By \cref{fclnontrivialsep}, $M \ba e$ has a $2$-separation $(X,Y)$ for which $\fcl_{M \ba e}(X) \neq E(M)$ and $\fcl_{M \ba e}(Y) \neq E(M)$, and neither $X$ nor $Y$ is a series class of $M \ba e$.
  By the definition of a nice path description, $P_m$ contains a path-generating set $P_m'$ of size $3$.
  Without loss of generality, $|P_m' \cap Y| \ge 2$.
  Since $\fcl_{M \ba e}(Y) \neq E(M)$, we may assume that $Y$ is fully closed, implying $E(M)-P_1 \subseteq Y$.
  But then $X \subseteq P_1 - e$, where $P_1 - e$ is a series class in $M \ba e$, a contradiction.
\end{proof}

\begin{lemma}
  \label{pathdescends}
  Let $M$ be a $3$-connected $\utfutf$-fragile $\mathbb{P}$-representable matroid with $|E(M)| \ge 10$, for $\mathbb{P} \in \{\mathbb{H}_5, \mathbb{U}_2\}$, such that $M$ has an $\{X_8,Y_8,Y_8^*\}$-minor and a nice path description $(P_1,P_2,\dotsc,P_m)$.
  Let $i \in \{1,m\}$.
  \begin{enumerate}
    \item If $P_i$ is a segment of $M$, then $|P_i|-1$ elements of $P_i$ are $\utfutf$-deletable, and the other element is $\utfutf$-contractible.\label{pde1}
    \item If $P_i$ is a cosegment of $M$, then $|P_i|-1$ elements of $P_i$ are $\utfutf$-contractible, and the other element is $\utfutf$-deletable.\label{pde2}
    \item If $P_i$ is a fan of size at least~$4$, then each spoke of $P_i$ is $\utfutf$-deletable, and each rim of $P_i$ is $\utfutf$-contractible.\label{pde3}
  \end{enumerate}
\end{lemma}
\begin{proof}
  Suppose $P_i$ is a $k$-cosegment.
  By \cref{nicepathdescription}\ref{nicepathdesciii}, $P_i$ has one element that is $\utfutf$-deletable, so, by \cref{pathdesctris}, the other $k-1$ elements are $\utfutf$-contractible as required.
  A similar argument applies when $P_i$ is a segment.
  When $P_i$ is a fan, the result follows from \cref{fragilefanelements,noessential}.
\end{proof}

The next property of $\utfutf$-fragile matroids with nice path descriptions builds on earlier results of the section.

\begin{lemma}
  \label{pathdescrank}
  Let $M$ be a $3$-connected $\{U_{2,5},U_{3,5}\}$-fragile $\mathbb{P}$-representable matroid with $|E(M)| \ge 10$, for $\mathbb{P} \in \{\mathbb{H}_5, \mathbb{U}_2\}$, 
  with an $\{X_8,Y_8,Y_8^*\}$-minor.
  Let $C$ be the set of $\{U_{2,5},U_{3,5}\}$-contractible elements and let $D$ be the set of $\{U_{2,5},U_{3,5}\}$-deletable elements of $M$.  Then $|C| = r(M)$ and $|D| = r^*(M)$.
\end{lemma}
\begin{proof}
  By \cref{nicepathdescription}, $M$ has a nice path description $(P_1,P_2,\dotsc,P_m)$.
  Since this is a path of $3$-separations, it is easily seen that if $M$ has $h$ coguts elements in $P_2 \cup \dotsm \cup P_{m-1}$, then $r(M) = r(P_1) +h+ r(P_m)-2$.  Each of the $h$ coguts elements are $\utfutf$-contractible, by \cref{pathdescprops}, whereas each guts element is $\utfutf$-deletable.

  It remains to show that an end, $P_1$ say, has exactly $r(P_1)-1$ elements that are $\utfutf$-contractible.
  If $P_1$ is a $k$-cosegment, then $k \in \{3,4\}$ and $r(P_1)=k$, and this follows from \cref{pathdescends}\ref{pde2}.
  On the other hand, if $P_1$ is a segment, then $r(P_1)=2$ and $P_1$ has exactly one $\utfutf$-contractible element, by \cref{pathdescends}\ref{pde1}.
  If $P_1$ is a fan of size at least~$5$ having $t$ rim elements, then $r(P_1) = t+1$.
  By \cref{pathdescends}\ref{pde3}, each rim element is $\utfutf$-contractible and each spoke element is $\utfutf$-deletable.
  It is also easily checked that if $P_1$ is a $4$-element fan then $r(P_1)=3$ and $P_1$ has exactly two $\utfutf$-contractible elements.

  We deduce that there are exactly $r(M)=r(P_1)+h+r(P_m)-2$ elements that are $\utfutf$-contractible.  Since $M$ has no $\utfutf$-essential elements, by \cref{noessential}, the result follows.
\end{proof}

We also require the following \lcnamecref{keepfragilelabels} that ensures elements can be removed while retaining a nice path description.

\begin{lemma}
  \label{keepfragilelabels}
  Let $M$ be a $3$-connected $\utfutf$-fragile $\mathbb{P}$-representable matroid with $|E(M)| \ge 10$, for $\mathbb{P} \in \{\mathbb{H}_5, \mathbb{U}_2\}$, such that $M$ has an $\{X_8,Y_8,Y_8^*\}$-minor and a nice path description $(P_1,P_2,\dotsc,P_m)$.
  \begin{enumerate}
    \item If $a \in P_1$ and $b \in P_m$ are $\utfutf$-deletable, then $M \ba a,b$ is $\utfutf$-fragile and has no $\utfutf$-essential elements.\label{kfl1}
    \item If $e \in P_1 \cup P_m$ is $\utfutf$-contractible, then $M / e$ is $\utfutf$-fragile and has no $\utfutf$-essential elements.\label{kfl2}
  \end{enumerate}
\end{lemma}
\begin{proof}
  We sketch the proof only.  Consider \ref{kfl1}.
  Using the terminology of \cite{CMvZW16}, $M$ has a path sequence from which we can obtain a path sequence for $M \ba a,b$, since $a$ and $b$ are at the ends.
  This latter path sequence certifies that $M \ba a,b$ is $\utfutf$-fragile, and that $M \ba a,b$ has an $\{X_8,Y_8,Y_8^*,M_{8,6}\}$-minor, by \cite[Lemma~6.1]{CMvZW16}.
  By \cref{noessential}, if $M\ba a,b$ has an $\{X_8,Y_8,Y_8^*\}$-minor, then $M \ba a,b$ has no $\utfutf$-essential elements.
  Using a similar approach as in the proof of \cref{noessential}, it is easily checked that $M \ba a,b$ has no $\utfutf$-essential elements when $M \ba a,b$ has only an $M_{8,6}$-minor.
  A similar argument also applies for \ref{kfl2}.
\end{proof}

Note that the previous \lcnamecref{keepfragilelabels} implies that after deleting the pair $\{a,b\}$ or contracting $e$, each element in the resulting matroid is $\utfutf$-deletable (or $\utfutf$-contractible) if and only if it was
$\utfutf$-deletable (or $\utfutf$-contractible) in $M$; otherwise, $M$ would have $\utfutf$-flexible elements.

\section{Excluded minors are almost \texorpdfstring{$\utfutf$}{\{U(2,5),U(3,5)\}}-fragile}
\label{utfutffragilesec}

Suppose that $M$ is an excluded minor for the class of $\mathbb{P}$-representable matroids, where $\mathbb{P} \in \{\mathbb{H}_5,\mathbb{U}_2\}$, and $M \ba a,b$ is a $3$-connected 
matroid with a $\utfutf$-minor, for some distinct $a,b \in E(M)$.
In this section, we show that if $|E(M)| \ge 16$, then $M \ba a,b$ is $\utfutf$-fragile.

\begin{lemma}
  \label{noessential2}
  For $\mathbb{P} \in \{\mathbb{H}_5,\mathbb{U}_2\}$ and $N \in \utfutf$,
  let $M$ be a $3$-connected $\utfutf$-fragile $\mathbb{P}$-representable matroid.
  If $M$ has an $\{X_8,Y_8,Y_8^*\}$-minor and $|E(M)| \ge 9$, 
  then $M$ has no $N$-essential elements.
\end{lemma}
\begin{proof}
  Note that for any $M$ satisfying the hypotheses of the lemma, $M$ has a $\utfutf$-minor, so $M$ is not a wheel or a whirl.
  Towards a contradiction, suppose that $M$ has an $N$-essential element.
  If $|E(M)| > 9$, then, by Seymour's Splitter Theorem, $M$ has a $3$-connected $\mathbb{P}$-representable $\utfutf$-fragile minor $M'$, with $|E(M')| = |E(M)|-1$, such that $M'$ has an $\{X_8,Y_8,Y_8^*\}$-minor.  Note that $M'$ also has an $N$-essential element.
  It follows that there exists a $3$-connected $\mathbb{P}$-representable $\utfutf$-fragile matroid $M''$, having an $\{X_8,Y_8,Y_8^*\}$-minor, such that $|E(M'')|=9$ and $M''$ has an $N$-essential element.
  The $3$-connected $\mathbb{P}$-representable $\utfutf$-fragile matroids on nine elements are given in \cite[Figure~8]{CCCMWvZ13}.
  It can be readily checked that for each such matroid having an $\{X_8,Y_8,Y_8^*\}$-minor, the matroid has no $N$-essential elements.
  So no such matroid $M''$ exists, a contradiction.
\end{proof}

\begin{lemma}
  \label{utfutfessentialrank4}
  For $\mathbb{P} \in \{\mathbb{H}_5,\mathbb{U}_2\}$ and $N \in \utfutf$,
  let $M$ be a $3$-connected $\utfutf$-fragile $\mathbb{P}$-representable matroid with 
  an $N$-essential
  element and rank at most four.  Then $|E(M)| \le 9$.
\end{lemma}
\begin{proof}
  By \cref{noessential2}, either $|E(M)| \le 8$ or $M$ has no $\{X_8,Y_8,Y_8^*\}$-minor.  So we may assume $M$ has no $\{X_8,Y_8,Y_8^*\}$-minor.
  By \cref{ccmwvz-result}, we may also assume that $M$ or $M^*$ can be obtained by gluing wheels to $U_{2,5}$ or $Y_8 \ba 4$.
  In this case, the fact that $r(M) \le 4$ forces $|E(M)| \le 9$; we omit the details.
\end{proof}

\begin{lemma}
  \label{hopeful}
  For $\mathbb{P} \in \{\mathbb{H}_5,\mathbb{U}_2\}$ and $N \in \utfutf$,
  let $M$ be a $3$-connected $\utfutf$-fragile $\mathbb{P}$-representable matroid with at least three $N$-essential elements.
  Then either $|E(M)| \le 7$ or $M \in \{Y_8,Y_8^*\}$.
\end{lemma}

\begin{proof}
  First, assume that $M$ has no $\{X_8,Y_8,Y_8^*\}$-minor.
  Then, by \cref{ccmwvz-result}, either $M \in \{M_{9,9},M_{9,9}^*\}$, $M$ or $M^*$ can be obtained by gluing wheels to $U_{2,5}$ or $Y_8 \ba 4$, or $|E(M)| \le 7$.
  It is easily checked that $M_{9,9}$ has no $U_{2,5}$- or $U_{3,5}$-essential elements, so the former is not possible.

  Assume that $M$ or $M^*$ can be obtained by gluing wheels to $U_{2,5}$ or $Y_8 \ba 4$.
  We claim that $|E(M)| \le 7$ in this case.
  Suppose not; then there exists some minor-minimal matroid $M'$ such that $M'$ can be obtained by gluing wheels to $U_{2,5}$ or $Y_8 \ba 4$, $|E(M')| \ge 8$, and $M'$ is a $3$-connected $\utfutf$-fragile $\mathbb{P}$-representable matroid with at least three $N$-essential elements.
  First, we observe that if $|E(M')|=8$, then $M'$ or $(M')^*$ is isomorphic to one of the four matroids referred to in \cite{CCCMWvZ13,CMvZW16} as $M_{8,1}$, $M_{8,3}$, $M_{8,5}$, and $M_{8,6}$ (in particular, $M' \cong M_{8,6}$ in the case that a wheel was glued to $Y_8 \ba 4$), but these matroids have no $U_{2,5}$- or $U_{3,5}$-essential elements.
  So $|E(M')| \ge 9$.

  Suppose that $M'$ has a maximal fan~$F$ of length at least~$4$.
  By contracting a rim end, or deleting a spoke end, of $F$, we obtain a $3$-connected minor~$M''$ of $M'$, by \cref{fanendsstrong}, with $|E(M'')| \ge 8$. 
  By \cref{fragilefanelements,fragilefans2}, $M''$ has a $\utfutf$-minor, so this matroid is still $\utfutf$-fragile, and $M''$ has at least as many $N$-essential elements as $M'$.
  But this shows $M'$ is not minor-minimal, a contradiction.
  So $M'$ has no fans of length at least~$4$.
  It follows that $M'$ can be obtained from $U_{2,5}$ by gluing three wheels so that the resulting fans each have length three: this matroid is referred to as $M_{9,18}$ in \cite{CCCMWvZ13,CMvZW16}.
  But it is easily checked that $M_{9,18}$ has no $U_{2,5}$- or $U_{3,5}$-essential elements.
  We deduce that $|E(M)| \le 7$ in the case that $M$ or $M^*$ can be obtained by gluing wheels to $U_{2,5}$ or $Y_8 \ba 4$.

  We may now assume that $M$ has an $\{X_8,Y_8,Y_8^*\}$-minor.
  Then $M \in \{X_8,Y_8,Y_8^*\}$, for otherwise $M$ has no $N$-essential elements, by \cref{noessential2}.
  It is readily checked that $X_8$ has exactly one $U_{2,5}$-essential element, and exactly one $U_{3,5}$-essential element, whereas $Y_8$ has three $U_{2,5}$-essential elements.  So $M \in \{Y_8,Y_8^*\}$.
\end{proof}

\begin{lemma}
  \label{utfutf3essentialv2}
  For $\mathbb{P} \in \{\mathbb{H}_5,\mathbb{U}_2\}$ and $N \in \utfutf$,
  let $M$ be a $3$-connected $\utfutf$-fragile $\mathbb{P}$-representable matroid with two $N$-essential elements.
  Then either $|E(M)| \le 8$, or $M$ or $M^*$ can be obtained from $U_{2,5}$ by gluing a single wheel such that the resulting fan has at least five elements.
\end{lemma}
\begin{proof}
  By \cref{noessential2}, either $|E(M)| \le 8$ or $M$ has no $\{X_8,Y_8,Y_8^*\}$-minor; so we may assume that $M$ has no $\{X_8,Y_8,Y_8^*\}$-minor.
  We first apply \cref{ccmwvz-result}, deducing that
  either $M \in \{M_{9,9},M_{9,9}^*\}$, $M$ or $M^*$ can be obtained by gluing wheels to $U_{2,5}$ or $Y_8 \ba 4$, or $|E(M)| \le 8$.
  It is easy to check that if $M \in \{M_{9,9},M_{9,9}^*\}$, then $M$ has no $N$-essential elements, so the former is not possible.

  Assume that $M$ or $M^*$ can be obtained by gluing a wheel to $Y_8 \ba 4$, or gluing at least two wheels to $U_{2,5}$.
  We claim that $|E(M)| \le 8$.
  Suppose not; then there exists some minor-minimal matroid $M'$ such that $M'$ can be obtained by gluing a wheel to $Y_8 \ba 4$, or gluing at least two wheels to $U_{2,5}$; $|E(M')| \ge 9$; and $M'$ is a $3$-connected $\utfutf$-fragile $\mathbb{P}$-representable matroid with at least two $N$-essential elements.
  First, we observe that if $|E(M')|=9$, then $M'$ or $(M')^*$ is isomorphic to one of the matroids referred to in \cite{CCCMWvZ13,CMvZW16} as $M_{9,1}$, $M_{9,2}$, $M_{9,7}$, $M_{9,15}$, and $M_{9,18}$ (in particular, $M' \cong M_{9,7}$ in the case that a wheel was glued to $Y_8 \ba 4$).
  But these matroids have at most one $N$-essential element (if such an element exists, it is $\utfutf$-essential).
  So $|E(M')| \ge 10$.

  Suppose that $M'$ has a maximal fan~$F$ of length at least~$4$.
  By contracting a rim end, or deleting a spoke end, of $F$, we obtain a $3$-connected minor~$M''$ of $M'$, by \cref{fanendsstrong}, with $|E(M'')| \ge 9$.
  By \cref{fragilefanelements,fragilefans2}, $M''$ has a $\utfutf$-minor, so this matroid is still $\utfutf$-fragile, and $M''$ has at least as many $N$-essential elements as $M'$.
  But this shows $M'$ is not minor-minimal, a contradiction.
  So $M'$ has no fans of length at least~$4$.
  But then $|E(M')| \le 9$, a contradiction.
  We deduce that $|E(M)| \le 8$ in the case that $M$ or $M^*$ can be obtained by gluing a wheel to $Y_8 \ba 4$, or gluing at least two wheels to $U_{2,5}$.

  Now, the only remaining possibility, when $|E(M)| \ge 9$, is that $M$ or $M^*$ can be obtained by gluing a single wheel to $U_{2,5}$, as required.
\end{proof}

Next we work towards \cref{jail}, which describes some properties of matroids that are $N$-fragile, for $N \in \utfutf$, but not $\utfutf$-fragile.

We require a definition.
Let $M$ be a $3$-connected matroid having the $3$-connected matroid $N$ as a minor.
For $e \in E(M)$, we say $e$ is \emph{$N$-elastic} if $e$ is $N$-flexible and both $\si(M/e)$ and $\co(M \ba e)$ are $3$-connected.

To prove \cref{jail}, we consider two cases: first, when $M$ has a $U_{2,5}$-elastic element, in \cref{jail-elastic}; and then when $M$ has no $U_{2,5}$-elastic elements, in \cref{jail-nonelastic}.

\begin{lemma}
  \label{jail-elastic}
  For $\mathbb{P} \in \{\mathbb{H}_5,\mathbb{U}_2\}$, 
  let $M$ be a $3$-connected $\mathbb{P}$-representable $U_{3,5}$-fragile matroid that is not $\utfutf$-fragile, where $r(M) \ge 4$ and $r^*(M) \ge 4$.
  If $M$ has a $U_{2,5}$-elastic element,
  then $|E(M)| \le 9$, and $M$ has at most two $U_{3,5}$-essential elements.
\end{lemma}
\begin{proof}
  Let $e$ be a $U_{2,5}$-elastic element of $M$, so $e$ is $U_{2,5}$-flexible, and both $\si(M/e)$ and $\co(M \ba e)$ are $3$-connected.
\begin{claim}
  \label{coranktwo}
  $r(\co(M \ba e))=2$.
\end{claim}
\begin{subproof}
  If $r(\co(M \ba e)) \ge 3$, then, as $\co(M \ba e)$ has a $U_{2,5}$-minor, \cref{utfutfprop} implies it also has a $U_{3,5}$-minor.
  Moreover, $\si(M / e)$ has a $U_{2,5}$-minor, and this matroid has rank at least~$3$ (since $r(M) \ge 4$) and corank at least~$3$ (since it has a $U_{2,5}$-minor).  So, by \cref{utfutfprop}, $M/e$ has a $U_{3,5}$-minor.
  But then $e$ is $U_{3,5}$-flexible, a contradiction.
\end{subproof}

  Now, by \cref{coranktwo}, and since $M \ba e$ has a $U_{2,5}$-minor, $\co(M \ba e) \cong U_{2,t}$ for some $t \ge 5$.
  Therefore, the union of three series classes of $M \ba e$ is a circuit.
  Since $M$ is $\mathbb{P}$-representable, $t \le 6$.

  We work towards showing, in \cref{simpcase,cosimpcase}, that when an element $f$ is in a non-trivial series class of $M \ba e$, then, except in some particular situations, both $\si(M/f)$ and $\co(M \ba f)$ are $3$-connected and have rank and corank at least~$3$.

  \begin{claim}
    \label{presimpcase1}
  Let $f \in E(M\ba e)$ where $f$ is in a non-trivial series class~$S$ of $M \ba e$ with $|S| \ge 3$.  Then $\si(M/f)$ is $3$-connected.
  \end{claim}
  \begin{subproof}
    Suppose $\si(M/f)$ is not $3$-connected.
    Then $M$ has a vertical $3$-separation $(X,f,Y)$.
    In particular, $f \notin \cocl(X)$ and $f \notin \cocl(Y)$.
    Let $f'$ and $f''$ be distinct elements in $S-f$.
    Since $M$ is $3$-connected, $\{e,f',f''\}$ is a triad of $M$.
    We may assume that at least two elements of this triad are contained in $X$, say.
    But then $f \in \cocl(X)$, a contradiction.
  \end{subproof}

  \begin{claim}
    \label{presimpcase2}
    Suppose $M \ba e$ has at least two non-trivial series classes, and $\{f,f'\}$ is a series class of $M \ba e$.
    If $\si(M/f)$ is not $3$-connected, then $M$ has a vertical $3$-separation $(X,f,Y)$ such that $e \in X$ and $f' \in Y$, and either
    \begin{enumerate}[label=\rm(\Roman*)]
      \item every non-trivial series class of $M \ba e$ distinct from $\{f,f'\}$ is contained in $X$, or\label{psc2i}
      \item $M \ba e$ has precisely two non-trivial series classes, $\{f,f'\}$ and $G$, and $X = (G-g) \cup e$ for some $g \in G$.\label{psc2ii}
    \end{enumerate}
  \end{claim}
  \begin{subproof}
    Suppose $\si(M/f)$ is not $3$-connected, so $M$ has a vertical $3$-separation $(X,f,Y)$.
    Since $f \notin \cocl(X)$ and $f \notin \cocl(Y)$, we may assume that $e \in X$ and $f' \in Y$. 
    Let $G$ be a non-trivial series class of $M \ba e$ distinct from $\{f,f'\}$, with distinct elements $g,g' \in G$.
    If $\{g,g'\} \subseteq Y$, then $e \in \cocl(Y)$, so $f \in \cocl(Y)$, a contradiction.

    Now suppose $g \in Y$ and $G-g \subseteq X$.
    It suffices to prove that when \ref{psc2ii} does not hold, then $(X \cup g, f, Y-g)$ is a vertical $3$-separation.
    Observe that $g \in \cocl(X)-X$ and $f \in \cl(X)-X$.
    By \cref{gutsandcoguts}, $\cl(X) = X \cup f$ and $\cocl(X) = X \cup g$.
    Hence, $G$ and $\{f,f'\}$ are the only non-trivial series classes of $M \ba e$ not contained in $X$.
    Recall that the union of three series classes of $M \ba e$ is a circuit.
    If $X$ contains two series classes of $M \ba e$, then, as $G-g \subseteq X$, we have $g \in \cl(X)$, a contradiction.
    In particular, if $M \ba e$ has at least four non-trivial series classes, then two are contained in $X$, a contradiction.
    So $M \ba e$ has at most three non-trivial series classes, 
    and $Y$ contains at least two trivial series classes.
    In particular, $|Y| \ge 4$.

    If $X$ does not contain any series class of $M \ba e$, then $X = (G-g) \cup e$ and \ref{psc2ii} holds.
    So assume $X$ contains a series class $S$ of $M \ba e$.
    Let $C^*$ be a cocircuit contained in $Y$, and suppose $g \in C^*$.
    If $f' \notin C^*$, then the circuit $S \cup G \cup \{f,f'\}$ intersects $C^*$ in a single element $g$, contradicting orthogonality.
    Similarly, if $C^*$ avoids some element $y \in Y-\{f',g\}$, this violates orthogonality with the circuit $S \cup G \cup y$.
    So $C^*=Y$.
    Now $r(Y)=4$ and $r^*(Y)=|Y|-1$, so $\lambda(Y) = 3$, a contradiction.
    Thus $g \notin C^*$.
    It now follows that $(X \cup g, f, Y-g)$ is a vertical $3$-separation, as required.
  \end{subproof}

  \begin{claim}
    \label{simpcase}
    Let $f \in E(M\ba e)$ where $f$ is in a non-trivial series class~$S$ of $M \ba e$.
    Suppose either $|S| \ge 3$, or $M \ba e$ has at least three non-trivial series classes.
    Then $\si(M/f)$ is $3$-connected and has rank and corank at least~$3$.
  \end{claim}
  \begin{subproof}
    Suppose $\si(M/f)$ is not $3$-connected.
    By \cref{presimpcase1}, $|S| = 2$, so $M \ba e$ has at least three non-trivial series classes.
    Let $S = \{f,f'\}$.
    By \cref{presimpcase2}, $M$ has a vertical $3$-separation $(X,f,Y)$ such that $f' \in Y$, and $X$ contains $e$ and every non-trivial series class of $M \ba e$ distinct from $S$. 
    Note, in particular, that there are at least two non-trivial series classes contained in $X$.
    The set $Y$ consists of $f'$ and a subset of the elements in trivial series classes of $M \ba e$, with $|Y| \ge 3$.
    Hence $X \cup f$ spans $E(M)$, contradicting that $(X,f,Y)$ is a vertical $3$-separation.
    We deduce that $\si(M/f)$ is $3$-connected.

    Clearly 
    $r(\si(M/f)) \ge 3$, since $r(M) \ge 4$.
    We claim that $r^*(\si(M/f)) \ge 3$.
    Suppose $f$ is in a triangle~$T$ of $M$.
    If $T$ is also a triangle of $M \ba e$, then $T$ contains $S$ by orthogonality, implying $\co(M \ba e)$ is not $3$-connected, a contradiction.
    So $e \in T$.
    But then, for every series pair $\{g,g'\}$ of $M \ba e$, the set $\{e,g,g'\}$ is a triad that meets $T$.
    By orthogonality, $T$ contains an element of each series pair of $M \ba e$.
    In the case that $M \ba e$ has at least three non-trivial series classes, we deduce that $f$ is not in a triangle, so $r^*(M/f) \ge 4$.
    Otherwise, when $|S| \ge 3$, it follows that $f$ is in at most one triangle, so $r^*(\si(M/f)) \ge r^*(M/f) - 1 \ge 3$.
  \end{subproof}

  Recall that $\co(M \ba e) \cong U_{2,t}$ for $t \in \{5,6\}$.
  \begin{claim}
    \label{precosimpcase}
    Let $f \in E(M \ba e)$ where $f$ is in a non-trivial series class~$S$ of $M \ba e$.
    Suppose $\co(M \ba f)$ is not $3$-connected. 
    Then $M$ has a cyclic $3$-separation $(X,f,Y)$ such that
    \begin{enumerate}[label=\rm(\Roman*)]
      \item $e \in X$,
      \item $X \cup f$ is coclosed, 
      \item $Y$ is the union of at least $t-2$ trivial series classes of $M \ba e$, and\label{pcsc3}
      \item there is a circuit~$C$ such that $(S-f) \cup e \subseteq C \subseteq X$.\label{pcsc4}
    \end{enumerate}
  \end{claim}
  \begin{subproof}
    Clearly $M$ has a cyclic $3$-separation $(X,f,Y)$ where, without loss of generality, $e \in X$ and $X \cup f$ is coclosed.  It remains to show \ref{pcsc3} and \ref{pcsc4} hold.
  First, observe that as $e \in X$ and $X \cup f$ is coclosed, for a series class $S'$ of $M \ba e$ distinct from $S$, either $S' \subseteq X$ or $S' \subseteq Y$.
  Similarly, either $S -f \subseteq X$ or $S -f \subseteq Y$.
  
  The set $Y$ contains a circuit~$C_Y$.
  Since $C_Y$ is also a circuit of $M \ba e$, if $C_Y$ meets a series class~$S'$ of $M \ba e$, then, by orthogonality, $C_Y$ contains $S'$.
  Recall that the union of three series classes of $M \ba e$ is a circuit.
  It follows that $C_Y$ is the union of three series classes of $M \ba e$; in particular, $C_Y$ avoids $S$.
  If $S -f \subseteq Y$, then $f \in \cl(Y)$ so $f \notin \cocl(X)$, a contradiction.
  So $S -f \subseteq X$.
  If there are series classes $S'$ and $S''$ of $M \ba e$ contained in $X$, such that $S$, $S'$ and $S''$ are distinct, then $f \in \cl(X)$ and thus $f \notin \cocl(Y)$, a contradiction.
  So either $X-e = S -f$, or $X-e = S' \cup (S - f)$ for some series class $S'$ of $M \ba e$.

  There is also a circuit~$C_X$ contained in $X$.
  If $e \notin C_X$, then $C_X$ is properly contained in the union of three series classes of $M \ba e$, a contradiction.
  So $e \in C_X$. 
  Then $e \notin \cocl(Y)$, but $e$ blocks each non-trivial series class of $M \ba e$.
  Thus $Y$ is the union of trivial serial classes of $M \ba e$, of which there are at least $t-2$, so \ref{pcsc3} holds.
  Moreover, $S-f \subseteq C_X$ by orthogonality. So \ref{pcsc4} holds, thus proving the claim.
  \end{subproof}

  \begin{claim}
    \label{cosimpcase}
    Let $S$ be a non-trivial series class of $M \ba e$.
    \begin{enumerate}[label=\rm(\Roman*)]
      \item If $M \ba e$ has at least three non-trivial series classes, then, for each $f \in S$, the matroid $\co(M\ba f)$ is $3$-connected.\label{csc1}
      \item If $M \ba e$ has precisely two non-trivial series classes, then there exists some $S' \subseteq S$ with $|S'| \ge |S|-1$ such that, for each $f \in S'$, the matroid $\co(M\ba f)$ is $3$-connected.\label{csc2}
    \end{enumerate}
    Moreover, if $\co(M \ba f)$ is $3$-connected for some $f \in S$, then $r(\co(M \ba f)) \ge 3$ and $r^*(\co(M \ba f)) \ge 3$.
  \end{claim}
  \begin{subproof}
  Let $f \in S$.
  Suppose $\co(M \ba f)$ is not $3$-connected. 
  Then $M$ has a cyclic $3$-separation $(X,f,Y)$ as described in \cref{precosimpcase}.
  In particular, \cref{precosimpcase}\ref{pcsc3} implies that when $M \ba e$ has at least three non-trivial series classes, $\co(M \ba f)$ is $3$-connected for each $f \in S$, proving \ref{csc1}.

  For \ref{csc2}, assume $M \ba e$ has precisely two non-trivial series classes.
  By the foregoing, $Y$ is the union of the $t-2$ trivial series classes of $M \ba e$, and there is a circuit~$C$ such that $(S-f) \cup e \subseteq C \subseteq X$.
  Let $f' \in S-f$ and suppose $\co(M \ba f')$ is not $3$-connected.
  Then, by another application of \cref{precosimpcase}, $M$ has a cyclic $3$-separation $(X', f', Y')$ where $Y'=Y$. 
  Since $Y' = Y$, we have $C \subseteq X \subseteq X' \cup f'$.
  But $f' \in C$, so $f' \in \cl(X')$, and hence $f' \notin \cocl(Y')$, a contradiction.
  This proves that for $S' = S-f$, each $f' \in S'$ has the property that $\co(M \ba f')$ is $3$-connected.
  
  Henceforth, let $f \in S$ such that $\co(M \ba f)$ is $3$-connected, and suppose $M \ba e$ has at least two non-trivial series classes.
  Clearly $r^*(\co(M \ba f)) \ge 3$, since $r^*(M) \ge 4$.
  It remains to show that $r(\co(M \ba f)) \ge 3$.
  Let $(S_1,S_2,\dotsc,S_t)$ be a partition of $E(M \ba e)$ into series classes, with $f \in S_1$.
  Recall that $t \ge 5$ and, for distinct $i,j,k \in [t]$, the set $S_i \cup S_j \cup S_k$ is a circuit.
  Suppose $f$ is in a triad $T^*$ of $M$ that is not contained in $S_1 \cup e$.
  Without loss of generality, $S_2 \cap T^* \neq \emptyset$.
  Let $h \in \{1,2\}$ and $i \in \{3,4,5\}$, so $C_{h,i} = (S_h \cup S_3 \cup S_4 \cup S_5) - S_i$ is a circuit.
  By orthogonality with $C_{2,i}$ for each such $i$, we have $|S_2 \cap T^*| = 2$.
  Similarly, orthogonality with $C_{1,i}$ for each $i$ implies that $|S_1 \cap T^*| = 2$, a contradiction.
  So the only triads of $M$ containing $f$ are contained in $S_1 \cup e$.
  Since $M \ba e$ has at least one non-trivial series classes other than $S_1$, we have $r(\co(M \ba f)) \ge r(\co(M \ba e)) + 1 = 3$.
  \end{subproof}

  Recall that $\co(M \ba e) \cong U_{2,t}$ for $t \in \{5,6\}$.
  Assume first that $M \ba e$ has precisely one non-trivial series class $S$.
  Then $S \cup e$ is a cosegment of $M$.
  Let $G = S \cup e$ and $L = E(M) - G$.
  Note that $M / L \cong U_{|G|-2,|G|}$ and $M|L \cong U_{2,|L|}$; moreover, $\co(M \ba e) \cong U_{2,|L|+1}$, so $|L| \in \{4,5\}$.
  Suppose that $|G| \ge 5$.
  Then each element $\ell \in L$ is $U_{3,5}$-contractible.
  Moreover, $L-\ell$ is a non-trivial parallel class in $M / \ell$, implying each $\ell \in L$ is $U_{3,5}$-deletable.
  So each $\ell \in L$ is $U_{3,5}$-flexible, a contradiction.
  Hence $|G| \le 4$.
  Now $|E(M)| = |L| + |G| \le 9$.
  Suppose $|E(M)|=9$. Then $|L|=5$ and $|G|=4$, in which case $r(M)=4$ and $r^*(M) = 5$.
  Each $s \in S$ is $U_{2,5}$-contractible so, by \cref{simpcase,utfutfprop}, $s$ is also $U_{3,5}$-contractible; and each $\ell \in L$ is $U_{2,5}$-deletable so, by \cref{lineconn}, $M \ba \ell$ is $3$-connected, and hence $\ell$ is $U_{3,5}$-deletable.
  So $M$ has at most one $U_{3,5}$-essential element, and thus the \lcnamecref{jail-elastic} holds in this case.

  Now suppose $|E(M)|=8$.  Then, since $r^*(M) \ge 4$, we have $|L|=|G|=4$.
  Consider $M/e$.  
  Since $M/e$ has a $U_{2,5}$-minor, there exist distinct $s,s' \in S$ such that $\{s,s',\ell\}$ is independent for each $\ell \in L$.
  Thus, $\{s,s',e,\ell\}$ is independent in $M$ for each $\ell \in L$.
  Let $\{s''\} = S-\{s,s'\}$ and note that $S \cup \ell$ is independent in $M$ for each $\ell \in L$.
  Choose $\ell' \in L$ so that $\{s,s'',e,\ell'\}$ is a circuit, or if no such circuit exists, then choose $\ell' \in L-\ell$ arbitrarily.
%
  Then $M /s \ba \ell' \cong P_6$ where $L-\ell'$ is the unique triangle in this matroid.
  It follows that each $\ell \in L$ is $U_{3,5}$-deletable.
  By orthogonality, each circuit~$C$ that meets $G$ has $|C \cap G| \ge 3$.
  Thus, by the foregoing, there exists $\ell \in L$ such that $G' \cup \ell$ is independent for all $G' \subseteq G$ with $|G'|=3$.
  Then $\si(M/\ell) \cong U_{3,5}$.
  So $\ell$ is $U_{3,5}$-flexible, a contradiction.

  Assume next that $M \ba e$ has precisely two non-trivial series classes, each of size two.
  Then $|E(M)| = t + 3 \le 9$.
  Let $S_1$ and $S_2$ be the two series pairs of $M \ba e$. 
  Suppose $|E(M)|=9$, so 
  $r(M) = 4$ and $r^*(M) = 5$.
  Since $\co(M \ba e) \cong U_{2,6}$, each element in $E(M)-e$ is $U_{2,5}$-deletable.
  By \cref{cosimpcase,utfutfprop}, for each $i \in \{1,2\}$ there exists $s_i \in S_i$ such that $M \ba s_i$ has a $U_{3,5}$-minor.
  Now $M / s_1$, say, also has a $U_{2,5}$-minor.
  If $s_1$ is in a triangle of $M$, then, by orthogonality, this triangle is contained in $S_1 \cup S_2 \cup e$, in which case $S_1 \cup S_2 \cup e$ is a $5$-element fan. It follows that $s_1$ is in at most one triangle, so 
  $\si(M/s_1)$ has rank and corank at least~$3$.
  If $\si(M/s_1)$ is $3$-connected, then, by \cref{utfutfprop}, $M / s_1$ has a $U_{3,5}$-minor, so $s_1$ is $U_{3,5}$-flexible, a contradiction.
  So $\si(M / s_1)$ is not $3$-connected.
  Then, by \cref{presimpcase2},  $(S_2 \cup e, s_1, Y)$ is a vertical $3$-separation.
  Similarly, $(S_1 \cup e, s_2, Y)$ is a vertical $3$-separation.
  So there is a circuit contained in $S_1 \cup \{e,s_2\}$, and a circuit contained in $S_2 \cup \{e,s_1\}$.  If these circuits are distinct, then, by circuit elimination, there is a circuit contained in $S_1 \cup S_2$, a contradiction.
  So $\{s_1,e,s_2\}$ is a triangle and $S_1 \cup S_2 \cup e$ is a $5$-element fan.
  Now, it is easily verified that $M$ has no $U_{3,5}$-essential elements, and thus the \lcnamecref{jail-elastic} holds in this case.

  Now suppose $|E(M)| = 8$. 
  Let $E(M)-(S_1 \cup S_2 \cup e) = \{s_3,s_4,s_5\}$.
  Each $s \in S_1 \cup S_2$ is $U_{2,5}$-contractible.
  First, suppose there is a circuit $\{s_1,e,s_2\}$, where $S_i = \{s_i,s_i'\}$ for $i \in \{1,2\}$.
  Then $S_1 \cup S_2 \cup e$ is a $5$-element fan, and, by \cref{fragilefanelements,fragilefans}, $s_i'$ is $U_{3,5}$-contractible, for $i \in \{1,2\}$.
  By \cref{fanmiddle}, $\si(M/e)$ is $3$-connected, since $\{s_1,s_2,s_\ell\}$ is not a triad for any $\ell \in \{3,4,5\}$, by orthogonality.
  Note that $\{s_1,e,s_2\}$ is the unique triangle containing $e$, so $\si(M/e)$ has rank and corank three.
  Since $e$ is $U_{2,5}$-contractible, $e$ is also $U_{3,5}$-contractible by \cref{utfutfprop}.
  Since $s_i$ is in a parallel pair in $M / e$, it follows that $s_i$ is $U_{3,5}$-deletable, for $i \in \{1,2\}$.
  Similarly, $s_i'$ is $U_{3,5}$-contractible for $i \in \{1,2\}$.
  Now $\co(M \ba e) \cong M \ba e / s_1',s_2' \cong U_{2,5}$.
  If $M / s_1',s_2' \cong U_{2,6}$, then each element in $\{s_3,s_4,s_5\}$ is $U_{3,5}$-deletable, so the \lcnamecref{jail-elastic} holds.
  Otherwise, it follows that there is a circuit $\{e,s_1',s_2',s_\ell\}$, for some $\ell \in \{3,4,5\}$.
  But then $s_\ell$ is $U_{3,5}$-deletable, and again the \lcnamecref{jail-elastic} holds.

  Now suppose there is no circuit of the form $\{s_1,e,s_2\}$ for $s_1 \in S_1$ and $s_2 \in S_2$.
  Let $S_1 = \{s_1,s_1'\}$ and $S_2 = \{s_2,s_2'\}$, and let $\{i,j\} = \{1,2\}$.
  First, observe that if $\{s_i,e\} \cup S_j$ is a circuit, then it is readily checked that $M / s_j' \ba s_j \cong M / s_j \ba s_j' \cong U_{3,5}$, so $s_j$ is $U_{3,5}$-flexible, a contradiction.
  By \cref{presimpcase2}, either $\si(M/s_i)$ is $3$-connected, or there is a vertical $3$-separation $(X, s_i, Y)$ such that $s_i' \in Y$ and $S_{j} \cup e \subseteq X$.
  In the latter case $r(Y)=3$, so, by closing $Y \cup s_i$, we may assume that $X = S_j \cup e$ and $Y = \{s_i',s_3,s_4,s_5\}$.
  Then $s_i$ is in a circuit contained in $\{s_i,e\} \cup S_j$.
  Such a circuit must contain $e$, so $\{s_i,e\} \cup S_j$ is a circuit, a contradiction.
  So $\si(M/s_i)$ is $3$-connected.
  Similarly, $\si(M/s)$ is $3$-connected for all $s \in S_1 \cup S_2$.
  Moreover, $\si(M/s)$ has rank and corank at least~$3$, so, by \cref{utfutfprop}, $s$ is $U_{3,5}$-contractible.
  If there exists some $\ell \in \{3,4,5\}$ such that $s_\ell$ is not in a $4$-element circuit with $e$ that meets $S_1$ and $S_2$, then $M/s_\ell \cong Q_6$, in which case $s_\ell$ is $U_{3,5}$-contractible, and $s_{\ell'}$ is $U_{3,5}$-deletable for each $\ell' \in \{3,4,5\} - \ell$.
  Then $M$ has no $U_{3,5}$-essential elements, so the \lcnamecref{jail-elastic} holds.
  So for each $\ell \in \{3,4,5\}$, there is a $4$-element circuit containing $\{s_\ell,e\}$ that meets $S_1$ and $S_2$.
  Note that no two of these three circuits intersects $S_1 \cup S_2$ in the same pair of elements, for otherwise $r(M)=3$.
  So, without loss of generality, $\{s_1,s_2,s_3,e\}$, $\{s_1,s_2',s_4,e\}$, and $\{s_1',s_2,s_5,e\}$ are circuits.
  Now $M/e$ has triangles $\{s_1,s_2,s_3\}$, $\{s_1,s_2',s_4\}$, $\{s_1',s_2,s_5\}$, and $\{s_3,s_4,s_5\}$.
  Since $M/e$ is a $7$-element rank-$3$ matroid with a $U_{2,5}$-minor, it has some element that is not in two distinct triangles.
  It follows that $M/e$ has precisely the four aforementioned triangles.
  But now $M/e \ba s_1 \ba s_5 \cong U_{3,5}$, so $s_1$ is $U_{3,5}$-flexible, a contradiction.

  We may now assume that $M \ba e$ has at least two non-trivial series classes where, if there are precisely two non-trivial series classes, then one has size at least~$3$.
  First, assume that $\co(M \ba e) \cong U_{2,6}$.
  Let $S$ be a non-trivial series class of $M \ba e$ where, if there are only two such series classes, then $|S| \ge 3$.
  By \cref{cosimpcase}, there exists some $f \in S$ such that $\co(M \ba f)$ is $3$-connected and has rank and corank at least~$3$.
  Now $\co(M \ba (S \cup e)) \cong U_{2,5}$, so $\co(M \ba f)$ has a $U_{2,5}$-minor.
  By \cref{utfutfprop}, $\co(M \ba f)$, and hence $M \ba f$, has a $U_{3,5}$-minor.
  Moreover, since $\co(M \ba e)$ has a $U_{2,5}$-minor, $M/f$ has a $U_{2,5}$-minor.
  By \cref{simpcase,utfutfprop}, $M/f$ has a $U_{3,5}$-minor, so $f$ is $U_{3,5}$-flexible, a contradiction.

  We may now assume that $\co(M \ba e) \cong U_{2,5}$.
  Let $(S_1,S_2,\dotsc,S_5)$ be a partition of $E(M \ba e)$ into series classes where, for some $h \in \{2,3,4,5\}$, we have $|S_i| \ge 2$ if and only if $i \in [h]$, and, in the case that $h=2$, we have $|S_2| \ge 3$.
  For $i \in [5]-[h]$, let $S_i = \{s_i\}$.

  \begin{claim}
    \label{noparallelcases}
    Let $s_i \in S_i$ for $i \in [h]$.
    Then $M / (\bigcup_{i \in [h]} S_i-s_i)$ is loopless and has a single parallel pair, which contains $e$.
  \end{claim}
  \begin{subproof}
  The matroid $M' = M / (\bigcup_{i \in [h]} S_i - s_i)$ has rank two, and $M' \ba e \cong U_{2,5}$, so either $M' \cong U_{2,6}$, or $e$ is a loop in $M'$, or $M'$ has a single parallel pair, which contains $e$.
  For $i \in [h]$, let $S_i^- = S_i-s_i$.
  Firstly, suppose that $M' \cong U_{2,6}$.
  Then $e \notin \cl_M(\bigcup_{i \in [h]} S_i^-)$, and it follows that $M / (\bigcup_{i \in [h-1]} S_i^-) \ba S_h \cong U_{2,5}$.
  By \cref{cosimpcase}, there exists $f \in S_h$ such that $\co(M \ba f)$ is $3$-connected with rank and corank at least~$3$.
  Since $M \ba f$ has a $U_{2,5}$-minor, \cref{utfutfprop} implies that $M \ba f$ has a $U_{3,5}$-minor.
  Moreover, by \cref{simpcase,utfutfprop}, $M/f$ has a $U_{3,5}$-minor, so $f$ is $U_{3,5}$-flexible, a contradiction.

  Now we may assume that $e$ is a loop in $M'$.
  Let $T = \bigcup_{i \in [h]} S_i^-$.
  Then there is a circuit~$C$ contained in $T \cup e$.
  Note that, by orthogonality, if $C$ meets $S_i^-$ for some $i \in [h]$, then $S_i^- \subseteq C$.
  Since $|C| \ge 3$, either $C$ meets some $S_i^-$ where $|S_i^-| \ge 2$, or $C$ meets $S_i^-$ and $S_j^-$ for distinct $i,j \in [h]$.
  Thus, in the case that 
  $h=2$, we have $S_2^- \subseteq C$.
  For any $c \in C$, the matroid $M \ba c$ has a $U_{2,5}$-minor, since $\co(M \ba e) \cong M \ba e / T \cong M \ba c / e / (T-c) \cong U_{2,5}$.
  If $h \ge 3$, then, by \cref{cosimpcase}\ref{csc1}, $\co(M \ba c)$ is $3$-connected and has rank and corank at least~$3$, for each $c \in C$.
  Otherwise, when $h=2$, the circuit $C$ contains $S_2^-$ with $|S_2^-| \ge 2$, and by \cref{cosimpcase}\ref{csc2} there exists some $c \in S_2^-$ such that $\co(M \ba c)$ is $3$-connected and has rank and corank at least~$3$.
  In either case, $M \ba c$ has a $U_{3,5}$-minor by \cref{utfutfprop}.
  Moreover, by \cref{simpcase,utfutfprop}, $M / c$ has a $U_{3,5}$-minor.
  So $c$ is $U_{3,5}$-flexible, a contradiction.
\end{subproof}

  By \cref{noparallelcases}, we may now assume, for every choice of $s_i$'s, that $\{e,s_j\}$ is a parallel pair in $M / (\bigcup_{i \in [h]} S_i - s_i)$, for some $j \in [5]$.
  So $M \ba s_j$ has a $U_{2,5}$-minor.

  Assume next that $M \ba e$ has precisely two non-trivial series classes $S_1$ and $S_2$, with $|S_1| = 2$ and $|S_2|=3$.
  We will show that this case is contradictory.
  Let $S_1 = \{s_1,s_1'\}$ and $S_2 = \{s_2,s_2',s_2''\}$.
  Since $|S_1 \cup S_2 \cup e|=6$ and $r(M)=5$, the set $S_1 \cup S_2 \cup e$ contains a circuit~$C$.
  Since $\co(M \ba e) \cong U_{2,5}$, it follows that $e \in C$.
  Then, by orthogonality, $|C| \ge 4$.
  By \cref{noparallelcases}, $|C| \neq 4$.
  Suppose $S_1 \cup S_2 \cup e$ is a circuit.
  It now follows from \cref{noparallelcases} and circuit elimination that, without loss of generality, $\{s_1,s_2,s_2',e,s_3\}$, $\{s_1,s_2,s_2'',e,s_4\}$, and $\{s_1,s_2',s_2'',e,s_5\}$ are circuits; and $\{s_1',s_2,s_2',e,s_5\}$, $\{s_1',s_2,s_2'',e,s_3\}$, and $\{s_1',s_2',s_2'',e,s_4\}$ are circuits.
  But then $M/e$ has no $U_{2,5}$-minor, a contradiction.
  Next suppose $S_2 \cup \{s_1,e\}$ is a circuit. 
  Then each element in $S_2$ is $U_{2,5}$-flexible.
  By \cref{simpcase,cosimpcase,utfutfprop}, there is some $f \in S_2$ that is $U_{3,5}$-flexible, a contradiction.

  Now, up to labels, we may assume that $S_1 \cup \{s_2',s_2'',e\}$ is a circuit.
  Then each element in $S_1$ is $U_{2,5}$-flexible.
  By \cref{cosimpcase,utfutfprop}, we may assume, up to labels, that $s_1$ is $U_{3,5}$-deletable.
  We claim that $s_1$ is also $U_{3,5}$-contractible.
  Consider $M / s_1 / s_2'$.
  This is a rank-$3$ matroid with rank-$2$ sets $\{s_1',s_3,s_4,s_5\}$ and $\{s_1',s_2'',e\}$.
  Moreover, by \cref{noparallelcases}, $M$ has a circuit contained in $\{s_1,s_2,s_2',e,q\}$, where $q \in \{s_1',s_2'',s_3,s_4,s_5\}$.
  This circuit is distinct from the circuit $\{s_1,s_1',s_2',s_2'',e\}$.
  By circuit elimination, and since no circuit is contained in $S_1 \cup S_2$, it follows that $q \in \{s_3,s_4,s_5\}$.
  Without loss of generality, $\{s_2,e,s_3\}$ is a circuit of $M / s_1/s_2'$.
  It now follows that $M / s_1 / s_2' \ba s_1' \ba s_3 \cong U_{3,5}$.
  So $s_1$ is $U_{3,5}$-flexible, a contradiction.

  Recall that when $h=2$, we may assume that $|S_2| \ge 3$.
  By the foregoing, we may now also assume, in this case, that $|S_1|+|S_2| \ge 6$.

  \begin{claim}
    \label{parallelgoodcase}
    Let $s_i \in S_i$ for $i \in [h]$.
    Then $M / (\bigcup_{i \in [h]} S_i-s_i)$ has a parallel pair $\{e,s_j\}$, for some $j \in [5]-[h]$.
  \end{claim}
  \begin{subproof}
    Suppose $M / (\bigcup_{i \in [h]} S_i-s_i)$ has a parallel pair $\{e,s_j\}$ for $j \in [h]$.
  To begin with, assume also that $h \ge 3$.
  Then \cref{cosimpcase}\ref{csc1} implies that $\co(M \ba s_j)$ is $3$-connected with rank and corank at least~$3$.
  Thus, by \cref{utfutfprop}, $M \ba s_j$ has a $U_{3,5}$-minor.
  Moreover, by \cref{simpcase,utfutfprop}, $M/s_j$ has a $U_{3,5}$-minor, so $s_j$ is $U_{3,5}$-flexible, a contradiction.
  So $h=2$.  In particular, $j \in \{1,2\}$.

  Suppose that $|S_j| \ge 3$.
  Let $S_i^- = S_i - s_i$ for $i \in \{1,2\}$, and let $\{j,j'\} = \{1,2\}$.
  Then $S_{j'} \cup S_j \cup e$ contains a circuit~$C$ of $M$, with $\{e,s_j\} \subseteq C$.
  Suppose $\co(M \ba s_j)$ is not $3$-connected.
  Then, by \cref{precosimpcase}, there exists a circuit~$C'$ such that $S_j^- \cup e \subseteq C' \subseteq S_{j'} \cup S_j^- \cup e$.
  Note that $s_j \in C-C'$, so $C \neq C'$, and $e \in C \cap C'$.
  By circuit elimination, there is a circuit contained in $(C \cup C') - e \subseteq S_{j'} \cup S_j$.
  By orthogonality, $S_{j'} \cup S_j$ is a circuit.
  But then $\co(M \ba e)$ contains a parallel pair, a contradiction.
  So $\co(M \ba s_j)$ is $3$-connected; moreover, this matroid has rank and corank at least~$3$, by \cref{cosimpcase}.
  Now $s_j$ is $U_{2,5}$-deletable, so, by \cref{utfutfprop}, $M \ba s_j$ has a $U_{3,5}$-minor.
  Moreover, by \cref{simpcase,utfutfprop}, $M/s_j$ has a $U_{3,5}$-minor, so $s_j$ is $U_{3,5}$-flexible, a contradiction.
  We deduce that $|S_j| = 2$. 
  In particular, $j=1$, as $|S_2| \ge 3$.  Since $|S_1| + |S_2| \ge 6$, we have $|S_2| \ge 4$.
  Now, as $S_2 \cup e$ is a coclosed cosegment, $M.(S_2 \cup e) \cong U_{|S_2|-1,|S_2|+1}$. 
  Since $|S_2| \ge 4$, the element $s_3$ is $U_{3,5}$-contractible.
  Moreover, $s_3$ is a loop in $M.(S_2 \cup \{e,s_3\})$, so it is $U_{3,5}$-flexible, a contradiction.
\end{subproof}

  By \cref{parallelgoodcase} we may now assume, for every choice of $s_i$'s, that $\{e,s_j\}$ is a parallel pair in $M / (\bigcup_{i \in [h]} S_i-s_i)$ for some $j \in [5] - [h]$ (in other words, for some $j$ such that $\{s_j\}$ is a series class of $M \ba e$).
  Then there is a circuit~$C$ such that $\{e,s_j\} \subseteq C \subseteq \{e,s_j\} \cup (\bigcup_{i \in [h]} S_i-s_i)$.
  By orthogonality, $C= \{e,s_j\} \cup (\bigcup_{i \in [h]} S_i-s_i)$.
  Note that $r(M) = 2 +\Sigma_{i \in [h]} (|S_i| - 1) = |C|$.
  So $r(C) = r(M)-1$.
  Moreover, any proper superset $D$ of $C$ contains at least two series classes of $M \ba e$, in which case $D$ spans $E(M)$.
  So $C$ is a circuit-hyperplane and $E(M)-C$ is a cocircuit.

  Suppose that $M \ba e$ has precisely one trivial series class, so $j=5$ for every choice of $s_i$'s.
  Then $\{e,s_5\} \cup \left(\bigcup_{i \in [4]} S_i - s_i\right)$ is a circuit and $C^*=\{s_1,s_2,s_3,s_4\}$ is a cocircuit.
  But $\{e,s_5\} \cup \left(\bigcup_{i \in [3]} S_i-s_i\right) \cup (S_4-s_4')$ is also a circuit for $s_4' \in S_4-s_4$, and this circuit intersects $C^*$ in a single element, $s_4$, contradicting orthogonality.

  Next suppose that $M \ba e$ has precisely two trivial series classes, $\{s_4\}$ and $\{s_5\}$.
  Suppose $|S_1|=|S_2|=|S_3|=2$, so $|E(M)| = 9$.
  Then it is readily checked that $M / S_i \ba s_4,s_5 \cong U_{3,5}$, for $i \in [3]$.
  So $M$ has at most one $U_{3,5}$-essential element, and thus the \lcnamecref{jail-elastic} holds in this case.

  We may now assume that $|S_3| \ge 3$, say.
  Let $s_3,s_3',s_3''$ be distinct elements in $S_3$.
  The set $C_1=\{e,s_j\} \cup \left(\bigcup_{i \in [3]} S_i - s_i\right)$ is a circuit and $C_1^*=\{s_1,s_2,s_3,s_{j'}\}$ is a cocircuit, for some $\{j,j'\} = \{4,5\}$.
  Also, $C_2=\{e,s_k\} \cup \left(\bigcup_{i \in [2]} S_i - s_i\right) \cup (S_3-s_3')$ is a circuit and $C^*_2=\{s_1,s_2,s_3',s_{k'}\}$ is a cocircuit, for some $\{k,k'\} = \{4,5\}$.
  By orthogonality between $C_1$ and $C^*_2$, we have $j=k'$, so $j'=k$.
  Furthermore, $C_3=\{e,s_\ell\} \cup \left(\bigcup_{i \in [2]} S_i - s_i\right) \cup (S_3-s_3'')$ is a circuit and $C^*_3=\{s_1,s_2,s_3'',s_{\ell'}\}$ is a cocircuit, for some $\{\ell,\ell'\} = \{4,5\}$.
  By orthogonality between $C_1$ and $C^*_3$, we have $j=\ell'$, 
  but by orthogonality between $C_2$ and $C^*_3$, we have $k=\ell'$, so $j=k$, a contradiction.

  Now suppose that $M \ba e$ has three trivial series classes, so $M \ba e$ has precisely two non-trivial series classes, $S_1$ and $S_2$, and $|S_2| \ge 3$.
  If $S_2$, say, has size at least~$4$, then $S_2 \cup e$ is a coclosed cosegment of size at least~$5$, and it follows that any $f \in S_1$ is $U_{3,5}$-flexible.
  Recall also that $|S_1| + |S_2| \ge 6$.
  So we may assume that $|S_1| = 3$ and $|S_2| = 3$.
  Let $S_1 = \{t_1,t_2,t_3\}$ and $S_2 = \{u_1,u_2,u_3\}$.
  As before, for $i,j \in [3]$, the set $C_{i,j}=\{e,w_{i,j}\} \cup (S_1 - t_i) \cup (S_2 - u_j)$ is a circuit and $C_{i,j}^*=\{t_i,u_j\} \cup (\{s_3,s_4,s_5\}-w_{i,j})$ is a cocircuit, with $\{w_{i,1},w_{i,2},w_{i,3}\} = \{3,4,5\}$ for $i \in [3]$ and $\{w_{1,j},w_{2,j},w_{3,j}\} = \{3,4,5\}$ for $j \in [3]$.
  Without loss of generality, 
  $w_{i,j}=s_{((i+j) \bmod 3)+3}$.
  As $M / e$ has a $U_{2,5}$-minor and $r(M)=6$, there exists a $3$-element independent set $C \subseteq E(M/e)$ such that $M/(C \cup e)$ has a $U_{2,5}$-minor.
  Suppose that $C \subseteq E(M/e) - S_2$.
  Then $(E(M/e)-S_2) - C$ is a parallel class of size three in $M/(C \cup e)$, in which case $|E(\co(M/(C \cup e)))| \le 4$, implying $M / (C \cup e)$ has no $U_{2,5}$-minor.
  So $C$ meets $S_2$ and, similarly, $C$ meets $S_1$.
  Let $C = (S_1 - t_i) \cup u_j$.
  Then $M / (C \cup e)$ has two distinct parallel pairs, due to the circuits $C_{i,j'}$ for $j' \in [3]-j$, so again $M / (C \cup e)$ has no $U_{2,5}$-minor.
  Now we may assume that $|C \cap S_1| = 1$ and $|C \cap S_2| = 1$.
  Without loss of generality let $C=\{t_1,u_1,s_4\}$.
  Then $\{s_3,s_5\}$ and $C_{2,2}-C=\{t_3,u_3\}$ are parallel pairs in $M/(C \cup e)$, so again $M / (C \cup e)$ has no $U_{2,5}$-minor.
  We deduce that $M/e$ has no $U_{2,5}$-minor, a contradiction.
\end{proof}

\begin{lemma}
  \label{jail-nonelastic}
  For $\mathbb{P} \in \{\mathbb{H}_5,\mathbb{U}_2\}$, 
  let $M$ be a $3$-connected $\mathbb{P}$-representable $U_{3,5}$-fragile matroid that is not $\utfutf$-fragile, where $r(M) \ge 4$ and $r^*(M) \ge 4$.
  Suppose $M$ has no $U_{2,5}$-elastic elements, and let $F$ be the set of $U_{2,5}$-flexible elements of $M$.
  Then one of the following holds:
  \begin{enumerate}
    \item $|F| \in \{3,4,5\}$, the set $F$ is a fan that is contained in a $5$-element fan $F'$, and there exists an element $g$ such that either $M|(F' \cup g)$ or $M^*|(F' \cup g)$ is isomorphic to $M(K_4)$.
      Moreover, if $|F|=3$, then $F$ is the set of internal elements of $F'$.
    \item $|F| \in \{4,5\}$ but there is no $4$- or $5$-element maximal fan that contains $F$. 
    \item $|F|\ge 6$.
  \end{enumerate}
\end{lemma}
\begin{proof}
  We start by proving the following claim:
  \begin{claim}
    \label{flexstart}
    If $e$ is a $\utfutf$-flexible element of $M$, then $e$ is $U_{2,5}$-flexible.
  \end{claim}
  \begin{subproof}
  Let $e$ be a $\utfutf$-flexible element of $M$.
  Clearly the claim holds if $e$ is $U_{3,5}$-essential, so assume otherwise.
  
  Suppose $e$ is $U_{3,5}$-deletable 
  and $U_{2,5}$-contractible.
  Then $\co(M \ba e)$ has a $U_{3,5}$-minor, so $r(\co(M \ba e)) \ge 3$.
  Since $r^*(M) \ge 4$, we have $r^*(\co(M \ba e)) \ge 3$.
  Now $\co(M \ba e)$ is a $3$-connected matroid with rank and corank at least~$3$, and having a $U_{3,5}$-minor.
  Hence, by \cref{utfutfprop}, $\co(M \ba e)$ has a $U_{2,5}$-minor.
  Then $e$ is $U_{2,5}$-flexible, as claimed.

  Suppose now that $e$ is $U_{3,5}$-contractible 
  and $U_{2,5}$-deletable.
  Since $r(M) \ge 4$, we have $r(\si(M / e)) \ge 3$.
  If $r^*(\si(M / e)) \ge 3$, then $\si(M / e)$ has both a $U_{2,5}$- and a $U_{3,5}$-minor, by \cref{utfutfprop}, in which case $e$ is $U_{2,5}$-flexible, as required.
  Similarly, as $r^*(M \ba e) \ge 3$, we have $r(\co(M \ba e)) = 2$, otherwise, by \cref{utfutfprop}, $e$ is $U_{3,5}$-flexible, a contradiction.
  So $r(\co(M \ba e)) = 2$ and we may assume that $r^*(\si(M / e)) = 2$.
  In particular, in $M$, the element~$e$ is in at least two distinct triangles, and at least two distinct triads.
  If $e$ is in a $4$-element segment $L$ of $M$, then each triad containing $e$ is contained in $L$, by orthogonality.  But then $M$ has a triangle-triad, contradicting that $M$ is $3$-connected.
  It follows that $e$ is in triangles $T_1$ and $T_2$, with $r(T_1 \cup T_2) = 3$; and, similarly, $e$ is in triads $T_1^*$ and $T_2^*$ with $r^*(T_1^* \cup T_2^*) = 3$. By orthogonality, $T_1 \cup T_2 = T_1^* \cup T_2^*$.
  But then $\lambda(T_1 \cup T_2) = 1$, so, as $M$ is $3$-connected, $|E(M)| \le 6$, a contradiction.
  Hence $e$ is $U_{2,5}$-flexible, as claimed.
\end{subproof}

  Since $M$ is not $\utfutf$-fragile, there exists an element~$e$ that is $\utfutf$-flexible.
  By \cref{flexstart}, $e$ is $U_{2,5}$-flexible.
\begin{claim}
  \label{f0}
  For $(M_0,N_0) \in \{(M,U_{3,5}), (M^*,U_{2,5})\}$, the matroid $M_0$ has at least three $N_0^*$-flexible elements.
\end{claim}
\begin{subproof}
  By hypothesis, $e$ is not $U_{2,5}$-elastic in $M$.
  Now, for some $(M_0,N_0) \in \{(M,U_{3,5}), (M^*,U_{2,5})\}$, the matroid $M_0$ is $N_0$-fragile and has an $N_0^*$-flexible element $e$ such that $\si(M/e)$ is not $3$-connected.
  By \cref{vert3sep}, $M_0$ has a vertical $3$-separation $(X,e,Y)$.
  By \cref{niceVertSep} we may assume that $|X \cap E(N_0^*)| \le 1$ and $Y \cup e$ is closed.
  By \cref{CPL}, at most one element of $X$ is not $N_0^*$-flexible so, as $|X| \ge 3$, the set $X$ contains at least two $N_0^*$-flexible elements.
  So $M_0$ has at least three $N_0^*$-flexible elements, as $e$ is also $N_0^*$-flexible.
  The claim follows by duality.
\end{subproof}

  Now, for some $(M_1,N_1) \in 
  \{(M,U_{3,5}), (M^*,U_{2,5})\}$, the matroid $M_1$ is $N_1$-fragile and has two $N_1^*$-flexible elements $e_1$ and $e_2$ such that $\si(M/e_i)$ is not $3$-connected for $i \in \{1,2\}$.
  By \cref{vert3sep}, $M_1$ has vertical $3$-separations $(X_i,e_i,Y_i)$ for $i \in \{1,2\}$, with $Y_i \cup e_i$ closed and $|X_i \cap E(N_1^*)| \le 1$.

  \begin{claim}
    \label{f3}
    If $Z$ is the set of $N_1^*$-flexible elements of $M_1$, and $|Z|=3$, then $Z$ is a triangle that is contained in a $5$-element fan $F'$, and there exists an element $g$ such that $M_1^*|(F' \cup g) \cong M(K_4)$.
  \end{claim}
  \begin{subproof}
  Suppose that $M_1$ has precisely three $N_1^*$-flexible elements.
  Let $i \in \{1,2\}$.
  Then, as $|X_i \cup e_i| \ge 4$ and at most one element in $X_i \cup e_i$ is not $N_1^*$-flexible, by \cref{CPL}, 
  $|X_i|=3$.
  Moreover, there exists $f_i \in X_i \cap \cocl(Y_i)$ that is not $N_1^*$-flexible, and $e_i \in \cl(X_i-f_i)$.
  Then $(X_i-f_i) \cup e_i$ is a triangle and $X_i$ is a triad, so $X_i \cup e_i$ is a $4$-element fan where $e_i$ is a spoke end and $f_i$ is a rim end.
  Now the $N_1^*$-flexible elements form a triangle $\{e_1,e_2,e_3\}$, for some element $e_3$.
  Let $F'$ be the fan with ordering $(f_2,e_1,e_3,e_2,f_1)$, noting that $f_1 \neq f_2$ follows from the fact that $M_1$ is $3$-connected.
  Since $e_3$ is $N_1^*$-flexible but $M_1$ has no $N_1^*$-elastic elements, at least one of $\si(M / e_3)$ or $\co(M \ba e_3)$ is not $3$-connected.

  Suppose $\co(M \ba e_3)$ is not $3$-connected.
  Thus, there exists a cyclic $3$-separation $(X_3,e_3,Y_3)$ with $Y_3 \cup e_3$ coclosed and $|X_3 \cap E(N_1^*)| \le 1$.
  By \cref{CPL}, at most one element of $X_3$ is not $N_1^*$-flexible.
  As $M_1$ has three $N_1^*$-flexible elements, $|X_3| = 3$, so $X_3$ is a triangle.
  But $\{e_1,e_2\} \subseteq X_3$, so $\{e_1,e_2\}$ is contained in a triangle distinct from $\{e_1,e_2,e_3\}$, contradicting orthogonality with the triads $\{f_2,e_1,e_3\}$ and $\{e_3,e_2,f_1\}$.
  Thus $\si(M / e_3)$ is not $3$-connected.
  Then, by \cref{fanmiddle}, there exists an element $g \in E(M_1)-F'$ such that $M_1^*|(F' \cup g) \cong M(K_4)$.
  Letting $F = F' \cup g$, the claim follows.
\end{subproof}

  \begin{claim}
    \label{f4}
    Let $Z$ be the set of $N_1^*$-flexible elements of $M_1$, and suppose $|Z|=4$.
    Then either $Z$ is a fan contained in a set $F$ such that $M_1^*|F$ is isomorphic to $M(K_4)$, or
    there is no $4$- or $5$-element maximal fan that contains $Z$.
  \end{claim}
  \begin{subproof}
  Let $i \in \{1,2\}$.
  Then, as at most one element in $X_i \cup e_i$ is not $N_1^*$-flexible, by \cref{CPL}, we have $|X_i| \in \{3,4\}$.
  First, assume $X_1$ and $X_2$ both contain only two $N_1^*$-flexible elements.
  Then, by \cref{CPL}, for each $i \in \{1,2\}$, we have $|X_i| = 3$, there exists $f_i \in X_i \cap \cocl(Y_i)$ that is not $N_1^*$-flexible, and $e_i \in \cl(X_i-f_i)$.
  Thus $(X_i-f_i) \cup e_i$ is a triangle and $X_i$ is a triad, so $X_i \cup e_i$ is a $4$-element fan where $e_i$ is a spoke end and $f_i$ is a rim end.
  Let $X_i' = (X_i-f_i) \cup e_i$ for $i \in \{1,2\}$.
  Since $X'_1 \cup X'_2 \subseteq Z$ and $|Z|=4$, we have $|X'_1 \cap X'_2| \ge 2$.
  But if $|X'_1 \cap X'_2|=2$, then $X'_1 \cup X'_2$ is a $4$-element segment, contradicting orthogonality with the triad $X_1$.
  So $X'_1 = X'_2$, and this set is a triangle of $N_1^*$-flexible elements.
  Now $X'_1 = \{e_1,e_2,e_3\}$ for some element $e_3$.
  Let $F'$ be the fan with ordering $(f_2,e_1,e_3,e_2,f_1)$,
  where $f_1 \neq f_2$ follows from the fact that $M_1$ is $3$-connected.

  Let $Z-\{e_1,e_2,e_3\} = \{e_4\}$.
  Since $e_3$ is $N_1^*$-flexible but $M_1$ has no $N_1^*$-elastic elements, at least one of $\si(M / e_3)$ or $\co(M \ba e_3)$ is not $3$-connected.
  Suppose
  $\co(M \ba e_3)$ is not $3$-connected.
%
  Thus, there exists a cyclic $3$-separation $(X_3,e_3,Y_3)$ with $Y_3 \cup e_3$ coclosed and $|X_3 \cap E(N_1^*)| \le 1$.
  As at most one element of $X_3 \cup e_3$ is not $N_1^*$-flexible, by \cref{CPL}, we have $|X_3 \cap Z| \in \{2,3\}$, so $|X_3| \in \{3,4\}$.
  If $|X_3| = 3$, then $X_3$ is a triangle that contains at least two elements of $\{e_1,e_2,e_4\}$, in which case it follows from orthogonality that either $X_3=\{e_1,f_2,e_4\}$ or $X_3=\{f_1,e_2,e_4\}$, so $Z$ is contained in a $6$-element fan and the claim holds.
  So we may assume that $|X_3| = 4$, in which case $X_3 = \{e_1,e_2,e_4,p\}$ for some element $p$.
  Since $p$ is not $N_1^*$-flexible, we have 
  $p \in \cl(Y_3)$ and $e_3 \in \cocl(\{e_1,e_2,e_4\})$ by \cref{CPL}.
  Now $\{e_1,e_2,e_4\}$ is $3$-separating, and $e_3$ is in the closure and coclosure of this set, so $\lambda(\{e_1,e_2,e_3,e_4\})=1$, a contradiction.
  So we may assume that $\si(M / e_3)$ is not $3$-connected.
  Then, by \cref{fanmiddle}, there exists an element $g \in E(M_1)-F'$ such that $M_1^*|(F' \cup g) \cong M(K_4)$.
  If $g = e_4$, then the claim holds with $F = F' \cup e_4$.
  So we may assume that $g \neq e_4$.

  Suppose, for a contradiction, that $Z$ is contained in a fan $F$ with $|F|\le 5$.
  By \cref{fanunique}, the unique triangle containing $e_3$ is $\{e_1,e_2,e_3\}$, and $\{e_1,e_2,g\}$ is the unique triad containing $\{e_1,e_2\}$ by orthogonality.
  Thus, if $e_3$ is a spoke end of $F$, then $(e_3,e_i,e_j,g,e_4)$ is a fan ordering of $F$ for some $\{i,j\}=\{1,2\}$.
  But then $\{e_j,g,e_4\}$ is a triangle, contradicting orthogonality.
  So $e_3$ is not a spoke end of $F$.
  As the unique triads containing $e_3$ are $\{e_3,e_1,f_2\}$ and $\{e_3,e_2,f_1\}$, by \cref{fanunique}, this implies that $f_i \in F$ for some $i \in \{1,2\}$.
  Then, without loss of generality, $F=\{e_1,e_2,e_3,f_1,e_4\}$, so either $e_4$ is in a triangle with $f_1$, 
  or $e_4$ is in a triad with $e_1$. 
  But $\{e_4,f_1\}$ is not contained in a triangle, by orthogonality.
  So, by orthogonality again, either $\{e_4,e_1,e_2\}$ or $\{e_4,e_1,e_3\}$ is a triad.
  In the former case, $\{e_1,e_2,g,e_4\}$ is a cosegment, and in the latter case $\{e_1,f_2,e_3,e_4\}$ is a cosegment; both contradict orthogonality with the triangle $\{e_1,e_2,e_3\}$.

  Now we may assume that $X_1$ contains precisely three $N_1^*$-flexible elements.
  Suppose $|X_1|=4$.  Then, by \cref{CPL}, there is a unique element $f_1 \in X_1$ that is not $N_1^*$-flexible, and, letting $X_1' = X_1 - f_1$ and $Y_1' = Y_1 \cup f_1$, there is a path of $3$-separations $(X_1', e_1, Y_1')$.
  If $r(X_1') = 2$, then $Z=X_1' \cup e_1$ is a segment, so there is no fan that contains $Z$.  So we may assume that $r(X_1') = 3$, in which case $(X_1', e_1, Y_1')$ is a vertical $3$-separation such that $Y_1' \cup e_1$ is closed, and $|X_1' \cap E(N_1^*)| \le 1$.
  Thus, by replacing $(X_1,e_1,Y_1)$ with $(X_1',e_1,Y_1')$ if necessary, we may assume that $|X_1| = 3$.
  By a similar argument, we may assume that $|X_2| = 3$.

  If each element of $X_2$ is $N_1^*$-flexible, then $X_1 \cup e_1 = X_2 \cup e_2$.  But then, as $X_1$ and $X_2$ are distinct triads, $X_1 \cup e_1$ is a cosegment, contradicting that $e_1$ is a guts element.
  So $X_2$ has two $N_1^*$-flexible elements, in which case, as before, there is a unique element $f_2 \in X_2$ that is not $N_1^*$-flexible, and $X_2 \cup e_2$ is a $4$-element fan where $e_2$ is a spoke end and $f_2$ is a rim end.
  Thus $e_2$ is in a triangle that contains $e_1$ and is contained in $X_1 \cup e_1$; we choose $e_3$ and $e_4$ so that this triangle is $\{e_1,e_2,e_3\}$, and $X_1 = \{e_2,e_3,e_4\}$.
  Note that
  $X_1$ is an independent triad, 
  $X_1 \cup e_1$ is a $4$-element fan where $e_1$ is a spoke end,
  and $X_2 = \{f_2,e_1,e_3\}$.
  Thus $(e_4,e_2,e_3,e_1,f_2)$ is an ordering of a $5$-element fan~$F'$, where 
  $\{e_1,e_2,e_3,e_4\}$ is the set of $N_1^*$-flexible elements in $F'$.

  As $e_3$ is $N_1^*$-flexible but not $N_1^*$-elastic, either $\co(M_1 \ba e_3)$ or $\si(M_1/e_3)$ is not $3$-connected.
  If $\si(M_1 / e_3)$ is not $3$-connected, then, by \cref{fanmiddle}, there exists an element $g$ such that $M_1^*|(F' \cup g) \cong M(K_4)$, so \ref{f4} holds.
  So we may assume that $\co(M_1 \ba e_3)$ is not $3$-connected.
  Then there exists a cyclic $3$-separation $(X_3,e_3,Y_3)$ with $Y_3 \cup e_3$ coclosed and $|X_3 \cap E(N_1^*)| \le 1$.
  By \cref{CPL}, at most one element of $X_3$ is not $N_1^*$-flexible, so $|X_3| \in \{3,4\}$.
  We may assume that $|X_3|=3$ (by the same argument used earlier for $X_1$ and $X_2$),
  in which case $X_3$ is a triangle that contains at least two elements of $\{e_1,e_2,e_4\}$.
  By orthogonality, $X_3 = \{e_4,e_2,x\}$ for some $x \notin F'$.
  But then $F' \cup x$ is a $6$-element fan, so \ref{f4} holds.
\end{subproof}

  \begin{claim}
    \label{f5}
    Let $F$ be the set of $N_1^*$-flexible elements of $M_1$, and suppose $|F|=5$.
    If $F$ is a maximal fan, then there exists an element $g \in E(M_1)-F$ such that either $M_1|(F \cup g)$ or $M_1^*|(F \cup g)$ is isomorphic to $M(K_4)$.
  \end{claim}
  \begin{subproof}
  Suppose $F$ forms a maximal fan with fan ordering $(f_1,f_2,\dotsc,f_5)$.
  For some $(M_2,N_2) \in \{(M_1,N_1), (M_1^*,N_1^*)\} = \{(M,U_{3,5}), (M^*,U_{2,5})\}$, the elements $f_1$ and $f_5$ are spoke ends of $F$, the matroid $M_2$ is $N_2$-fragile, and $F$ is the set of $N_2^*$-flexible elements in $M_2$.
  As $f_3$ is $N_2^*$-flexible but not $N_2^*$-elastic, at least one of $\si(M_2/f_3)$ or $\co(M_2 \ba f_3)$ is not $3$-connected.
  Suppose $\si(M_2/f_3)$ is not $3$-connected.
  Then there exists a vertical $3$-separation $(X,f_3,Y)$ with with $Y \cup f_3$ closed and $|X \cap E(N_2^*)| \le 1$.
  By \cref{CPL}, at most one element of $X$ is not $N_2^*$-flexible, so $|X| \in \{3,4\}$.

  Suppose $|X|=3$, in which case $X$ is a triad that contains at least two elements of $F - f_3$.
  If $X \subseteq F-f_3$, then $X$ intersects one of the triangles $\{f_1,f_2,f_3\}$ or $\{f_3,f_4,f_5\}$ in a single element, contradicting orthogonality.
  So, by orthogonality, $X \cap F$ is either $\{f_1,f_2\}$ or $\{f_4,f_5\}$.
  But then $F \cup X$ is a $6$-element fan, contradicting that the fan $F$ is maximal.
  Now $|X| = 4$.
  If each element of $X$ is $N_2^*$-flexible, then $X = F-f_3$.
  But then $f_3 \in \cocl(X)$, so $f_3 \notin \cl(Y)$, a contradiction.
  So there is an element $x \in X-F$ that is not $N_2^*$-flexible.
  Then, by \cref{CPL}, $x \in \cocl(Y)$. 
  It follows that $X-x$ is $3$-separating.
  But $X-x \subseteq F-f_3$, so $X-x$ is not a triad, by orthogonality, and $X - x$ is not a triangle, as $r(F)=3$, a contradiction.

  We deduce that $\co(M_2 \ba f_3)$ is not $3$-connected.
  Then, by \cref{fanmiddle}, there exists an element $g$ such that $M_2|(F \cup g) \cong M(K_4)$.
\end{subproof}
It is easily seen that if there is a fan $F \subseteq X \subsetneqq E(M_1)$ such that $M_1|X \cong M(K_4)$, then $|F| \le 5$.
The \lcnamecref{jail-nonelastic} now follows from this fact and \cref{f0,f3,f4,f5}.
\end{proof}

\begin{proposition}
  \label{jail}
  Let $\mathbb{P} \in \{\mathbb{H}_5,\mathbb{U}_2\}$.
  Suppose $M$ is a $3$-connected $\mathbb{P}$-representable $U_{3,5}$-fragile matroid that is not $\utfutf$-fragile, where $r(M) \ge 4$ and $r^*(M) \ge 4$.
  Let $F$ be the set of $U_{2,5}$-flexible elements of $M$.
  Then one of the following holds:
  \begin{enumerate}
    \item $|E(M)| \le 9$, and $M$ has at most two $U_{3,5}$-essential elements.\label{9eltfewess}
    \item $|F| \ge 4$ and $F$ is not contained in a maximal fan of size at most five.\label{othercase}
    \item $|F| \in \{3,4,5\}$, the set $F$ is a fan that is contained in a $5$-element fan $F'$, and there exists an element $g$ such that either $M|(F' \cup g)$ or $M^*|(F' \cup g)$ is isomorphic to $M(K_4)$. Moreover, $F$ is the set of internal elements of $F'$ when $|F|=3$.\label{annoyingextracase}
  \end{enumerate}
\end{proposition}
\begin{proof}
  If $M$ has a $U_{2,5}$-elastic element, then \ref{9eltfewess} holds by \cref{jail-elastic}.
  Otherwise, $M$ has no $U_{2,5}$-elastic elements, and \ref{othercase} or \ref{annoyingextracase} holds by \cref{jail-nonelastic}.
\end{proof}

The next lemma was verified by computer.  Note that computational techniques for efficiently enumerating $3$-connected $\mathbb{P}$-representable matroids, for a partial field $\mathbb{P}$, are described in \cite{BP20}.

\begin{lemma}
  \label{maxsized}
  For $\mathbb{P} \in \{\mathbb{H}_5,\mathbb{U}_2\}$,
  suppose $M$ is a $3$-connected $\mathbb{P}$-representable matroid with $r(M) \le 3$. Then $|E(M)| \le 12$.
  Moreover, 
  \begin{enumerate}
    \item if $|E(M)| =9$, then $M$ has no $U_{2,5}$-essential elements, at most three $U_{3,5}$-essential elements, and at least six $U_{3,5}$-deletable elements; and\label{ms1}
    \item if $|E(M)| \in \{10,11,12\}$, then $M$ has at least six $U_{2,5}$-flexible elements.\label{ms2}
  \end{enumerate}
\end{lemma}

\begin{theorem}
  \label{utfutffragile}
  Let $M$ be an excluded minor for the class of $\mathbb{P}$-representable matroids where $\mathbb{P} \in \{\mathbb{H}_5,\mathbb{U}_2\}$.
  Suppose $|E(M)| \ge 16$, and there are distinct elements $a,b \in E(M)$ such that $M \ba a,b$ is $3$-connected and has a $\utfutf$-minor.
  Then $M \ba a,b$ is a $\utfutf$-fragile matroid with rank and corank at least~$4$.
\end{theorem}
\begin{proof}
  Clearly $r(M) \ge 3$ and $r^*(M \ba a,b) \ge 3$, for otherwise $M \ba a,b$ has a $U_{2,7}$- or $U_{5,7}$-minor, so is not $\mathbb{P}$-representable.
  So, by \cref{utfutfprop}, $M \ba a,b$ has both a $U_{2,5}$-minor and a $U_{3,5}$-minor.
  Moreover, if $M \ba a,b$ has rank or corank three, then it has at most 12 elements, by \cref{maxsized}, so $|E(M)| \le 14$, a contradiction.
  So $r(M) \ge 4$ and $r^*(M \ba a,b) \ge 4$.

  Towards a contradiction, assume that $M \ba a,b$ is not $\utfutf$-fragile.
  By \cref{utfutfequiv}, $M \ba a,b$ is not $N$-fragile for some $N \in \utfutf$.
  Suppose $M \ba a,b$ is $N^*$-fragile but not $N$-fragile.
  Let $F$ be the set of $N$-flexible elements of $M \ba a,b$.
  Then, by \cref{jail}, either $|F| \ge 4$ and $F$ is not contained in a maximal fan of size at most five, or $F$ is contained in a set $F'$ such that either $M|F'$ or $M^*|F'$ is isomorphic to $M(K_4)$.
  By \cref{bcosw-thm}\ref{bcoswii}, if $|F| \ge 4$, then $F$ is contained in a maximal fan of size at most five, so the former does not hold.
  Moreover, $F$ is not contained in an $M(K_4)$ restriction or co-restriction,
  by \cref{gadgetnotMK4}, so the latter does not hold.
  We deduce that $M \ba a,b$ is neither $U_{2,5}$-fragile nor $U_{3,5}$-fragile.

  Let $N \in \utfutf$.
  By \cref{bcosw-thm}, $M$ has a basis~$B$ with $x,y \in B$, where $\{b,x,y\}$ is a triangle up to switching the labels of $a$ and $b$, and there is an $(N,B)$-strong element $u \in B^* - \{a,b\}$.
  By \cref{wmatype1again}, we may assume, up to switching the labels of $b$ and $x$, that the $N$-flexible 
  elements of $M \ba a,b$ are contained in the set $\{u,x,y\}$.
  (We note that due to ``switching labels'' in this way, in order to avoid cumbersome notation, the elements henceforth referred to as $a$ and $b$ may not be the same as those given in the statement of the theorem, as we work towards a contradiction.)
%
%
  Recall, by \cref{bcosw-thm}\ref{bcoswiii}, that $\{u,x,y\}$ is the unique triad containing $u$ in $M \ba a,b$.
  Thus $\{x,y\}$ is a series pair in $M \ba a,b,u$.
  Note also that $u$ is $N$-flexible in $M \ba a,b$, for otherwise $M \ba a,b$ is $N$-fragile. 
  Moreover, by applying \cref{bcosw-thm} with the minor~$N^*$, the matroid $M \ba a,b$ has at most five $N^*$-flexible elements and, in the case that $M \ba a,b$ has five $N^*$-flexible elements, they form a $5$-element maximal fan of $M \ba a,b$.

  \begin{claim}
    \label{wmautfutffrag}
    For some $M'' \in \{M \ba a,b \ba u / x,\, M \ba a,b \ba u / x \ba y,\, M \ba a,b \ba u / x / y \}$, the matroid $M''$ is $3$-connected, $\utfutf$-fragile, and has rank and corank at least~$4$.
  \end{claim}
  \begin{subproof}
    Let $M' = M \ba a,b \ba u$.
    By \cref{subfrag3conn}, we can choose $M'' \in \{M' / x,\, M' / x \ba y,\, M' / x / y \}$ such that $M''$ is $3$-connected and $N$-fragile.
    Since $|E(M)| \ge 15$, we have $|E(M'')| \ge 10$.
    If $M''$ has rank or corank at most three, then, by \cref{maxsized}\ref{ms2}, $M''$ has at least six $N'$-flexible elements for some $N' \in \utfutf$. But then $M \ba a,b$ has six $N'$-flexible elements, a contradiction.
    So $M''$ has rank and corank at least~$4$.

    It remains to show that $M''$ is $\utfutf$-fragile.
    Suppose not.
    Then, by \cref{utfutfequiv}, $M''$ is not $N^*$-fragile.
    Let $F$ be the $N^*$-flexible elements of $M''$.
    By \cref{jail} and since $|E(M'')| \ge 10$,
    either $|F| \ge 4$, 
    or $F$ is a triangle or triad of $M''$ that are the internal elements of a $5$-element fan~$F'$ such that either $M''|(F' \cup g)$ or $(M'')^*|(F' \cup g)$ is isomorphic to $M(K_4)$ for some element $g \in E(M'')-F'$.
    Note that $F \subseteq E(M \ba a,b)-\{u,x\}$, and the elements in $F$ are also $N^*$-flexible in $M \ba a,b$.
    By \cref{bcosw-thm}\ref{bcoswii}, there are at most three $N^*$-flexible elements in $E(M \ba a,b)-\{u,x\}$, where in the case there are precisely three, 
    $(y,u,x,z,w)$ is a maximal fan of $M \ba a,b$, and the $N^*$-flexible elements of $M''$ are $\{y,z,w\}$.
    It follows that $F=\{y,z,w\}$ is a triad in $M''$.
    Now \cref{jail}\ref{annoyingextracase} holds, and $F$ is contained in a $5$-element fan~$F'$ such that $M''|(F' \cup g) \cong M(K_4)$ for some element $g$.
    It follows that $\{z,w\}$ is contained in a triangle in $M''$, which, by orthogonality, is also a triangle in $M \ba a,b$, contradicting that the fan $(y,u,x,z,w)$ in $M \ba a,b$ is maximal.
    From this contradiction we deduce that $M''$ is $\utfutf$-fragile.
  \end{subproof}

  \begin{claim}
    \label{maylemma6}
    Let $M' \in \{M \ba a,b \ba u, M \ba a,b / u\}$.
    Then $M'$ has at most two $N$-essential elements.
  \end{claim}
  \begin{subproof}
    Let $M' = M \ba a,b \ba u$ and suppose $M'$ has at least three $N$-essential elements.
    By \cref{wmautfutffrag}, we can choose $M'' \in \{M' / x,\, M' / x \ba y,\, M' / x / y \}$ such that $M''$ is $3$-connected and $\utfutf$-fragile.
    Note that $M''$ also has at least three $N$-essential elements.
    By \cref{hopeful}, $|E(M'')| \le 8$, so $|E(M)| \le 13$, a contradiction.
    Henceforth, we may assume that $M \ba a,b \ba u$ has at most two $N$-essential elements. 

    Let $M' = M \ba a,b / u$ and suppose $M'$ has at least three $N$-essential elements.
    We claim that
    for some $M'' \in \{M',\, M' \ba x,\, M' \ba y,\, M' \ba x,y\}$, the matroid $M''$ is $N$-fragile and $3$-connected up to series classes.
    Since the $N$-flexible elements of $M \ba a,b$ are contained in $\{u,x,y\}$, there is certainly some $M_0'' \in \{M',\, M' \ba x,\, M' \ba x,y\}$ that is $N$-fragile and $3$-connected up to series and parallel classes, by \cref{genfragileconn}.
    Suppose $M_0''$ has a parallel pair.
    Then $u$ is in a triangle of $M \ba a,b$.
    By orthogonality, this triangle meets $\{x,y\}$, so each parallel pair of $M_0''$ meets $\{x,y\}$.
    In particular, if $M_0''$ has a parallel pair, then $M_0'' \neq M' \ba x,y$.
    If $y$ is in a parallel pair of $M_0''$, then $y$ is $N$-deletable, and so $M_0'' \ba y$ is $N$-fragile.
    If $x$ is in a parallel pair of $M_0''$, then $M_0'' = M'$ and $M_0'' \ba x$ is $N$-fragile.
    Now, for some $M'' \in \{M_0'', M_0'' \ba x,\, M_0'' \ba y\}$, the matroid $M''$ is $N$-fragile and $3$-connected up to series classes.

    Since $\{u,x,y\}$ is the unique triad of $M \ba a,b$ containing $u$, there is no triad of $M'$ containing $\{x,y\}$, and hence $x,y \notin S$ for each series class $S$ of $M''$.
    By orthogonality with the triad $\{u,x,y\}$ of $M \ba a,b$, either $x$ or $y$ is in the fundamental circuit $C(u,B)$ of $u$ with respect to $B$.
    Without loss of generality, say $x \in C(u,B)$; then $B-x$ is a basis of $M\ba a,b /u$.
    Moreover, in this matroid $y$ is not in a series pair, and $\{x,y\}$ is not contained in a triad, so there exists some $q \in B^* - \{a,b,u\}$ such that $B' = (B-\{x,y\}) \cup q$ is a basis of $M \ba a,b / u$, and also of $M''$.
    Suppose either $M''$ has two non-trivial series classes, or a series class of size at least~$3$.  Since $|S \cap B'| \ge |S|-1$ for each series class of $M''$, there exists some $b'_1 \in B'-q = B-\{x,y\}$ that is $N$-contractible in $M''$.
    But then $b'_1 \in B-\{x,y\}$ is also $N$-contractible in $M \ba a,b$, contradicting that all the $(N,B)$-robust elements are in $\{u,x,y\}$.
    So $M''$ has at most one non-trivial series class, and this series class has size two.
    In particular, $|E(M'')| \le |E(\co(M''))| +1$.

    Suppose $\co(M'')$ is not $\utfutf$-fragile.
    Then, by \cref{utfutfequiv}, $\co(M'')$ is not $N^*$-fragile.
    Let $F$ be the set of $N^*$-flexible elements of $\co(M'')$.
    Then, by \cref{jail}, either $|E(\co(M''))| \le 9$; the matroid $\co(M'')$ has rank or corank three; 
    $|F| \ge 4$ and $F$ is not contained in a maximal fan of size at most five; or $|F|=3$ and $F$ is the set of internal elements of a $5$-element fan~$F'$ such that either $\co(M'')|(F' \cup g)$ or $(\co(M''))^*|(F' \cup g)$ is isomorphic to $M(K_4)$ for some element $g$.

    Consider the latter two cases.
    Note that $F \subseteq E(M \ba a,b)-\{u\}$, and the elements in $F$ are also $N^*$-flexible in $M \ba a,b$.
    By \cref{bcosw-thm}\ref{bcoswii}, there are at most four $N^*$-flexible elements in $E(M \ba a,b)-\{u\}$, and these elements are contained in a maximal fan of size four or five, with fan ordering $(y,x,u,z)$ or $(y,u,x,z,w)$ respectively.
    In particular, $u$ is in a triangle $\{u,x,z\}$ in $M \ba a,b$.
    Since $M''$ is simple, it follows that $M'' \in \{M' \ba x, M' \ba x,y\}$.
    Then $F \subseteq E(M \ba a,b)-\{u,x\}$, and there are at most three $N^*$-flexible elements in $E(M \ba a,b)-\{u,x\}$ (by \cref{bcosw-thm}\ref{bcoswii} again), so $F=\{y,z,w\}$ and $M'' = M'\ba x$.
    But then $\{z,w\}$ is a series pair in $M''$, so $F \nsubseteq E(\co(M''))$, a contradiction.

    Now suppose $r(\co(M'')) = 3$ or $r^*(\co(M'')) = 3$.
    If $|E(\co(M''))| \le 8$, then $|E(M)| \le 8 + 1 + 5 = 14$, a contradiction.
    By \cref{maxsized}\ref{ms2}, $|E(\co(M''))| \le 12$, and if $|E(\co(M''))| \in \{10,11,12\}$, then $\co(M'')$ has at least six $N'$-flexible elements for some $N' \in \utfutf$, in which case $M \ba a,b$ has six $N'$-flexible elements, a contradiction.
    So $|E(\co(M''))| =9$.
    By \cref{maxsized}\ref{ms1}, and since $M''$ has at least three $N$-essential elements, either $\co(M'')$ has rank three, three $N$-essential elements, and the other six elements are $N$-deletable, where $N \cong U_{3,5}$; or $\co(M'')$ has corank three, three $N$-essential elements, and the other six elements are $N$-contractible, where $N \cong U_{2,5}$.
    If $M''$ has no series pairs, then $|E(M)| \le 14$, a contradiction.
    So let $\{s,s'\}$ be the unique series pair of $M''$, and, without loss of generality, $\co(M'') = M'' / s'$.
    Consider the case where $r(\co(M''))=3$ and $N \cong U_{3,5}$.
    If $s$ is $N$-deletable in $\co(M'')$, then $s$ is $N$-flexible in $M''$, contradicting that $M''$ is $N$-fragile.
    Otherwise, $s$ is $N$-essential in $\co(M'')$, but $s$ and $s'$ are $N$-contractible in $M''$, in which case $M''$ has at most two $N$-essential elements, a contradiction.
    Now consider the case where $r^*(\co(M''))=3$ and $N \cong U_{2,5}$.
    Then $M''$ has rank~$7$, and has $B' = (B-\{x,y\}) \cup q$ as a basis, so there are at least five elements of $\co(M'')$ in $B-\{x,y\}$.
    As $\co(M'')$ has six $N$-contractible elements, and $|E(\co(M''))| = 9$, there is some $b_1' \in B-\{x,y\}$ that is $N$-contractible in $\co(M'')$.
    But then $b_1'$ is also $N$-contractible in $M \ba a,b$, contradicting that all the $(N,B)$-robust elements are in $\{u,x,y\}$.

    Finally, suppose $\co(M'')$ has rank and corank at least~$4$, but $|E(\co(M''))| \le 9$.
    Then \cref{jail}\ref{9eltfewess} holds, so $\co(M'')$ has at most two $N$-essential elements.
    But, as $M'$ has at least three $N$-essential elements, so does $\co(M'')$, a contradiction.

    We deduce that $\co(M'')$ is $\utfutf$-fragile.
    Then, as $\co(M'')$ has at least three $N$-essential elements, \cref{hopeful} implies that $|E(\co(M''))| \le 8$.
    But now $|E(M)| \le |E(\co(M''))| + 1 + 5 = 14$, a contradiction.
  \end{subproof}

  \begin{claim}
    \label{mayclaim}
    $M \ba a,b \ba u$ has at most one $N$-essential element.
  \end{claim}
  \begin{subproof}
    By \cref{wmautfutffrag}, we can choose some $$M'' \in \{M \ba a,b,u / x,\, M \ba a,b,u / x \ba y,\, M \ba a,b,u / x / y \}$$ such that $M''$ is $3$-connected and $\utfutf$-fragile.

    Towards a contradiction, suppose $M \ba a,b \ba u$ has two $N$-essential elements.
    By \cref{utfutf3essentialv2}, either $M''$ or $(M'')^*$ can be obtained from $U_{2,5}$ by gluing a wheel.
    In either case, the resulting fan~$F$ has $|F| \ge 8$, since $|E(M'')| \ge 10$.
    By \cref{fragilefanelements,fragilefans}, $F$ has at least four elements that are $N$-deletable in $M''$.
    Let $d$ be an $N$-deletable element of $M''$.
    If $d \notin B$, then $d$ is $(N,B)$-robust in $M''$ and hence also in $M \ba a,b$.
    But $u$ is the only $(N,B)$-robust element of $M\ba a,b$ that is not in $B$, and $u \notin E(M'')$, so this is contradictory.
    So each $N$-deletable element of $M''$ is in $B$.
    In particular, $F \cap B$ has at least four elements that are $N$-deletable in $M''$.
    As $x \notin E(M'')$, at least three of these are in $B - \{x,y\}$.
    So $M''$, and hence $M \ba a,b \ba u$, has at least three elements in $B-\{x,y\}$ that are not $N$-essential.
    Therefore, $M \ba a,b / u$ has at least three $N$-essential elements, contradicting \cref{maylemma6}.
    We deduce that $M \ba a,b \ba u$ has at most one $N$-essential element.
\end{subproof}

  By \cref{maylemma6}, $M \ba a,b / u$ has at most two $N$-essential elements, and, by \cref{mayclaim}, $M \ba a,b \ba u$ has at most one $N$-essential element. 
  By \cref{thegrandfantasy}, for every $b' \in B-\{x,y\}$, the element $b'$ is $N$-essential in either $M \ba a,b \ba u$ or $M \ba a,b / u$.

  Suppose $r(M) \ge 6$.
  Then $|B-\{x,y\}| \ge 4$, so either $M \ba a,b \ba u$ has at least two $N$-essential elements, or $M \ba a,b / u$ has at least three $N$-essential elements, a contradiction.
  So $r(M) \le 5$.
  Moreover, \cref{wmautfutffrag} implies that $r(M) \ge 5$.
  So $r(M)=5$, $M \ba a,b / u$ has precisely two $N$-essential elements, and $M \ba a,b \ba u$ has precisely one $N$-essential element. 
  Let $M''$ be the matroid given by \cref{wmautfutffrag}.
  Then $r(M'') = 4$ and $M''$ has an $N$-essential element, so $|E(M'')| \le 9$ by \cref{utfutfessentialrank4}, implying $|E(M)| \le 14$, a contradiction.
\end{proof}

\section{Fragile matroids appearing in an excluded minor}
\label{fragilepropssec}

Suppose that $M$ is an excluded minor for the class of $\mathbb{P}$-representable matroids, where $\mathbb{P} \in \{\mathbb{H}_5,\mathbb{U}_2\}$, and $M \ba a,b$ is a $3$-connected 
matroid with a $\utfutf$-minor, for some distinct $a,b \in E(M)$.
By \cref{utfutffragile}, if $|E(M)| \ge 16$, then $M \ba a,b$ is $\utfutf$-fragile.
In this section we consider further properties of such a $\utfutf$-fragile matroid $M \ba a,b$.

We work under the following hypotheses throughout this section.
Let $M$ be an excluded minor for the class of $\mathbb{P}$-representable matroids where $\mathbb{P} \in \{\mathbb{H}_5,\mathbb{U}_2\}$.
Let $M\ba a,b$ be a $3$-connected $\utfutf$-fragile matroid with rank and corank at least~$4$, for distinct $a,b \in E(M)$.
Let $N \in \utfutf$; then $N$ is a non-binary $3$-connected strong $\mathbb{P}$-stabilizer by \cref{u2stabs}, and $M \ba a,b$ is $N$-fragile by \cref{utfutfequiv}.
By \cref{fragilecase}, there exists a bolstered basis~$B$ for $M$ and a $B \times B^*$ companion $\mathbb{P}$-matrix~$A$ for which $\{x,y,a,b\}$ incriminates $(M,A)$ where $\{x,y\} \subseteq B$ and $\{a,b\} \subseteq B^*$, and $M\del a,b$ has at most one $(N,B)$-robust element outside of $\{x,y\}$, where if such an element $u$ exists, then $u\in B^{*}-\{a,b\}$ is an $(N,B)$-strong element of $M\del a,b$, and $\{u,x,y\}$ is a coclosed triad of $M\del a,b$.

\begin{lemma}
  \label{lemmaC}
  Suppose that $|E(M)| \ge 15$.  Then
  \begin{enumerate}[label=\rm(\Roman*)]
    \item $M\ba a,b$ has an $\{X_8,Y_8,Y_8^*\}$-minor,
    \item $M\ba a,b$ has a nice path description, and
    \item for all $e \in E(M \ba a,b)$, exactly one of $M \ba a,b\ba e$ and $M \ba a,b/e$ has a $\utfutf$-minor.
  \end{enumerate}
\end{lemma}
\begin{proof}
  We can begin by applying \cref{ccmwvz-result} to $M \ba a,b$.
  If (i) holds, then the \lcnamecref{lemmaC} holds by \cref{nicepathdescription,noessential}.
  So one of (iii)--(v) holds, and $M\ba a,b$ or $(M \ba a,b)^*$ can be obtained by gluing up to three wheels to $U_{2,5}$ or $Y_8 \ba 4$.
  Each glued wheel corresponds to a fan of $M \ba a,b$, by \cref{cmvzw-fans}.
  Each of these fans has at most five elements, by \cref{fragilefanscase}.
  So if (iii) holds, then $|E(M \ba a,b)| \le 9$, a contradiction.
  Similarly, if (iv) holds, then $|E(M \ba a,b)| \le 10$, a contradiction.

  So we may assume that (v) of \cref{ccmwvz-result} holds.
  Then $M \ba a,b$ or its dual can be obtained from $U_{2,5}$ with ground set $\{x_1,x_2,x_3,x_4,x_5\}$ by gluing wheels to $(x_1,x_3,x_2)$, $(x_1,x_4,x_2)$, and $(x_1,x_5,x_2)$, and each of the resulting fans has size at most five. 
  Let $F$ be one of these fans of $M \ba a,b$ with $|F| = 5$.
  By \cref{noMK4conn}, if $(f_1,f_2,\dotsc,f_5)$ is a fan ordering of $F$, then $\si(M/f_3)$ is $3$-connected.

  Now suppose $F = (f_1,f_2,f_3,f_4,f_5)$ and $F' = (f'_1,f'_2,f'_3,f'_4,f'_5)$ are distinct $5$-element fans obtained by gluing wheels.  We claim that 
  $\{f_2,f_3,f_4\}$ and $\{f'_2,f'_3,f'_4\}$ each contain at least one element that is not $N$-essential.
  Without loss of generality, $F$ is the fan obtained by gluing a wheel to $(x_1,x_3,x_2)$, and $F'$ is the fan obtained by gluing a wheel to $(x_1,x_4,x_2)$.
  Let $F'' = E(M) -(F \cup F')$.
  There is a $\utfutf$-fragile minor~$M'$ of $M\ba a,b$, obtained by deleting or contracting elements of $F''$, such that $M'$ can be obtained from $U_{2,5}$, with ground set $\{x_1,x_2,\dotsc,x_5\}$, by gluing one wheel to $(x_1,x_3,x_2)$ and gluing a second wheel to $(x_1,x_4,x_2)$.
  If at most one of $x_1$ and $x_2$ is in the remove set when gluing these two wheels, then each fan has at most two $\utfutf$-essential elements.
  So we may assume that both $x_1$ and $x_2$ are removed as part of the operation of gluing these two wheels.
  By contracting $f_1$ and $f_5$ from $M'$, we obtain a $\utfutf$-fragile matroid where each element of $F'$ is not $\utfutf$-essential; whereas by contracting $f'_1$ and $f'_5$ we obtain a $\utfutf$-fragile matroid where each element of $F$ is not $\utfutf$-essential.
  This proves the claim.

  Now, by \cref{fragilefanscase}, each fan of $M \ba a,b$ has at most five elements, and if a fan has size five, then it contains $\{x,y\}$.
  So at most one of the three fans has size five.
  Observe that the size of each of these three fans of $M \ba a,b$ has the same parity, due to how the wheels are glued to $U_{2,5}$.
  Thus, if $M \ba a,b$ has a $5$-element fan, then $|E(M \ba a,b)| \le 11$; whereas if each fan of $M \ba a,b$ has size at most four, then $|E(M \ba a,b)| \le 12$.
  Either case is contradictory, so this completes the proof.
\end{proof}

\begin{lemma}
  \label{atmostonetri}
  Suppose that $|E(M)| \ge 15$.
  Then $M \ba a,b$ has at most one triangle, and if such a triangle~$T$ exists, then 
  \begin{enumerate}
    \item $M \ba a,b$ has an $(N,B)$-robust element $u \in B^*$ that is in a coclosed triad $\{u,x,y\}$,
    \item $T$ contains $u$ and either $x$ or $y$, and
    \item $T \cup \{x,y\}$ is a $4$-element fan.
  \end{enumerate}
\end{lemma}
\begin{proof}
  Let $T$ be a triangle of $M'=M \ba a,b$.
  By \cref{lemmaC,pathdesctris},
  the triangle~$T$ either consists of three $N$-deletable elements, or two $N$-deletable elements and an $N$-contractible element.
  In the former case, there exists an element in $T \cap B^*$, as $r(T) = 2$, and this element is $(N,B)$-robust. 
  So $M'$ has an $(N,B)$-robust element $u$ outside of $\{x,y\}$, with $u \in T$.
  Now $u$ is in a triad $\{u,x,y\}$.
  By orthogonality, one of $x$ and $y$ is in $T$, and the other is not, since $M'$ is $3$-connected.
  In particular, $T \cup \{x,y\}$ is a $4$-element fan, as required.

  We may now assume that $T$ consists of two elements that are $N$-deletable in $M'$, and one that is $N$-contractible.
  First, suppose $\{x,y\} \subseteq T$.
  Let $T= \{x,y,p\}$.  Then $p \in B^*$.
  Consider the case when $M$ has an $(N,B)$-robust element~$u$.
  Note that $p \neq u$, since $\{u,x,y\}$ is a triad and $M$ is $3$-connected.
  Moreover, $\co(M' \ba u)$ is $3$-connected, but this matroid is isomorphic to $\co(M' \ba u/x)$, which has a parallel pair $\{y,p\}$, a contradiction.
  So $M'$ has no $(N,B)$-robust elements.
  Now, by the definition of a bolstered basis, no allowable pivot can introduce an $(N,B)$-robust element.
  If $p$ is $N$-deletable, then it is $(N,B)$-robust, a contradiction.
  So without loss of generality we may assume that $p$ is the $N$-contractible element of $T$; thus $x$ and $y$ are $N$-deletable.
  But a pivot on $A_{xp}$ 
  is allowable, by \cref{allowablexyrow2}, where $\{p,y,a,b\}$ incriminates $(M,A^{xp})$, and $x$ is an $(N,B\triangle \{x,p\})$-robust element outside of $\{p,y\}$, contradicting that $B$ is bolstered.

  Next, suppose $|\{x,y\} \cap T| = 1$.
  Without loss of generality, $x \in T$ and $y \notin T$.
  Suppose $M'$ has no $(N,B)$-robust elements.
  There is at least one $N$-deletable element in $T-x$.  Let $q$ be such an element; then $q \in B$.  Now $T=\{x,p,q\}$ where $p \in B^*$ and $p$ is $N$-contractible, since $M'$ has no $(N,B)$-robust elements, so $x$ is $N$-deletable.
  By \cref{allowablexyrow2}, a pivot on $A_{xp}$ is allowable, where $\{p,y,a,b\}$ incriminates $(M,A^{xp})$.
  But now $x$ is an $(N,B \triangle \{x,p\})$-robust element outside of $\{p,y\}$, which contradicts that $B$ is a bolstered basis.
  So we may assume that $M'$ has an $(N,B)$-robust element~$u$, in which case $\{u,x,y\}$ is a triad.
  By orthogonality, $u \in T$, so let $T = \{x,u,q\}$.
  Then $\{x,y,u,q\}$ is a $4$-element fan as required.

  Finally, suppose $x,y \notin T$.
  Recall that if $M'$ has an $(N,B)$-robust element~$u$ outside of $\{x,y\}$, then $\{u,x,y\}$ is a triad.
  Thus, if such a $u$ exists, then by orthogonality $u \notin T$.
  So $T$ does not contain an $(N,B)$-robust element.
  Thus, the two $N$-deletable elements of $T$ are in $B$ and the $N$-contractible element is in $B^*$.
  Let $\alpha,\beta \in B$ be the $N$-deletable elements of $T$ and let $\gamma \in B^*$ be the $N$-contractible element of $T$.

  By \cref{lemmaC}, $M'$ has a nice path description $(P_1,P_2,\dotsc,P_m)$.
  We claim that either $P_1$ or $P_1 \cup T$ is a maximal $4$-element fan with $\gamma$ as an internal element.
  By \cref{allowablenonxy2}, a pivot on $A_{\alpha\gamma}$ is allowable, after which $\alpha$ and $\gamma$ become $(N,B')$-robust elements, where $B' =B\triangle \{\alpha,\gamma\}$, with $\alpha \in (B')^*$ and $\gamma \in B'$.
  Now $\gamma \in P_i$, for some $i \in \seq{m}$, where,
  by \cref{pathdescprops}, either $P_i$ is a coguts set, or $i \in \{1,m\}$.
  In the former case, $\si(M' / \gamma)$ is $3$-connected, again by \cref{pathdescprops}, in which case $\gamma$ is an $(N,B')$-strong element in $B'-\{x,y\}$, contradicting \cref{fragilecase}\ref{nostronginbasis}.
  So, without loss of generality, $\gamma \in P_1$.
  If $P_1$ is a triangle or $4$-segment, then $\si(M' / \gamma)$ is $3$-connected by \cref{pathdescendconn}, so again $\gamma$ is an $(N,B')$-strong element in $B'-\{x,y\}$, a contradiction to \cref{fragilecase}\ref{nostronginbasis}.
  If $P_1$ is a $4$-cosegment, then by orthogonality $T \subseteq P_1$, so $T$ is a triangle-triad, contradicting that $M'$ is $3$-connected.
  Suppose $P_1$ is a fan of size at least~$4$.
  Since $\gamma$ is $N$-contractible, \cref{fragilefanelements} implies that $\gamma$ is not a spoke element of the fan~$P_1$.
  If $\gamma$ is a rim element, then $\si(M' / \gamma)$ is $3$-connected by \cref{fanends,noMK4conn}, a contradiction to \cref{fragilecase}\ref{nostronginbasis}.
  So $\gamma$ is an internal element of $P_1$, where $P_1$ is a maximal $4$-element fan, as claimed.
  Finally, if $P_1$ is a triad, then $F=P_1 \cup T$ is a $4$-element fan, by orthogonality.
  As in the previous case, $\gamma$ is not a spoke or a rim element of $F$, so $F$ is a maximal $4$-element fan with $\gamma$ as an internal element.

  Now let $F = P_1 \cup T$ if $P_1$ is a triad, otherwise let $F = P_1$; in either case, $F$ is a maximal $4$-element fan with $\gamma$ as an internal element.
  By orthogonality and the maximality of $F$, we have $T \subseteq F$.
  So, without loss of generality, $F$ has a fan ordering $(\alpha, \gamma, \beta, \delta)$, where $P_1 - T = \{\delta\}$.
  Note that the only triangle containing $\gamma$ is $T$, and the only triad containing $\beta$ is $\{\gamma,\beta,\delta\}$, by orthogonality and the maximality of $F$.
  Thus $\co(M \ba \beta) \cong M \ba \beta / \gamma \cong \si(M / \gamma)$.
  By \cref{allowablenonxy2}, a pivot on $A_{\beta\gamma}$ is allowable, after which $\beta$ and $\gamma$ become $(N,B'')$-robust elements, where $B'' =B\triangle \{\beta,\gamma\}$, with $\beta \in (B'')^*$ and $\gamma \in B''$.
  Recall that $\si(M/\gamma)$ is not $3$-connected, by \cref{fragilecase}\ref{nostronginbasis}, so $\co(M \ba \beta)$ is not $3$-connected.
  Thus both $\beta$ and $\gamma$ are $(N,B'')$-robust but not $(N,B'')$-strong.
  As neither $\beta$ nor $\gamma$ is $(N,B)$-robust, this contradicts that $B$ is a bolstered basis.
\end{proof}

\begin{lemma}
  \label{rkcorkbounds}
  Suppose that $|E(M)| \ge 15$.  Then 
  \begin{enumerate}
    \item $r(M\ba a,b) \le r^*(M \ba a,b) + 2$ and
    \item $r^*(M\ba a,b) \le r(M\ba a,b) + 1$.
  \end{enumerate}
  Moreover, if $M \ba a,b$ has an $(N,B)$-robust element outside of $\{x,y\}$, then $r(M\ba a,b) \le r^*(M \ba a,b) + 1$.
\end{lemma}
\begin{proof}
  By \cref{lemmaC}, every element of $M \ba a,b$ is $N$-deletable or $N$-contractible (but not both).
  Let $r = r(M \ba a,b)$ and $r^* = r^*(M \ba a,b)$, and let $C$ and $D$ be the set of $N$-contractible and $N$-deletable elements of $M \ba a,b$ respectively.
  Recall that $M \ba a,b$ has at most one $(N,B)$-robust element outside of $\{x,y\}$, and if this element exists it is in $B^*$.
  So each of the $r-2$ elements of $B-\{x,y\}$ are $N$-deletable, $x$ and $y$ might be $N$-deletable, and at most one element in $B^*$ is $N$-deletable.  In total, $|D| \le r+1$. 
  On the other hand, all of the $r^*$ elements of $B^*$ are $N$-contractible when $M\ba a,b$ has no $(N,B)$-robust elements outside of $\{x,y\}$, but none of the elements in $B-\{x,y\}$ are $N$-contractible.  So $|C| \le r^*+2$.
  If $M \ba a,b$ has an $(N,B)$-robust element outside of $\{x,y\}$, then $|C| \le r^*+1$.
  By \cref{pathdescrank}, $r(M\ba a,b)=|C|$, so $r = |C| \le r^*+2$; and $r^*(M\ba a,b) =|D|$, so $r^* = |D| \le r+1$, as required.
\end{proof}

The next two lemmas are used to simplify the arguments in \cref{ndtsec}.

\begin{lemma}
  \label{oneendisatriad}
  Suppose that $|E(M)| \ge 15$ and $M \ba a,b$ has a nice path description $(P_1,P_2,\dotsc,P_m)$.
  Then either $P_1$ or $P_m$ is a coclosed triad.
\end{lemma}
\begin{proof}
  Let $i \in \{1,m\}$ and suppose that $P_i$ is a $4$-cosegment.
  We claim that $P_i \cap \{x,y\} \neq \emptyset$.
  By \cref{pathdescends}, there is some $e \in P_i$ that is $\utfutf$-deletable, and each element in $P_i-e$ is $\utfutf$-contractible.
  Recall that $M \ba a,b$ has at most one $(N,B)$-robust element outside of $\{x,y\}$, where if such an element $u$ exists, then $u \in B^*$ and $\{u,x,y\}$ is a triad of $M \ba a,b$.
  Since $r^*_{M \ba a,b}(P_i-e)=2$, we have $|(P_i-e) \cap B^*| \le 2$, so there is an $(N,B)$-robust element in $(P_i-e) \cap B$.
  So $P_i \cap \{x,y\} \neq \emptyset$ as claimed.

  Towards a contradiction, suppose that $P_1$ and $P_m$ are both $4$-cosegments.
  Then, without loss of generality, $x \in P_1$ and $y \in P_m$.
  Let $e_1 \in P_1$ and $e_m \in P_m$ be $\utfutf$-deletable elements, so,
  letting $T_1^* = P_1-e_1$ and $T_m^* = P_m-e_m$, each element in $T_1^* \cup T_m^*$ is $\utfutf$-contractible.
  Note that $x \in T^*_1$ and $y \in T^*_m$, and $(T_1^* \cup T_m^*)-\{x,y\} \subseteq B^*$.
  Let $Z = E(M \ba a,b)-(T_1^* \cup T_m^*)$.
  Then $Z \cap B = B-\{x,y\}$ and, as $r(Z) \leq r(M \ba a,b)-2$, the set $B-\{x,y\}$ spans $Z$.
  Suppose $M \ba a,b$ has an $(N,B)$-robust element $u \in Z$.  Then $\{u,x,y\}$ is a triad, so $r^*(T_1^* \cup T_2^* \cup u) \le 4$.  But $|(T_1^* \cup T_2^* \cup u) \cap B^*| \ge 5$, a contradiction.
  So $M \ba a,b$ has no $(N,B)$-robust elements.
  If $r^*(M \ba a,b)\le 4$, then, by \cref{rkcorkbounds}, $r(M) \le 6$, so $|E(M)| \le 12$, a contradiction.
  So we may assume $r^*(M \ba a,b) > 4$.
  Thus, there exists an element $q \in B^* \cap Z$ such that $q$ is not $(N,B)$-robust, and $A_{xq} = A_{yq} = 0$.
  Since $q$ is not a loop, there exists an element $p \in B-\{x,y\}$ such that $A_{pq} \neq 0$.
  Now $A^{pq}$ is an allowable pivot, by \cref{allowablenonxy2}, and $q$ is $N$-contractible in $B' = B \triangle \{p,q\}$.
  So $q$ is $(N,B')$-robust, but $M \ba a,b$ has no $(N,B)$-robust elements, contradicting that $B$ is a bolstered basis.
  We deduce that $P_1$ and $P_m$ are not both $4$-cosegments.

  Next, suppose that $P_1$ is a $4$-cosegment, and $P_m$ is a triad that is not coclosed.  Then, by definition, there is an element $p_1 \in P_1$ such that $P_m \cup p_1$ is a $4$-cosegment.
  The $4$-cosegments $P_1$ and $P_m \cup p_1$ each have a unique $\utfutf$-deletable element, whereas the other elements are $\utfutf$-contractible; and contain at most two elements in $B^*$, so at least two elements in $B$.
  Since $M \ba a,b$ has at most one $(N,B)$-robust element, it follows that $p_1$ is $\utfutf$-contractible.
  Moreover, $r^*_{M \ba a,b}(P_1 \cup P_m) =3$, so $|(P_1 \cup P_m) \cap B^*| \le 3$ and $|(P_1 \cup P_m) \cap B| \ge 4$, implying that $p_1 \in B^*$.
  Now, $P_1 - p_1$ and $P_m$ each contain two elements of $B$, at least one of which is $N$-contractible, and therefore $(N,B)$-robust.  So we may assume that $x \in P_1-p_1$ and $y \in P_m$.
  Note that $P_2$ and $P_{m-1}$ are guts sets and $m$ is odd.
  Let $i \in \{2,4,\dotsc,m-1\}$, so that $P_i$ is a guts set.
  Since $\lc^*(P_1,P_m)=1$, it follows from the duals of \cref{growpi,pflancoguts} that $|P_i| = 1$.
  Hence $m \ge 5$.
  Now consider the coguts set $P_3$.  By \cref{pathdescprops}, each $e \in P_3$ is $\utfutf$-contractible, so, as $e \notin \{x,y\}$, we have $e \in B^*$.
  Thus, if $|P_3| \ge 2$, then $r^*_{M \ba a,b}(P_1 \cup P_2 \cup P_3) = 3$ but $|(P_1 \cup P_2 \cup P_3) \cap B^*| \ge 4$, a contradiction.
  So $|P_3| = 1$.
  Now $P_2 \cup P_3 \cup P_4$ is a triangle, by the duals of \cref{growpi,pflantriad}.
  But as $\{x,y\} \subseteq P_1 \cup P_m$, this contradicts \cref{atmostonetri}.
  By symmetry, we deduce that if $P_i$ is a $4$-cosegment and $P_{j}$ is a triad for $\{i,j\} = \{1,m\}$, then $P_{j}$ is coclosed.
  
  By \cref{atmostonetri} and \cref{nicepathdescription}\ref{npd1}, 
  neither $P_1$ nor $P_m$ is a triangle.
  Let $\{i,j\}=\{1,m\}$ and suppose that $P_i$ is a fan of size at least~$4$; then, by \cref{atmostonetri} again, $\{x,y\} \subseteq P_i$, so $P_j$ is a (coclosed) triad.
\end{proof}

\begin{lemma}
  \label{niceends}
  Suppose that $|E(M)| \ge 15$ and $M \ba a,b$ has a nice path description $(P_1,P_2,\dotsc,P_m)$. Let $i \in \{1,m\}$.  Then $P_i$ is either a coclosed cosegment, or a maximal fan.
\end{lemma}
\begin{proof}
  By \cref{oneendisatriad}, we may assume that $P_1$ is a coclosed triad.
  Then $P_2$ is a guts set.  Moreover, for any $p_1 \in P_1$, we have $p_1 \notin \cl(P_m)$, so $P_m$ is closed.

  Suppose $P_m$ is a cosegment that is not coclosed, or $P_m$ is a fan that is not maximal.
  In either case, there is some $p_1 \in P_1 \cap \cocl(P_m)$.
  It follows that $(P_1-p_1,P_2,\dotsc,P_{m-1},\{p_1\},P_m)$ is a path of $3$-separations.
  Hence $p_2 \in \cl(P_1-p_1)$ for each $p_2 \in P_2$, implying $P_1 \cup p_2$ is a fan.
  This contradicts the definition of a nice path description, so we deduce that if $P_m$ is a cosegment, then it is coclosed, and if $P_m$ is a fan, then it is maximal.
\end{proof}



\section{The delete-triple case}
\label{dtsec}

We work under the following assumptions throughout this section.
  Let $M$ be an excluded minor for the class of $\mathbb{P}$-representable matroids where $\mathbb{P} \in \{\mathbb{H}_5,\mathbb{U}_2\}$, and $M$ has no triads.
  Suppose also that $|E(M)| \ge 16$.
  Note that, by \cref{utfutffragile}, for any pair $\{a,b\} \subseteq E(M)$ such that $M \ba a,b$ is $3$-connected with a $\utfutf$-minor, the matroid $M \ba a,b$ is $\utfutf$-fragile.

We say that a triple $\{a,b,c\} \subseteq E(M)$ is a \emph{delete triple} for $M$ if $M \ba a,b,c$ is $3$-connected with a 
$\utfutf$-minor.
In this section, we prove \cref{deltripleprop}, which says that, under the above assumptions, $M$ has no delete triples.

\begin{lemma}
  \label{specialdeletetriples}
  If $M$ has a delete triple, then it has some delete triple $\{a,b,c\}$ such that $M \ba a,b,c$ has no triangles. 
\end{lemma}
\begin{proof}
  Suppose $\{a,b,e\}$ is a delete triple for $M$ but $M \ba a,b,e$ has at least one triangle.
  Observe that $M \ba a,b$ is $\utfutf$-fragile and has rank and corank at least~$4$, by \cref{utfutffragile}, and this matroid has at least one triangle.
  Since $M \ba a,b,e$ has a $\utfutf$-minor, $M \ba a,b,e$ is also $\utfutf$-fragile.
  Let $N \in \utfutf$ such that $M\ba a,b$ has an $N$-minor.
  By \cref{fragilecase}, there exists a basis~$B$ for $M$ and a $B \times B^*$ companion $\mathbb{P}$-matrix~$A$ for which $\{x,y,a,b\}$ incriminates $(M,A)$ where $\{x,y\} \subseteq B$ and $\{a,b\} \subseteq B^*$. 
  By \cref{atmostonetri}, $M \ba a,b$ has exactly one triangle~$T$, there is a unique $(N,B)$-robust element $u \in T \cap B^*$, and $T \cup \{x,y\}$ is a $4$-element fan~$F$.
  If the fan~$F$ is maximal, then there is a spoke end $c$ of $F$. 
  Now $c$ is not $\utfutf$-contractible in $M \ba a,b$, by \cref{fragilefanelements}, so it is $\utfutf$-deletable, by \cref{lemmaC},
  and $M \ba a,b,c$ is $3$-connected, by \cref{fanendsstrong}.
  So $\{a,b,c\}$ is a delete triple such that $M \ba a,b,c$ has no triangles, as required.

  Now we may assume that $F$ is properly contained in a fan $F'$.  By \cref{fragilefanscase}, $|F'| =5$, and, by \cref{atmostonetri}, we may assume up to swapping $x$ and $y$ that $F'$ has an ordering $(x,u,y,f_4,f_5)$ where $\{x,u,y\}$ is a triad.
  Note that $e \notin F'$, since each element of $F'$ is in a triad of $M \ba a,b$ but $M \ba a,b,e$ is $3$-connected.
  So the triangle $\{u,y,f_4\}$ of $M \ba a,b$ is also a triangle of $M \ba a,b,e$.
  Moreover, by \cref{fanunique}, $e$ is not in the coclosure of either of the triads of $F'$, so $F'$ is also a fan of $M \ba a,b,e$.
  Again by \cref{fanunique}, the only triads containing $y$ in $M \ba a,b,e$ are $\{u,x,y\}$ and $\{y,f_4,f_5\}$.

  Since $M$ has no triads, either $a$ or $b$ blocks the triad $\{u,x,y\}$ of $M \ba a,b$.
  Without loss of generality, say $a$ blocks $\{u,x,y\}$.
  Then, $\{u,x,y\}$ is not a triad in $M \ba b,e$.
  Furthermore, as $\{u,y,f_4\}$ is a triangle in $M \ba a,b,e$, it is also a triangle in $M \ba b,e$.
  By \cref{utfutffragile,atmostonetri}, this is the unique triangle in $M \ba b,e$, and it is contained in a $4$-element fan.  
  Applying the argument from the first paragraph, if this $4$-element fan is maximal, then there is a delete triple, $\{b,e,c'\}$ say, such that $M \ba b,e,c'$ has no triangles, as required.
  So we may assume that $M \ba b,e$ has a $5$-element fan $F_2'$ whose internal elements are $\{u,y,f_4\}$.
  Now $y$ is in at least one triad of $M \ba b,e$.
  Since $y$ is in exactly two triads of $M \ba a,b,e$, at least one of which is blocked by $a$, the unique triad of $M \ba b,e$ containing $y$ is $\{y,f_4,f_5\}$.
  So 
  $f_4$ is a rim element of $F_2'$, and 
  $\{u,f_4\}$ is contained in a triad $T^*$ of $M \ba b,e$.
  Since $M \ba a,b,e$ is $3$-connected, $T^*$ is also a triad of $M \ba a,b,e$.
  Then it follows that $T^* = \{u,f_4,q\}$ for some $q \in E(M \ba a,b) - \{e,u,x,f_4,f_5\}$.
  Now $(f_5,y,f_4,u,q)$ is a fan ordering of $F_2'$, which is a fan in $M \ba b,e$ and $M \ba a,b,e$.
  Note that $x,q,f_5 \in \cocl_{M \ba a,b,c}(\{u,y,f_4\})$, and it follows that $\{x,q,f_5\}$ is also a triad of $M \ba a,b,c$.
  Thus $(M\ba a,b,c)^*|(F' \cup q) \cong M(K_4)$. 
  But $M \ba a,b,c$ is $\utfutf$-fragile, so this contradicts \cref{noMK4}.
\end{proof}

We say that a delete triple $\{a,b,c\}$ for $M$ is \emph{special} if $M \ba a,b,c$ has no triangles.

The next lemma is straightforward, but important for the arguments that follow.
\begin{lemma}
  \label{pathbasics}
  Let $M'$ be a $3$-connected matroid with $x \in E(M')$.
  Suppose $M' \ba x$ is $3$-connected, and both $M'$ and $M'\ba x$ have path width three.
  Let $(e_1,e_2,\dotsc,e_n)$ be a sequential ordering of $M'$, with $x = e_i$ for some $i \in \seq{n}$.
  Then $\sigma=(e_1,\dotsc,e_{i-1},e_{i+1},\dotsc,e_n)$ is a sequential ordering of $M' \ba x$.
  Moreover, any triad of $M' \ba x$ contained in either $\{e_1,\dotsc,e_{i-1}\}$ or $\{e_{i+1},\dotsc,e_n\}$ is not blocked by $x$.
\end{lemma}
\begin{proof}
  For $j \in \seq{n-1}$, let $X_j = \{e_1,e_2,\dotsc,e_j\}$ and $Y_j = \{e_{j+1},e_{j+2},\dotsc,e_n\}$, and $X'_j = X_j - x$ and $Y'_j = Y_j - x$.
  Suppose $|X'_j|,|Y'_j|\ge 2$ for some $j \in \seq{n-1}$.
  To show that $\sigma$ is a sequential ordering of $M' \ba x$, it suffices to show that $\lambda_{M' \ba x}(X'_j) = 2$.
  Since $M' \ba x$ is $3$-connected, $\lambda_{M' \ba x}(X'_j) \ge 2$.
  Moreover, $\lambda_{M' \ba x}(X'_j) = r(X'_j) + r(Y'_j) - r(M' \ba x) \le r(X_j) + r(Y_j) - r(M') = \lambda_{M'}(X_j) = 2$, as required.

  Now suppose that $\{e_1,\dotsc,e_{i-1}\}$ contains a triad $T^*$ of $M' \ba x$.
  It remains to prove that $T^*$ is not blocked by $x$.
  First, assume that $r^*_{M'}(\{e_1,\dotsc,e_{i-1}\}) \ge 3$ and $r^*_{M'}(\{e_{i+1},\dotsc,e_n\}) \ge 3$.
  Then, by the dual of \cref{vert3sep}, $x$ is not a coguts element, since $M' \ba x$ is $3$-connected.
  So $x$ is a guts element, in which case $x \in \cl(\{e_{i+1},\dotsc,e_n\}) \subseteq \cl(E(M' \ba x) - T^*)$, implying that $x$ does not block $T^*$.
  So we may assume that $r^*_{M'}(\{e_1,\dotsc,e_{i-1}\}) \le 2$ or $r^*_{M'}(\{e_{i+1},\dotsc,e_n\}) \le 2$.
  In the former case, $x \notin \cocl_{M'}(\{e_1,\dotsc,e_{i-1}\})$, for otherwise $M \ba x$ is not $3$-connected, so $x \in \cl(\{e_{i+1},\dotsc,e_n\})$ by orthogonality.
  In the latter case, $x \notin \cocl_{M'}(\{e_{i+1},\dotsc,e_n\})$, similarly, and  thus, as $(e_1,\dotsc,e_n)$ is a sequential ordering of $M'$, we must have $x \in \cl(\{e_{i+1},\dotsc,e_n\})$.
  Since $x \in \cl(\{e_{i+1},\dotsc,e_n\})$ in either case, $x$ does not block $T^*$.
\end{proof}

We come to the main result of this section.  For ease of reference, we restate the section assumptions.
\begin{theorem}
  \label{deltripleprop}
  Let $M$ be an excluded minor for the class of $\mathbb{P}$-representable matroids where $\mathbb{P} \in \{\mathbb{H}_5,\mathbb{U}_2\}$, and $M$ has no triads.
  Suppose $|E(M)| \ge 16$.
  Then $M$ has no delete triples.
\end{theorem}
\begin{proof}
  Towards a contradiction, suppose $M$ has a delete triple. 
By \cref{specialdeletetriples}, we may assume that $M$ has a special delete triple $\{a,b,c\}$.
Then, by \cref{utfutffragile}, each of $M \ba a,b$, $M \ba a,c$, and $M \ba b,c$ is $\utfutf$-fragile.
By \cref{lemmaC}, each of these matroids has an $\{X_8,Y_8,Y_8^*\}$-minor, a nice path description, and no $\utfutf$-essential elements.
Since $M \ba a,b$ is $\utfutf$-fragile, and $c$ is $\utfutf$-deletable in this matroid, $M \ba a,b,c$ is also $\utfutf$-fragile.

\begin{claim}
  \label{mabcfragile}
  $M \ba a,b,c$ has path width three and no $\utfutf$-essential elements.
\end{claim}
\begin{proof}
  The matroid $M \ba a,b,c$ is a $3$-connected $\utfutf$-fragile $\mathbb{P}$-representable matroid, with $|E(M \ba a,b,c)| \ge 12$.
  Moreover, $M \ba a,b,c$ has no triangles, so it has no fans of size at least~$4$.
  Thus, by \cref{ccmwvz-result}, $M \ba a,b,c$ has an $\{X_8,Y_8,Y_8^*\}$-minor.
  Now $M \ba a,b,c$ has path width three, by \cref{nicepathdescription}, and no $\utfutf$-essential elements, by \cref{noessential}.
\end{proof}

Let $L$ and $R$ be the ends of a sequential ordering of $M \ba a,b,c$.
Note that, by \cref{welldefinedends}, for every sequential ordering $\sigma$ of $M \ba a,b,c$, we have $\{L,R\} = \{L(\sigma),R(\sigma)\}$.
Since $M \ba a,b,c$ has no triangles, each of $L$ and $R$ is either a triad or a $4$-cosegment.

\begin{claim}
  \label{only4cosegs}
  $L$ and $R$ are $4$-cosegments.
\end{claim}
\begin{subproof}
  Say $L$ 
  is a triad.
  Since $M$ has no triads, $L$ is blocked by at least one of $a$, $b$, and $c$.
  Without loss of generality we may assume that $a$ blocks $L$.
  Consider a sequential ordering $\sigma_a=(p_1,p_2,p_3,\dotsc,p_n)$ for $M \ba b,c$.
  Now $a=p_i$ for some $i \in \seq{n}$.
  So $\sigma_a^-=(p_1,p_2,\dotsc,p_{i-1},p_{i+1},\ldots,p_n)$ is a sequential ordering for $M \ba a,b,c$ by \cref{pathbasics}.
  By \cref{welldefinedends}, we may assume (up to reversing the order of $\sigma_a^-$) that $L=L(\sigma_a^-)$ and $R=R(\sigma_a^-)$.

  Suppose $i > 3$. Then $L=\{p_1,p_2,p_3\}$ is a triad of $M \ba a,b,c$ by \cref{endslipperiness}\ref{esi}, and, as $\sigma_a$ is a sequential ordering for $M \ba b,c$, the set $\{p_1,p_2,p_3\}$ is either a triangle or a triad of $M \ba b,c$.
  If $\{p_1,p_2,p_3\}$ is a triangle of $M \ba b,c$, then it is also a triangle of $M \ba a,b,c$, so $M \ba a,b,c$ is not $3$-connected, a contradiction.
  So $\{p_1,p_2,p_3\}$ is a triad of $M \ba b,c$ and $M \ba a,b,c$, in which case $a$ does not block $L$, a contradiction.
  We deduce that $a = p_i$ for $i \in \{1,2,3\}$.

  Now, if $\{p_1,p_2,p_3\}$ is a triad of $M \ba b,c$, then $M \ba a,b,c$ is not $3$-connected, a contradiction.
  So $\{p_1,p_2,p_3\}$ is a triangle of $M \ba b,c$.
  As $L = L(\sigma_a^-)$ is a triad of $M \ba a,b,c$, \cref{endslipperiness}\ref{esi} implies that this triad is $\{p_1,p_2,p_3,p_4\}-p_i$.
  As the triad $L$ is blocked by $a$, we have that $\{p_1,p_2,p_3,p_4\}$ is a cocircuit of $M \ba b,c$.
  By \cref{atmostonetri}, the triangle $\{p_1,p_2,p_3\}$ is contained in a $4$-element fan $F$ of $M \ba b,c$.
  Since $M \ba a,b,c$ is $3$-connected, $a$ is not contained in the triad of $F$.
  So, for some element $z$ and $\{i,j,k\} = \{1,2,3\}$, the fan $F$ has ordering $(a,p_j,p_k,z)$ where $\{p_j,p_k,z\}$ is a triad.
  Note that $p_4 \neq z$, since in $M \ba b,c$ the set $\{p_j,p_k,p_4\}$ is properly contained in a cocircuit, whereas $\{p_j,p_k,z\}$ is a triad.
  But then $\{p_j,p_k,p_4,z\}$ is a $4$-cosegment of $M \ba a,b,c$ containing $L$, so $L$ is not a triad end of $M \ba a,b,c$, a contradiction.

  So $L$ is a cosegment of size at least~$4$.  The fact that $|L| = 4$ follows from the fact that $M \ba a,b,c$ is $\utfutf$-fragile.
  The result then follows by symmetry.
\end{subproof}

By \cref{only4cosegs}, we may now assume that $|L|=4$ and $|R|=4$.

\begin{claim}
  \label{unblockedtriad}
  For each $x \in \{a,b,c\}$ and $X \in \{L,R\}$, the element $x$ does not block every triad contained in $X$.
\end{claim}
\begin{subproof}
  It suffices to show that $a$ does not block every triad contained in $L$. 
  Consider a sequential ordering $\sigma_a=(p_1,p_2,p_3,\dotsc,p_n)$ for $M \ba b,c$.
  We have $a=p_i$ for some $i \in \seq{n}$, and $$\sigma_a^-=(p_1,p_2,\dotsc,p_{i-1},p_{i+1},\ldots,p_n)$$ is a sequential ordering for $M \ba a,b,c$ by \cref{pathbasics}.
  By reversing $\sigma_a^-$, if necessary, we may assume that $L(\sigma_a^-) = L$ and $R(\sigma_a^-) = R$, due to \cref{welldefinedends}.

  Suppose $i > 3$.
  Then $\{p_1,p_2,p_3\}$ is a triad of $M \ba a,b,c$, by \cref{endslipperiness}\ref{esii}, and $\{p_1,p_2,p_3\}$ is either a triangle or a triad of $M \ba b,c$.
  However, if $\{p_1,p_2,p_3\}$ is a triangle of $M \ba b,c$, then it is a triangle-triad in $M \ba a,b,c$, contradicting $3$-connectivity.
  So $\{p_1,p_2,p_3\}$ is a triad of $M \ba b,c$ and $M \ba a,b,c$.
  Then $\{p_1,p_2,p_3\} \subseteq L$ and $\{p_1,p_2,p_3\}$ is a triad that is not blocked by $a$, as required.

  So we may assume $a=p_i$ for some $i \le 3$.
  Now, if $\{p_1,p_2,p_3\}$ is a triad of $M \ba b,c$, then $M \ba a,b,c$ is not $3$-connected, a contradiction.
  So $\{p_1,p_2,p_3\}$ is a triangle.
  As $L = L(\sigma_a^-)$ is a $4$-cosegment of $M \ba a,b,c$, \cref{endslipperiness}\ref{esii} implies that $\{p_1,p_2,p_3,p_4\}-p_i$ is a triad~$T^*$ contained in $L$.
  We may assume that $\{p_1,p_2,p_3,p_4\}$ is a cocircuit of $M \ba b,c$, for otherwise $T^*$ is a triad contained in $L$ that is not blocked by $a$, as required.
  By \cref{atmostonetri}, the triangle $\{p_1,p_2,p_3\}$ is contained in a $4$-element fan $F$ of $M \ba b,c$.
  Since $M \ba a,b,c$ is $3$-connected, $a$ is not contained in the triad of $F$.
  So, for some element $z$ and $\{i,j,k\} = \{1,2,3\}$, the fan $F$ has ordering $(a,p_j,p_k,z)$ where $\{p_j,p_k,z\}$ is a triad.
  Note that $p_4 \neq z$, since, in $M \ba b,c$ the set $\{p_j,p_k,p_4\}$ is properly contained in a cocircuit, whereas $\{p_j,p_k,z\}$ is a triad.
  But then $\{p_j,p_k,p_4,z\}$ is a $4$-cosegment of $M \ba a,b,c$, and it follows that $L = \{p_j,p_k,p_4,z\}$.
  But then $\{p_j,p_k,z\}$ is a triad contained in $L$ that is not blocked by $a$, as required.
\end{subproof}

\begin{claim}
  \label{applyphp}
  Some $x \in \{a,b,c\}$ blocks a triad in $L$ and a triad in $R$.
\end{claim}
\begin{subproof}
  Every triad of $L$, and every triad of $R$, is blocked by at least one of $a$, $b$, and $c$, since $M$ has no triads.
  Without loss of generality, $a$ blocks one of the four triads of $L$.
  By \cref{unblockedtriad}, one of the other three triads of $L$ is not blocked by $a$; without loss of generality, there is a triad of $L$ blocked by $b$.
  By \cref{unblockedtriad}, at least one triad of $R$ is not blocked by $c$.
  So some triad of $R$ is blocked by $a$ or $b$, and hence \cref{applyphp} holds.
\end{subproof}

By \cref{applyphp}, we may assume that $c$ blocks a triad in $L$ and a triad in $R$.
Consider a sequential ordering $\sigma_c=(p_1,p_2,\dotsc,p_n)$ for $M \ba a,b$, where $c=p_{i_c}$ for some $i_c \in \seq{n}$.
Then $\sigma_c^-=(p_1,p_2,\dotsc,p_{i_c-1},p_{i_c+1},\ldots,p_n)$ is a sequential ordering for $M \ba a,b,c$ by \cref{pathbasics}.

We now break into two cases depending on whether or not $L$ and $R$ are disjoint.
We first consider the case where $L$ meets $R$.

\begin{claim}
  \label{singlecoflanoutcome}
  The ends $L$ and $R$ are disjoint.
\end{claim}
\begin{subproof}
  Towards a contradiction, suppose $L$ meets $R$.
  If $|L \cap R| \ge 2$, then $L \cup R$ is a cosegment of $M \ba a,b,c$, and it follows that $r^*(M \ba a,b,c) = 2$.
  But then, as $M \ba a,b,c$ is $\mathbb{P}$-representable, $M \ba a,b,c$ is isomorphic to a minor of $U_{4,6}$, implying $|E(M)| \le 9$, a contradiction.
  So we may assume that $|L \cap R| = 1$.  Let $L \cap R = \{s\}$ and $s = p_{i_s}$.

  Suppose that $i_c \le 3$ up to reversing the ordering of $\sigma_c$.
  Without loss of generality, $i_c = 3$.
  Then $\{p_1,p_2,p_4\} \subseteq L$, by \cref{endslipperiness}\ref{esii}.
  If $i_s > i_c$, then, by \cref{pathbasics}, $c$ does not block any triad contained in $R$, a contradiction.
  So $i_s \in \{1,2\}$.
  Now $L = \left\{p_1,p_2,p_4,p_{i_L}\right\}$ for some $i_L \ge 5$.
  Since $p_{i_L} \in \cocl_{M \ba a,b}(\{p_1,p_2,p_4\})$, we may also assume that $i_L = 5$.
  Let $(P_1,\dotsc,P_m)$ be the guts-coguts concatenation of $\sigma_c$ with ends $P_1 = \{p_1,p_2,c\}$ and $P_m = \{p_{n-2},p_{n-1},p_n\}$, where $s \in \{p_1,p_2\}$.
  Then $P_1$ is a triangle, $P_m$ is a triad, $P_2$ is a coguts set containing $\{p_4,p_{i_L}\}$, and $P_{m-1}$ is a guts set.
  Since $p_{i_s} \in \cocl_{M \ba a,b}(\{c,p_{n-2},p_{n-1},p_n\})$, we have $r^*_{M \ba a,b}(P_m \cup \{c,p_{i_s}\})=3$.
  Thus $\lc^*(\{c,s\},P_m) = 1$, and $\lc^*(P_1,P_m) \ge 1$ by the dual of \cref{growpi}.
  Now, also using the dual of \cref{pflancoguts}, if $P_j$ is a guts set for $3 \le j \le m-1$, then $|P_j| = 1$.
  Moreover, by the dual of \cref{pflantriad}, if $P_j$ is a coguts set with $|P_j|=1$ and $3 < j < m-1$, then $P_{j-1} \cup P_j \cup P_{j+1}$ is a triangle of $M \ba a,b$.
  As this triangle avoids $c$, it is also a triangle of $M \ba a,b,c$, contradicting that $\{a,b,c\}$ is a special delete triple.
  So any coguts set $P_j$ with $3 < j < m-1$ has $|P_j| \ge 2$.

  Observe that $m$ is even, $P_j$ is a guts set for each odd $j>1$, and $P_j$ is a coguts set for each even $j < m$.
  It follows that $r(M \ba a,b) = 3+q$ where $q=|P_2|+|P_4|+|P_6|+\dotsb+|P_{m-2}|$.
  Let $g = m/2 - 1$.
  There are $g$ guts sets, each of size one, so $|E(M \ba a,b)| = 6+q+g$, and thus $r^*(M \ba a,b) = 3+g$.
  By \cref{rkcorkbounds}, $3+q = r(M \ba a,b) \le r^*(M \ba a,b) + 2 = 5+g$, so $q \le g+2$.
  On the other hand, there are $g$ coguts sets (excluding ends), and all have size at least~$2$.
  So $q \ge 2g$.
  Now $2g \le q \le g+2$, so $g \le 2$.
  Moreover, $q \le g+2$, so $q \le 4$.
  So $|E(M\ba a,b)| = 6+q+g \le 12$ in the case that $i_c \le 3$, a contradiction.

  Now we assume that $3 < i_c < n-2$.
  If $3 < i_s < n-2$, then $L=\{p_1,p_2,p_3,p_{i_s}\}$ and $R=\{p_{i_s},p_{n-2},p_{n-1},p_n\}$, by \cref{endslipperiness}\ref{esii}, and either no triad of $L$ is blocked by $p_{i_c}$ when $i_s < i_c$, or no triad of $R$ is blocked by $p_{i_c}$ when $i_c < i_s$.
  So up to reversing the ordering we may assume that $i_s \le 3$.

  By \cref{endslipperiness}\ref{esii}, $\{p_1,p_2,p_3\} \subseteq L$ and $\{p_{n-2},p_{n-1},p_n\} \subseteq R$. 
  Since $i_s \in \{1,2,3\}$, we have $R=\{p_{i_s},p_{n-2},p_{n-1},p_n\}$.
  Let $L = \left\{p_1,p_2,p_3,p_{i_L}\right\}$ and observe that $i_L > i_c$, for otherwise $c$ does not block any triad of $L$ by \cref{pathbasics}.
  As $p_{i_s} \in \cocl_{M \ba a,b}(\{c,p_{n-2},p_{n-1},p_n\})$, we may assume that $i_s = i_c-1$.
  Using \cref{endslipperiness}\ref{esii}, we deduce that $i_s = 3$ and $i_c = 4$.
  As $p_{i_L} \in \cocl_{M \ba a,b}(\{p_1,p_2,p_3,c\})$, we may also assume that $i_L = i_c+1=5$.

  Let $(P_1,\dotsc,P_m)$ be the guts-coguts concatenation of $\sigma_c$ with ends $P_1 = \{p_1,p_2,s\}$ and $P_m = \{p_{n-2},p_{n-1},p_n\}$.
  Note that $P_2 = \{c\}$ and $p_{i_L} \in P_3$.
%
  Observe that $r^*_{M \ba a,b}(P_m) = 2$ and $p_{i_s} \in \cocl_{M \ba a,b,c}(P_m)$, so $r^*_{M \ba a,b}(P_m \cup \{c,p_{i_s}\})=3$.
  It follows that $\lc^*(P_m,\{c,p_{i_s}\})=2+2-3=1$.
  By the dual of \cref{growpi}, $\lc^*(P_1 \cup P_2, P_m) \ge 1$, and, by the dual of \cref{pflancoguts}, for every guts set $P_j$ with $j > 2$ we have $|P_j|=1$.  We have also seen that $|P_2|=1$.
  Suppose $P_j$ is a coguts set with 
  $j \neq 3$.
  Then $5 \le j \le m-2$.
  Then, by the duals of \cref{growpi,pflantriad}, $P_{j-1} \cup P_j \cup P_{j+1}$ is a triangle of $M \ba a,b$.
  As this triangle avoids $c$, it is also a triangle of $M \ba a,b,c$, contradicting that $\{a,b,c\}$ is a special delete triple.
  So for each coguts set $P_j$ with $j \neq 3$, we have $|P_j| \ge 2$.

  Observe that $m$ is odd, $P_j$ is a guts set for each even $j$, and $P_j$ is a coguts set for each odd $j \notin \{1,m\}$.
  It follows that $r(M \ba a,b) = 4+q$ where $q=|P_3|+|P_5|+|P_7|+\dotsb+|P_{m-2}|$.
  Let $g = (m-1)/2$.
  There are $g$ guts sets, each of size one, so $|E(M \ba a,b)| = 6+q+g$, and thus $r^*(M \ba a,b) = 2+g$.
  By \cref{rkcorkbounds}, $4+q = r(M \ba a,b) \le r^*(M \ba a,b) + 2 = 4+g$, so $q \le g$.
  On the other hand, there are $g - 1$ coguts sets (excluding ends), and all except possibly $P_3$ has size at least~$2$.
  So $q \ge 2(g - 1) - 1 = 2g - 3$.
  Now $2g - 3 \le q \le g$, so $g \le 3$.
  Moreover, $q \le g$, so $q \le 3$.
  So $|E(M\ba a,b)| = 6+q+g \le 12$, a contradiction.
\end{subproof}

By \cref{singlecoflanoutcome}, we may now assume that $L$ and $R$ are disjoint.
We may also assume that $\sigma_c=(p_1,p_2,\dotsc,p_n)$ is a sequential ordering for $M \ba a,b$ such that some initial segment and some terminal segment of $\sigma_c$ are ends of a nice path description for $M \ba a,b$.
%
Suppose that $i_c \le 3$.  Then $\{p_1,p_2,p_3,p_4\}-p_{i_c} \subseteq L$, by \cref{endslipperiness}\ref{esii}.  Since $L$ and $R$ are disjoint, $R \subseteq \{p_{i_c+1},\ldots,p_n\}$. By \cref{pathbasics}, no triad contained in $R$ is blocked by $c$, a contradiction. 
By symmetry, we deduce that $3 < i_c < n-2$.

By \cref{endslipperiness}\ref{esii}, $\{p_1,p_2,p_3\} \subseteq L$ and $\{p_{n-2},p_{n-1},p_n\} \subseteq R$. In particular, $\{p_1,p_2,p_3\}$ and $\{p_{n-2},p_{n-1},p_n\}$ are triads of $M \ba a,b,c$.
Let $L = \left\{p_1,p_2,p_3,p_{i_L}\right\}$ and $R = \left\{p_{i_R},p_{n-2},p_{n-1},p_n\right\}$.
If $i_L < i_c$, then, by \cref{pathbasics}, $c$ does not block any triad of $L$, a contradiction.
So $i_L > i_c$ and, similarly, $i_R < i_c$.
Moreover, since $p_{i_R} \in \cocl_{M \ba a,b,c}(\{p_{n-2},p_{n-1},p_n\})$, we have $p_{i_R} \in \cocl_{M\ba a,b}(\{p_{i_c},\dotsc,p_n\})$, so we may assume that $i_R = i_c-1$.  Similarly, we may assume that $i_L = i_c+1$.

Let $(P_1,\dotsc,P_m)$ be the guts-coguts concatenation of $\sigma_c$ with ends $P_1=\{p_1,p_2,p_3\}$ and $P_m=\{p_{n-2},p_{n-1},p_n\}$.
%
%
Choose $j_c \in \{2,3,\dotsc,m-1\}$ such that $c \in P_{j_c}$.
Since $i_L-1 = i_c = i_R+1$, where $p_{i_c}$ is a guts element but $p_{i_L}$ and $p_{i_R}$ are coguts elements, we have $|P_{j_c}| = 1$.

Observe that $r^*_{M \ba a,b}(P_1) = 2$ and $p_{i_L} \in \cocl_{M \ba a,b,c}(P_1)$, so $r^*_{M \ba a,b}(P_1 \cup \{c,p_{i_L}\})=3$.
It follows that $\lc_{M \ba a,b}^*(P_1,\{c,p_{i_L}\})=2+2-3=1$.
By the dual of \cref{growpi}, $\lc^*(P_1,P_{j_c} \cup \dotsb \cup P_m) \ge 1$, so, by the dual of \cref{pflancoguts}, for every guts set $P_j$ such that $j < j_c$ we have $|P_j| = 1$.
We have also seen that $|P_{j_c}| = 1$.
By symmetry, every guts set $P_j$ 
has size one.

Now suppose $P_j$ is a coguts set with $|P_j|=1$ and $j \notin \{j_c-1,j_c+1\}$.
Then, by the duals of \cref{growpi,pflantriad} and symmetry, $P_{j-1} \cup P_j \cup P_{j+1}$ is a triangle of $M \ba a,b$.
As this triangle avoids $c$, it is also a triangle of $M \ba a,b,c$, contradicting that $\{a,b,c\}$ is a special delete triple.
So for each coguts set $P_j$, where $j \notin \{j_c-1,j_c+1\}$, we have $|P_j| \ge 2$.

Observe that $m$ is odd, $P_j$ is a guts set for each even $j$, and $P_j$ is a coguts set for each odd $j \notin \{1,m\}$.
It follows that $r(M \ba a,b) = 4+q$ where $q=|P_3|+|P_5|+|P_7|+\dotsb+|P_{m-2}|$.
Let $g = (m-1)/2$.
There are $g$ guts sets, each of size one, so $|E(M \ba a,b)| = 6+q+g$, and thus $r^*(M \ba a,b) = 2+g$.
By \cref{rkcorkbounds}, $4+q = r(M \ba a,b) \le r^*(M \ba a,b) + 2 = 4+g$, so $q \le g$.
On the other hand, there are $g - 1$ coguts sets (excluding ends), at most two of which have size one.
So $q \ge 2(g - 1) - 2 = 2g - 4$.
Now $2g - 4 \le q \le g$, so $g \le 4$.
Moreover, $q \le g$, so $q \le 4$.

So far, we have shown that $|E(M\ba a,b)| =6+q+g \le 14$.
If $q \le 3$, then, as $2g-4 \le q \le 3$, we have $g \le 3$, and $|E(M\ba a,b)| =6+q+g \le 12$, as required.
So it remains only to rule out the possibility that $q=4$.

Assume $q=4$.
Then $g=4$, $m=9$, $r(M \ba a,b) = 8$, and $|E(M \ba a,b)| = 14$, and there are three coguts sets: two are singletons, and one has size two.
Recall that for each coguts set $P_j$ with $j \notin \{j_c-1,j_c+1\}$ we have $|P_j| \ge 2$.
So we may assume, without loss of generality, that $|P_3|=2$ and $|P_5|=|P_7| = 1$, where $j_c=6$, thus $(P_1,\dotsc,P_m) =$
$$(\{p_1,p_2,p_3\}, \{p_4\}, \{p_5,p_6\}, \{p_7\}, \{p_{i_R}\}, \{c\}, \{p_{i_L}\}, \{p_{11}\}, \{p_{12},p_{13},p_{14}\}).$$

Recall that, by \cref{mabcfragile}, $M \ba a,b,c$ has no $\utfutf$-essential elements.
Moreover, if $e$ is $\utfutf$-deletable (or $\utfutf$-contractible) in $M \ba a,b,c$, then it is $\utfutf$-deletable (or $\utfutf$-contractible respectively) in $M \ba a,b$; 
and, since $M \ba a,b$ is $\utfutf$-fragile, the converse also holds.
Thus, in what follows, when we say an element is $\utfutf$-deletable (or $\utfutf$-contractible), it is $\utfutf$-deletable (or $\utfutf$-contractible, respectively) in each of the matroids $M \ba a,b,c$, $M \ba a,b$, $M \ba a,c$, and $M \ba b,c$.
  
\begin{claim}
  \label{revealtriads}
  Up to labels,
  \begin{enumerate}[label=\rm(\Roman*)]
    \item $\{p_1,p_4,p_5\}$ and $\{p_2,p_4,p_6\}$ are triads of $M \ba a,b,c$ not blocked by $c$, and $p_3$ is the unique $\utfutf$-deletable element in $L$; and
    \item $\{p_{i_L},p_{11},p_{14}\}$ is a triad of $M \ba a,b,c$ not blocked by $c$, and $p_{12}$ is the unique $\utfutf$-deletable element in $R$.
  \end{enumerate}
\end{claim}
\begin{subproof}
  If $p_4$ is $\utfutf$-contractible, then some element in $\{p_1,p_2,p_3\}$ is $\utfutf$-flexible by \cref{minor3conn}, a contradiction.
  So $p_4$ is $\utfutf$-deletable. 
  Similarly, $p_5$ and $p_6$ are $\utfutf$-contractible and not $\utfutf$-deletable.
  Consider $M \ba a,b \ba p_5$.
  Observe that $r^*_{M \ba a,b \ba p_5}(\{p_1,p_2,p_3,p_4,p_6\}) = 2$.
  Thus, if $\{p_4,p_5\}$ does not cospan an element of $\{p_1,p_2,p_3\}$ in $M \ba a,b$, then $(M \ba a,b \ba p_5)^*|\{p_1,p_2,p_3,p_4,p_6\} \cong U_{2,5}$, so $p_5$ is $\utfutf$-deletable, a contradiction.
  So $\{p_4,p_5\}$ cospans an element of $\{p_1,p_2,p_3\}$ in $M \ba a,b$, and,
  similarly $\{p_4,p_6\}$ cospans an element of $\{p_1,p_2,p_3\}$.
  Without loss of generality, $\{p_1,p_4,p_5\}$ and $\{p_2,p_4,p_6\}$ are triads of $M \ba a,b$, and hence also of $M \ba a,b,c$.
  It now follows that $p_1$ and $p_2$ are $\utfutf$-contractible. 
  Since some initial segment of $\sigma_c$ is an end of a nice path description $\mathbf{P}_c$ for $M \ba a,b$, and
  $\{p_1,p_2,p_3\}$ is a coclosed triad in $M \ba a,b$ not contained in a $4$-element fan, this triad is an end of $\mathbf{P}_c$. 
  By \cref{nicepathdescription}\ref{nicepathdesciii}, $p_3$ is $\utfutf$-deletable. 

  In a similar manner, $r^*_{M \ba a,b,c}\left(\left\{p_{i_R},p_{i_L},p_{11},p_{12},p_{13},p_{14}\right\}\right) = 3$ and $R=\left\{p_{i_R},p_{12},p_{13},p_{14}\right\}$ is a $4$-cosegment of $M \ba a,b,c$, so if $\left\{p_{i_L},p_{11}\right\}$ does not cospan an element of $R$ in $M \ba a,b,c$, then $M \ba a,b,c \ba p_{i_L}$ has a $5$-cosegment, implying $M \ba a,b,c$ is not $\utfutf$-fragile, a contradiction.
  So we may assume that $\{p_{i_L},p_{11},p_{14}\}$ is a triad of $M \ba a,b,c$.
  As $c$ is a guts element, we know from $\sigma_c$ that $c \notin \cocl_{M \ba a,b}(\{p_{i_L},p_{11},p_{14}\})$, so $\{p_{i_L},p_{11},p_{14}\}$ is also a triad of $M \ba a,b$.
  Moreover, $p_{11}$ is $\utfutf$-deletable, and hence $p_{i_L}$ and $p_{14}$ are $\utfutf$-contractible.
  Since some terminal segment of $\sigma_c$ is an end of a nice path description for $M \ba a,b$, and by \cref{nicepathdescription}\ref{nicepathdesciii},
  $\{p_{12},p_{13},p_{14}\}$ has a unique element that is $\utfutf$-deletable; 
  we may assume that $p_{12}$ is this element, whereas $p_{13}$ is $\utfutf$-contractible.
\end{subproof}

Recall that each triad of $M \ba a,b,c$ is blocked by at least one of $a$, $b$, or $c$, and observe that neither $\{p_1,p_2,p_3\}$ nor $\{p_{12},p_{13},p_{14}\}$ is blocked by $c$.
By \cref{revealtriads}, we may assume that $\{p_1,p_4,p_5\}$, $\{p_2,p_4,p_6\}$, and $\{p_{i_L},p_{11},p_{14}\}$ are also triads of $M \ba a,b,c$ not blocked by $c$.
  Without loss of generality, assume $\{p_1,p_2,p_3\}$ is blocked by $a$.
  We next consider what other triads can be blocked by $a$.

\begin{claim}
  \label{whatablocks}
  If a triad $T^*$ of $M \ba a,b,c$ is blocked by $a$, then $T^* \cap L \neq \emptyset$.
  Moreover, at most one of $\{p_1,p_4,p_5\}$ and $\{p_2,p_4,p_6\}$ is blocked by $a$.
\end{claim}
\begin{subproof}
  Let $\sigma_a = (p'_1,p'_2,\dotsc,p'_{14})$ be a sequential ordering that is a refinement of a nice path description $\mathbf{P}_a$ of $M \ba b,c$.
  If the left end (or right end) of $\mathbf{P}_a$ is a fan of size at least~$4$, then we choose $\{p'_1,p'_2,p'_3\}$ (or $\{p'_{12},p'_{13},p'_{14}\}$, respectively) to be a triad.
  Let $\sigma_a^-$ be the sequential ordering of $M \ba a,b,c$ obtained from $\sigma_a$ by removing $a$, as described in \cref{pathbasics}.
  By reversing these orderings, if necessary, we may assume that $L(\sigma_a^-) = L$ and $R(\sigma_a^-) = R$, due to \cref{welldefinedends}.

  First, suppose $a \notin \{p'_1,p'_2,p'_3\}$.
  Then $\{p'_1,p'_2,p'_3\} \subseteq L$ by \cref{endslipperiness}\ref{esii}, and $a$ does not block the triad $\{p'_1,p'_2,p'_3\}$.
  But $a$ blocks $\{p_1,p_2,p_3\}$, so $\{p_1,p_2,p_3\} \neq \{p_1',p_2',p'_3\}$, implying $p_{i_L} \in \{p'_1,p'_2,p'_3\}$.
  Now $\{p'_1,p'_2,p'_3\}$ is a triad in $M \ba b,c$, and this triad is contained in the left end of the nice path description $\mathbf{P}_a$ for $M \ba b,c$.
  If this end is a $4$-cosegment, then it is $\{p'_1,p'_2,p'_3,p'_4\}$, in which case $L=\{p'_1,p'_2,p'_3,p'_4\}$ and $a$ does not block the triad $\{p_1,p_2,p_3\}$, a contradiction.
  So the left end of $\mathbf{P}_a$ is either a triad, or a $4$- or $5$-element fan where $a$ forms a triangle with elements of the triad $\{p'_1,p'_2,p'_3\}$ (since $M \ba a,b,c$ has no triangles).

  We show, in any case, that $p_3 \in \{p'_1,p'_2,p'_3\}$.
  Recall that $p_3$ is the unique $\utfutf$-deletable element 
  in $L$.
  Suppose $\{p'_1,p'_2,p'_3,a\}$ is a $4$-element fan in $M \ba b,c$ that is contained in the left end of $\mathbf{P}_a$.
  Since $\{p'_1,p'_2,p'_3\}$ is contained in a $4$-element fan with no $\utfutf$-essential elements, this set contains a $\utfutf$-deletable element.
  So $p_3 \in \{p'_1,p'_2,p'_3\}$ as claimed.
  Now suppose the left end of $\mathbf{P}_a$ is a triad.
  Then this end is $\{p_1',p_2',p_3'\}$, and this set contains a $\utfutf$-deletable element by \cref{nicepathdescription}\ref{nicepathdesciii}.
  Since $p_3$ is the unique $\utfutf$-deletable element 
  in $L$, we again have $p_3 \in \{p'_1,p'_2,p'_3\}$.

  Now $\{p'_1,p'_2,p'_3\}=\{p_h,p_3,p_{i_L}\}$ for some $h \in \{1,2\}$.
  We claim that $\{p_h,p_3,p_{i_L}\}$ is closed in $M \ba a,b,c$.
  Suppose not.
  Assume $h=1$ and say $p_k \in \cl_{M \ba a,b,c}(\{p_1,p_3,p_{i_L}\})$ for some $k \in \seq{14} - \{1,3,i_L\}$.
  Then $\{p_1,p_3,p_{i_L},p_k\}$ is a circuit and, since $p_{i_L}$ is a coguts element in $\sigma_c^-$, we have $k=11$,
  contradicting orthogonality with the triad $\{p_1,p_4,p_5\}$.
  The argument is essentially the same when $h=2$, but with $(p_1,p_5)$ and $(p_2,p_6)$ swapped.
  So $\{p_1',p_2',p_3'\}$ is closed in $M \ba a,b,c$.
  Now, since the left end of $\mathbf{P}_a$ is not a $4$-cosegment, $p'_4 \in \cl_{M \ba b,c}(\{p'_1,p'_2,p'_3\})$, which implies that $p'_4=a$.
  Thus, by \cref{pathbasics}, if $a$ blocks a triad, then this triad meets $\{p_h,p_3,p_{i_L}\} \subseteq L$. Moreover, $a$ does not block both $\{p_1,p_4,p_5\}$ and $\{p_2,p_4,p_6\}$.

  Now suppose $a \in \{p'_1,p'_2,p'_3\}$.
  We may assume, without loss of generality, that $a=p'_3$. Then $\{p'_1,p'_2,p'_4\} \subseteq L$, by \cref{endslipperiness}\ref{esii}, so $\{p'_1,p'_2,p'_4\}$ is a triad of $M \ba a,b,c$.
  Since $M \ba a,b,c$ is $3$-connected, $\{p'_1,p'_2,p'_3\}$ is a triangle of $M \ba b,c$.
  Note that $\{p'_1,p'_2,p'_3,p'_4\}$ is not a $4$-element fan of $M \ba b,c$, for otherwise we would have chosen $\sigma_a$ so that $\{p'_1,p'_2,p'_3\}$ is a triad.
  So $\{p'_1,p'_2,a,p'_4\}$ is a cocircuit of $M \ba b,c$, and the triangle $\{p'_1,p'_2,a\}$ is the left end of $\mathbf{P}_a$.
  By \cref{nicepathdescription}\ref{nicepathdesciii} and \cref{pathdesctris}, the set $\{p'_1,p'_2,a\}$ contains one $\utfutf$-contractible element and two $\utfutf$-deletable elements.
  Hence either $p'_1$ or $p'_2$ is $\utfutf$-deletable.
  Since $p_3$ is the unique $\utfutf$-deletable element 
  in $L$, we have $p_3 \in \{p'_1,p'_2\}$.
  Now $\{p'_1,p'_2\} \in \left\{\{p_1,p_3\}, \{p_2,p_3\}, \{p_{i_L},p_3\} \right\}$.
  Thus, if $a$ blocks a triad, then the triad meets $\{p'_1,p'_2\} \subseteq L$.
  Moreover, $a$ does not block both $\{p_1,p_4,p_5\}$ and $\{p_2,p_4,p_6\}$.
\end{subproof}


By \cref{whatablocks}, we may now assume that $b$ blocks $\{p_{12},p_{13},p_{14}\}$.

\begin{claim}
  \label{whatbblocks}
  Either
  \begin{enumerate}[label=\rm(\Roman*)]
    \item $b$ blocks neither $\{p_1,p_4,p_5\}$ nor $\{p_2,p_4,p_6\}$; or
    \item $\{p_5,p_7,p_{i_R}\}$ is a triad of $M \ba a,b,c$ that is not blocked by $b$, up to swapping $(p_1,p_5)$ and $(p_2,p_6)$.
  \end{enumerate}
\end{claim}
\begin{subproof}
  Let $\sigma_b = (p''_1,p''_2,\dotsc,p''_{14})$ be a sequential ordering for $M \ba a,c$ such that some initial segment and some terminal segment of $\sigma_b$ are ends of a nice path description $\mathbf{P}_b$ for $M \ba a,c$, where if the left (or right) end of $\mathbf{P}_b$ is a fan of size at least~$4$, then we choose $\{p''_{1},p''_{2},p''_{3}\}$ (or $\{p''_{12},p''_{13},p''_{14}\}$, respectively) to be a triad.
  Let $\sigma_b^-$ be the sequential ordering of $M \ba a,b,c$ obtained from $\sigma_b$ by removing $b$, as described in \cref{pathbasics}.
  By reversing these orderings, if necessary, we may assume that $L(\sigma_b^-) = L$ and $R(\sigma_b^-) = R$, due to \cref{welldefinedends}.

  First we assume that $b \in \{p''_{12},p''_{13},p''_{14}\}$.
  Without loss of generality, $b=p''_{12}$.
  Then $\{p''_{11},p''_{13},p''_{14}\} \subseteq R$, by \cref{endslipperiness}\ref{esii}, so $\{p''_{11},p''_{13},p''_{14}\}$ is a triad of $M \ba a,b,c$.
  Since $M \ba a,b,c$ is $3$-connected, $\{b,p''_{13},p''_{14}\}$ is a triangle of $M \ba a,c$.
  Note that $\{p''_{11},b,p''_{13},p''_{14}\}$ is not a $4$-element fan of $M \ba a,c$, for otherwise we would have chosen $\sigma_b$ so that $\{p''_{12},p''_{13},p''_{14}\}$ is a triad.
  So $\{p''_{11},b,p''_{13},p''_{14}\}$ is a cocircuit of $M \ba a,c$ and the triangle $\{b,p''_{13},p''_{14}\}$ is the right end of $\mathbf{P}_b$.
  By \cref{nicepathdescription}\ref{nicepathdesciii} and \cref{pathdesctris}, $\{b,p''_{13},p''_{14}\}$ contains one $\utfutf$-contractible element, and two $\utfutf$-deletable elements.
  Hence either $p''_{13}$ or $p''_{14}$ is $\utfutf$-deletable.
  Since $p_{12}$ is the unique $\utfutf$-deletable element in $R$, we have $p_{12} \in \{p''_{13},p''_{14}\}$.
  So $\{p''_{13},p''_{14}\} \in \{\{p_{12},p_{14}\},\{p_{12},p_{13}\},\{p_{12},p_{i_R}\}\}$.
  In any case, $b$ blocks neither $\{p_1,p_4,p_5\}$ nor $\{p_2,p_4,p_6\}$, as required.

  Now we may assume that $b \notin \{p''_{12},p''_{13},p''_{14}\}$.
  Then $\{p''_{12},p''_{13},p''_{14}\} \subseteq R$ by \cref{endslipperiness}\ref{esii}, and $b$ does not block the triad $\{p''_{12},p''_{13},p''_{14}\}$.
  But $b$ blocks $\{p_{12},p_{13},p_{14}\}$, so $\{p_{12},p_{13},p_{14}\} \neq \{p''_{12},p''_{13},p''_{14}\}$, implying $p_{i_R} \in \{p''_{12},p''_{13},p''_{14}\}$.
  Now $\{p''_{12},p''_{13},p''_{14}\}$ is a triad in $M \ba a,c$, and this triad is contained in the right end of $\mathbf{P}_b$.
  If this end is a $4$-cosegment, then it is $\{p''_{11},p''_{12},p''_{13},p''_{14}\}$, in which case $R = \{p''_{11},p''_{12},p''_{13},p''_{14}\}$ and $b$ does not block the triad $\{p_{12},p_{13},p_{14}\}$, a contradiction.
  So the right end of $\mathbf{P}_b$ is either a triad, or a $4$- or $5$-element fan where $b$ forms a triangle with elements of the triad $\{p''_{12},p''_{13},p''_{14}\}$.

  We show, in any case, that $p_{12} \in \{p''_{12},p''_{13},p''_{14}\}$.
  Recall that $p_{12}$ is the unique $\utfutf$-deletable element 
  in $R$.
  Suppose $\{p''_{12},p''_{13},p''_{14},b\}$ is a $4$-element fan in $M \ba a,c$ contained in the right end of $\mathbf{P}_b$.
  Since $\{p''_{12},p''_{13},p''_{14}\}$ is contained in a $4$-element fan with no $\utfutf$-essential elements, this set contains a $\utfutf$-deletable element.
  So $p_{12} \in \{p''_{12},p''_{13},p''_{14}\}$ as claimed.
  Now suppose the right end of $\mathbf{P}_b$ is a triad.
  Then this end is $\{p''_{12},p''_{13},p''_{14}\}$, and it contains a $\utfutf$-deletable element, by \cref{nicepathdescription}\ref{nicepathdesciii}.
  Since $p_{12}$ is the unique $\utfutf$-deletable element 
  in $R$, we again have $p_{12} \in \{p''_{12},p''_{13},p''_{14}\}$.

  Now $\{p''_{12},p''_{13},p''_{14}\} = \{p_{i_R},p_{12},p_{g}\}$ for some $g \in \{13,14\}$.
  We claim that either $\{p_{i_R},p_{12},p_{g}\}$ is closed in $M \ba a,b,c$, or $g=13$ and 
  $\cl_{M \ba a,b,c}(\{p_{i_R},p_{12},p_{13}\}) = \{p_7,p_{i_R},p_{12},p_{13}\}$.
  Suppose $p_k \in \cl_{M \ba a,b,c}(\{p_{i_R},p_{12},p_{g}\})$ for some $k \in \seq{14} - \{i_R,12,g\}$.
  Then $\{p_k,p_{i_R},p_{12},p_{g}\}$ is a circuit and, 
  since $p_{i_R}$ is a coguts element in $\sigma_c^-$, we have $k \in \{4,7\}$.
  By orthogonality with the triads $\{p_1,p_4,p_5\}$ and $\{p_{i_L},p_{11},p_{14}\}$, we have $k=7$ and $g = 13$.
  So either $\{p''_{12},p''_{13},p''_{14}\}$ is closed in $M \ba a,b,c$, or $\cl_{M \ba a,b,c}(\{p''_{12},p''_{13},p''_{14}\}) = \{p_7,p_{i_R},p_{12},p_{13}\}$, as claimed.

  Now, since the right end of $\mathbf{P}_b$ is not a $4$-cosegment, $p''_{11} \in \cl_{M \ba a,c}(\{p''_{12},p''_{13},p''_{14}\})$, which implies that either $p''_{11}=b$, or $p''_{11} = p_7$.
  But in the former case, $b$ does not block either of the triads $\{p_1,p_4,p_5\}$ or $\{p_2,p_4,p_6\}$, as required.
  So we may assume that $p''_{11} = p_7$.

  Consider $p''_{10}$.
  If $p''_{10}=b$, then neither $\{p_1,p_4,p_5\}$ nor $\{p_2,p_4,p_6\}$ is blocked by $b$, as required; so may assume that $p''_{10} \neq b$.
  Let $Q=E(M \ba a,b,c) - \{p''_{11},p''_{12},p''_{13},p''_{14}\} = \{p_1,p_2,p_3,p_4,p_5,p_6,p_{i_L},p_{11},p_{14}\}$, so $p''_{10} \in Q$.
  Observe that each element in $Q$ is in a triad of $M \ba a,b,c$ that is contained in $Q$.
  Hence $p''_{10}$ is not a guts element, so $p''_{10} \in \cocl_{M \ba a,b,c}(\{p''_{11},p''_{12},p''_{13},p''_{14}\}) = \cocl_{M \ba a,b,c}(\{p_7,p_{i_R},p_{12},p_{13}\})$.
  Note also that $p''_{10} \neq p_{14}$, for otherwise $b$ does not block $\{p_{12},p_{13},p_{14}\}$.
  Since $C=\{p_1,p_2,p_3,p_4\}$ is a circuit, $p''_{10} \notin C$.
  If $p''_{10} \in \{p_{11},p_{i_L}\}$, then $\{p_{14},p_{11},p_{i_L}\} \subseteq \cocl_{M \ba a,b,c}(\{p_7,p_{i_R},p_{12},p_{13}\})$, in which case $\{p_5,p_6\} \subseteq \cocl_{M \ba a,b,c}(\{p_7,p_{i_R},p_{12},p_{13}\})$ since $p_5$ and $p_6$ are coguts elements in $\sigma_c$.
  It follows that $r^*(M \ba a,b,c) \le 4$, so 
  $r(M \ba a,b) \ge 9$, a contradiction.
  Thus $p''_{10} \in \{p_5,p_6\}$.

  Up to possibly swapping the labels on $(p_1,p_5)$ and $(p_2,p_6)$, we may now assume that $p''_{10} = p_5$.
  We claim that $\{p_5,p_7,p_{i_R}\}$ is a triad that is not blocked by $b$.
  As $\{p_{i_R},p_{12},p_{13}\}$ is a triad in $M \ba a,b,c$ and $p_5 \in \cocl_{M \ba a,b,c}(\{p_7,p_{i_R},p_{12},p_{13}\})$,
  we have that $\{p_5,p_7\}$ is contained in a $3$- or $4$-element cocircuit~$C^*$ that is contained in $\{p_5,p_7,p_{i_R},p_{12},p_{13}\}$.
  By orthogonality with the circuit $\{p_{11},p_{12},p_{13},p_{14}\}$, either $C^* = \{p_5,p_7,p_{i_R}\}$ or $C^* = \{p_5,p_7,p_{12},p_{13}\}$, so we may assume the latter.
  But then, by cocircuit elimination with $\{p_{i_R},p_{12},p_{13}\}$, there is a cocircuit contained in $\{p_5,p_7,p_{i_R},p_{12}\}$, which, again by orthogonality, is the triad $\{p_5,p_7,p_{i_R}\}$.
  By \cref{pathbasics}, this triad is not blocked by $b$, as required.
\end{subproof}

By \cref{whatablocks}, $a$ blocks at most one of $\{p_1,p_4,p_5\}$ and $\{p_2,p_4,p_6\}$.
As neither of these triads is blocked by $c$, at least one of $\{p_1,p_4,p_5\}$ and $\{p_2,p_4,p_6\}$ is blocked by $b$.
Now, by \cref{whatbblocks}, we may assume that $\{p_5,p_7,p_{i_R}\}$ is a triad of $M \ba a,b,c$ that is not blocked by $b$.
But $\{p_5,p_7,p_{i_R}\}$ is not blocked by $a$, by \cref{whatablocks}, and, recalling that $i_c > i_R > 7$, it is also not blocked by $c$.
From this contradiction, we deduce that $M$ has no delete triples, thus completing the proof.
\end{proof}

\section{The no-delete-triples case}
\label{ndtsec}

In this section we prove the following:

\begin{theorem}
  \label{nodeltriplethm}
  Let $M$ be an excluded minor for the class of $\mathbb{P}$-representable matroids where $\mathbb{P} \in \{\mathbb{H}_5,\mathbb{U}_2\}$, and $M$ has no triads.
  Suppose there is a pair $\{a,b\} \subseteq E(M)$ such that $M \ba a,b$ is $3$-connected with a $\utfutf$-minor.
  If $M$ has no delete triples, then $|E(M)| \le 15$.
\end{theorem}

The bulk of the work in proving this \lcnamecref{nodeltriplethm} is accomplished by \cref{nodeltripleprop}, which proves the result except when $M$ has $16$ elements and specific structure.  In \cref{nodeltriplematrices}, we show that the specific structure implies that $M \ba a,b$ is, up to isomorphism, one of three particular $2$-regular matroids.
We performed a computer search to show that, in fact, when $M$ satisfies the hypotheses of the theorem and contains one of these three matroids as a minor, then $M$ has a delete triple.

We first require some definitions.
For a guts-coguts path $\mathbf{P}=(P_1,P_2,\dotsc,P_m)$, we say that $\mathbf{P}$ is \emph{left-justified} if for all $i \in \{2,3,\dotsc,m-1\}$,
\begin{enumerate}[label=\rm(\Roman*)]
  \item if $P_i$ is a guts set, then $\cl\bigl(\bigcup_{j \in [i]} P_j\bigr) - \bigl(\bigcup_{j \in [i]}P_{j}\bigr) \subseteq P_m$; and
  \item if $P_i$ is a coguts set, then $\cocl\bigl(\bigcup_{j \in [i]} P_j\bigr) - \bigl(\bigcup_{j \in [i]} P_{j}\bigr) \subseteq P_m$.
\end{enumerate}
Similarly, we say that $\mathbf{P}$ is \emph{right-justified} if $(P_m,P_{m-1},\dotsc,P_1)$ is left-justified.
Given a guts-coguts path $\mathbf{P}$, one can easily obtain a left-justified guts-coguts path $\mathbf{P}' = (P'_1,P'_2,\dotsc,P'_{m'})$ with $P_1 = P'_1$ and $P_m = P'_{m'}$; we call $\mathbf{P}'$
%
the \emph{left-justification} of $\mathbf{P}$.

Suppose that $\mathbf{P} = (P_1,P_2,\dotsc,P_m)$ is a nice path description.
Recall that a nice path description is a guts-coguts path.
Note that the left-justification of $\mathbf{P}$, and the left-justification of $(P_m,P_{m-1},\dotsc,P_1)$, are also nice path descriptions.
We say that the \emph{reversal} of $(P_1,P_2,\dotsc,P_m)$ is the nice path description $\mathbf{P'}$ obtained from the left-justification of $(P_m,P_{m-1},\dotsc,P_1)$.
By \cref{niceends}, the ends of $\mathbf{P'}$ are $P_m$ and $P_1$, when $|E(M)| \ge 15$.

For the remainder of this section we let $M$ be an excluded minor for the class of $\mathbb{P}$-representable matroids where $\mathbb{P} \in \{\mathbb{H}_5,\mathbb{U}_2\}$, and $M$ has no triads.

\begin{proposition}
  \label{nodeltripleprop}
  Suppose there is a pair $\{a,b\} \subseteq E(M)$ such that $M \ba a,b$ is $3$-connected with a $\utfutf$-minor.
  If $M$ has no delete triples, then either 
  \begin{enumerate}
    \item $|E(M)| \le 15$; or
    \item $|E(M)| = 16$ and\label{ndtendgameoutcomes16}
      \begin{enumerate}[label=\rm(\alph*)]
        \item $M \ba a,b$ has a nice path description \[(\{a',p_1',p_1\},\{p_2\},\{p_3\},\{p_4\},\{p_5,p_5'\},\{p_6\},\{p_7\},\{p_8\},\{p_9,p_9',b'\})\] where, for some $\{q,q'\} = \{p_5,p_5'\}$,
          \begin{itemize}
            \item $\{a',p_1',p_1\}$, $\{p_1,p_2,p_3\}$, $\{p_3,p_4,p_5\}$, $\{q,p_6,p_7\}$, $\{p_7,p_8,p_9\}$, and $\{p_9,p_9',b'\}$ are triads of $M \ba a,b$,
            \item $\{a',p_1,p_4,p_5'\}$ and $\{q',p_6,p_9,b'\}$ are cocircuits of $M \ba a,b$,
          \end{itemize}
          and $M \ba a,b$ has no triangles;\label{ndtendgameoutcome16a}
        \item $M \ba a,b$ has a nice path description \[(\{a',p_1'',p_1',p_1\},\{p_2,p_2'\},\{p_3,p_3'\},\{p_4\},\{p_5\},\{p_6\},\{p_7,p_7',b'\})\] where
          \begin{itemize}
            \item $\{a',p_1'',p_1',p_1\}$ is a cosegment of $M \ba a,b$, 
            \item $\{p_1',p_2',p_3'\}$, $\{p_1,p_2,p_3\}$, $\{p_3,p_4,p_5\}$, $\{p_5,p_6,p_7\}$, and $\{p_7,p_7',b'\}$ are triads of $M \ba a,b$, and
            \item $\{p_3',p_4,p_7,b'\}$ is a cocircuit of $M \ba a,b$;
          \end{itemize}\label{ndtendgameoutcome16b}
          and $M \ba a,b$ has no triangles.
      \end{enumerate}
  \end{enumerate}
\end{proposition}
\begin{proof}
  Suppose that $M$ has no delete triples and $|E(M)| \ge 16$.
  By \cref{utfutffragile}, $M \ba a,b$ is $\utfutf$-fragile.
  Loosely speaking, our strategy is to use the structure of this matroid to find a triple that, once we add back $a$ and $b$, is a delete triple in $M$; if we cannot find such a triple, then \labelcref{ndtendgameoutcomes16} holds.
  The crux to this approach is the following:

\begin{claim}
  \label{ndtcrux}
  Suppose there are distinct elements $a',b',c' \in E(M\ba a,b)$ such that the matroid $M \ba \{a,b,a',b',c'\}$ 
  \begin{itemize}
    \item is $3$-connected up to series classes of size at most three,
    \item has at least three distinct series classes, and
    \item has a $\utfutf$-minor,
  \end{itemize}
  and, in $M \ba a,b$,
  \begin{itemize}
    \item $\{a',b',c'\}$ is not contained in a $5$-element cocircuit,
    \item no pair of elements of $\{a',b',c'\}$ is contained in a $4$-element cocircuit, and
    \item $\{a',b',c'\}$ is coindependent. 
  \end{itemize}
  Then $\{a',b',c'\}$ is a delete triple for $M$.
\end{claim}
\begin{subproof}
  First, we claim that if $S$ is a series pair of $M \ba \{a,b,a',b',c'\}$, then it is blocked by $a$ or $b$ in $M \ba a',b',c'$.
  Let $S$ be a series pair of $M \ba \{a,b,a',b',c'\}$.
  Then $S \cup \{a',b',c'\}$ contains a cocircuit in $M \ba a,b$.
  If this cocircuit has size four or five, then it intersects $\{a',b',c'\}$ in two or three elements, respectively, a contradiction.
  So $S \cup e$ is a triad of $M \ba a,b$ for some $e \in \{a',b',c'\}$.
  This triad is blocked by $a$ or $b$ in $M$, since $M$ has no triads.
  Without loss of generality, say $a$ is the element that blocks the triad $S \cup e$ of $M \ba a,b$.
  Then 
  $a \in \cocl_{M \ba b}(S \cup e)$,
  so $a \in \cocl_{M \ba b,a',b',c'}(S)$, and hence
  the series pair $S$ of $M \ba \{a,b,a',b',c'\}$ is blocked by $a$, thus proving the first claim.

  Suppose $M \ba a',b',c'$ has a coloop $e$.
  Since $M \ba \{a,b,a',b',c'\}$ has no coloops, $e \in \{a,b\}$.
  Then $\{e,a',b',c'\}$ is a cocircuit of $M$, since $M$ is $3$-connected and has no triads.
  But then $\{a',b',c'\}$ is not coindependent in $M \ba a,b$, a contradiction.

  Now suppose $M \ba a',b',c'$ has a series pair $\{e,f\}$.
  Then $\{e,f,a',b',c'\}$ contains a cocircuit of $M$.
  If $\{a,b\} \cap \{e,f\} = \emptyset$, then $r^*_{M \ba a,b,a',b',c'}(\{e,f\}) \le 1$.
  Note that $\{e,f\}$ is not a series pair of $M \ba \{a,b,a',b',c'\}$, by the first claim.  So $e$ or $f$ is a coloop of $M \ba \{a,b,a',b',c'\}$.
  But this contradicts that $M \ba \{a,b,a',b',c'\}$ is $3$-connected up to series classes.
  Similarly, if $|\{a,b\} \cap \{e,f\}| = 1$, then $e$ or $f$ is a coloop of $M \ba \{a,b,a',b',c'\}$, again contradicting that $M \ba \{a,b,a',b',c'\}$ is $3$-connected up to series classes.
  So $\{a,b\} = \{e,f\}$.
  Then $\{a,b,a',b',c'\}$ contains a cocircuit of $M$, which meets $\{a',b',c'\}$ since $M$ is $3$-connected, so $\{a',b',c'\}$ is not coindependent in $M \ba a,b$, a contradiction.

  Next, we work towards proving that if $(U,V)$ is a $2$-separation of $M \ba a',b',c'$, then we may assume, up to swapping $U$ and $V$, that $U$ is a cosegment such that for some $\{e,e'\} = \{a,b\}$, we have $e \in U \cap \cl(U-e)$ and $e' \in V$.
  Let $(U,V)$ be a $2$-separation of $M \ba a',b',c'$ with $a \in U$.
  Since $M \ba a',b',c'$ has no loops, coloops, parallel pairs or series pairs, $|U|,|V| \ge 3$.
  Note that $\{a,b\}$ is coindependent in $M \ba a',b',c'$, since $M$ is $3$-connected and $\{a',b',c'\}$ is coindependent in $M \ba a,b$.
  Now $(U-a,V)$ is a $2$-separation in $M \ba \{a, a',b',c'\}$.
  Since $a$ does not block the $2$-separating set $V$, we have $a \in \cl(U-a)$.

  Suppose $U-a$ is contained in a series class of $M \ba \{a, a',b',c'\}$.
  Then $b \in V$, for otherwise the non-empty set $U-\{a,b\}$ consists of coloops of $M \ba \{a,b,a',b',c'\}$, contradicting that $M \ba \{a,b,a',b',c'\}$ is $3$-connected up to series classes.
Now each pair of $U-a$ is a series pair of $M \ba \{a,b,a',b',c'\}$ not blocked by $b$.
  So $a$ blocks $U-a$, implying $a \in \cocl_{M \ba a',b',c'}(U-a)$.
  Recalling that $a \in \cl(U-a)$, we see that $U$ is a cosegment of $M \ba a',b',c'$ with $a \in \cl(U-a)$ as required.

  So we may assume that $U-a$ is not contained in a series class of $M \ba \{a,a',b',c'\}$. 
  Suppose $b \in U$.
  Then $(U-\{a,b\},V)$ is a $2$-separation of $M \ba \{a,b,a',b',c'\}$.
  As neither $a$ nor $b$ blocks the $2$-separating set $V$, we have $\{a,b\} \subseteq \cl(U-\{a,b\})$.
  If $V$ is contained in a series class of $M \ba \{a,b,a',b',c'\}$, then it is also contained in a series class of $M \ba a',b',c'$, a contradiction.
  Since $M \ba \{a, b, a',b',c'\}$ is $3$-connected up to series classes, $U-\{a,b\}$ is contained in a series class $S$ say. 
  Now $M \ba \{a,b,a',b',c'\}$ contains some series class $S'$ distinct from $S$.
  Since $a,b \in \cl(U-\{a,b\})$, neither $a$ nor $b$ blocks $S'$, a contradiction.
  We deduce that $b \in V$.
  Now, by symmetry, 
  $b \in \cl(V-b)$. 
  Since $M \ba \{a, b, a',b',c'\}$ is $3$-connected up to series classes, either $U-a$ or $V-b$ is contained in a series class of $M \ba \{a,b,a',b',c'\}$; without loss of generality, say it is $V-b$.
  Since $a \in \cl(U-a)$, the element $a$ does not block $V-b$.
  So $b$ blocks $V-b$, that is, $b \in \cocl_{M \ba a',b',c'}(V-b)$.
  Now $V$ is a cosegment in $M \ba a',b',c'$ with $b \in V \cap \cl(V-b)$ and $a \notin V$, as required.

  Now we may assume that $M \ba a',b',c'$ has a cosegment $G$, with $a \in G \cap \cl(G-a)$ and $b \notin G$, up to swapping $a$ and $b$, for otherwise $M$ has a delete triple, $\{a',b',c'\}$, as required.
  Without loss of generality, $G$ is coclosed in $M \ba a',b',c'$,
  and it follows that $G-a$ is a series class in $M \ba \{a, b, a',b',c'\}$.
  Let $G-a$, $S'$ and $S''$ be distinct series classes of $M \ba \{a, b, a',b',c'\}$.
  Since $a \in \cl(G-a)$, it follows that $a$ blocks neither $S'$ nor $S''$.
  So $b$ blocks both $S'$ and $S''$.  Note that $b \notin \cl(S')$, for otherwise $b$ does not block $S''$; and similarly $b \notin \cl(S'')$.
  
  We deduce that the only $2$-separations of $M \ba a',b',c'$ are of the form $(G', E(M \ba a',b',c')-G')$ where $G' \subseteq G$ is a cosegment and $a \in \cl(G'-a)$, and $b \notin G'$.
  Now $G-a$ is a series class of $M \ba \{a,b,a',b',c'\}$ that is blocked in $M \ba a,b$.
  Since $M \ba a,b$ has no $4$-element cocircuits containing a pair of elements in $\{a',b',c'\}$, and no $5$-element cocircuit containing $\{a',b',c'\}$, each series pair of $M \ba \{a,b,a',b',c'\}$ contained in $G-a$ is blocked by exactly one of $a'$, $b'$, and $c'$.
  Observe that $|G-a| \in \{2,3\}$.

  We claim that there is some $e \in \{a',b',c'\}$ that blocks every pair contained in $G-a$.
  Clearly this is the case when $|G-a| =2$, so let $G-a=\{s,t,q\}$, and suppose $a'$ blocks $\{s,t\}$ but $a'$ does not block $\{s,q\}$.
  Then, we may assume that $b'$ blocks $\{s,q\}$, so $\{s,t,a'\}$ and $\{s,q,b'\}$ are triads of $M \ba a,b$.
  By cocircuit elimination, $\{t,q,a',b'\}$ contains a cocircuit; so either $\{t,q,a'\}$ or $\{t,q,b'\}$ is a triad of $M \ba a,b$.
  In the former case, $\{s,q,a'\}$ is also a triad of $M \ba a,b$, by cocircuit elimination, so $a'$ blocks $\{s,q\}$, a contradiction.
  So $\{t,q,b'\}$ is a triad of $M \ba a,b$.
  But then, by cocircuit elimination with $\{s,q,b'\}$, so is $\{s,t,b'\}$, so $b'$ blocks each pair in $G-a$.
  We may now assume that $a'$ blocks every pair contained in $G-a$.
  Then $(G-a) \cup a'$ is a cosegment of $M \ba a,b$.

  Suppose $M \ba b',c'$ is not $3$-connected.  Let $(U,V)$ be a $2$-separation in $M \ba b',c'$ with $a' \in U$.
  Note that $|U| \ge 3$, since $a'$ is not in a parallel or series pair of $M \ba b',c'$.
  Since $a'$ is not a coloop in $M \ba b',c'$, we have $\lambda_{M \ba a',b',c'}(U-a') \le \lambda_{M \ba b',c'}(U) \le 1$.
  So $(U-a',V)$ is a $2$-separation in $M \ba a',b',c'$.
  Thus either $U-a'$ or $V$ is a cosegment $G' \subseteq G$, with $a \in G' \cap \cl(G'-a)$ and $b \notin G'$.
  As $V$ is $2$-separating in both $M \ba a',b',c'$ and $M \ba b',c'$, we see that $a'$ does not block $V$, so
  $a' \in \cl_{M\ba b',c'}(U-a')$.
  If $U-a' = G'$, then $a' \in \cl(G') = \cl(G'-a)$, and so in $M \ba a,b$, the set $(G'-a) \cup a'$ is a dependent cosegment and is therefore $2$-separating, a contradiction.
  So $V = G'$, and $a'$ does not block $G'$, with $a \in V$ and $b \in U$.  
  Note that $|V-a| \ge 2$, 
  so $a'$ does not block each series pair of $M \ba a,b,a',b',c'$ contained in $G-a$, a contradiction.

  So $M \ba b',c'$ is $3$-connected. 
  As $M \ba a',b',c'$ has a $\utfutf$-minor, $M \ba a',b',c'$ is $\utfutf$-fragile, by \cref{utfutffragile}.
  But then $M \ba a',b',c'$ is $3$-connected up to series and parallel classes, by \cref{genfragileconn}, contradicting that $G$ is $2$-separating in $M \ba a',b',c'$.
\end{subproof}

Now, if $M \ba a,b$ has a triple $\{a',b',c'\}$ as described in \cref{ndtcrux}, then $M$ has a contradictory delete triple.  Our strategy is to attempt to find such a triple $\{a',b',c'\}$; when we cannot, we have the structure described in \ref{ndtendgameoutcomes16}.

Recall that $M \ba a,b$ is $3$-connected and $\utfutf$-fragile, and, due to \cref{lemmaC}, $M\ba a,b$ has an $\{X_8,Y_8,Y_8^*\}$-minor, $M \ba a,b$ has a nice path description $\mathbf{P} = (P_1,P_2,\dotsc,P_m)$, and every element of $M \ba a,b$ is either $\utfutf$-deletable or $\utfutf$-contractible. 
  Let $N \in \utfutf$ such that $M\ba a,b$ has an $N$-minor.
  By \cref{fragilecase}, there exists a basis~$B$ for $M$ and a $B \times B^*$ companion $\mathbb{P}$-matrix~$A$ for which $\{x,y,a,b\}$ incriminates $(M,A)$ where $\{x,y\} \subseteq B$ and $\{a,b\} \subseteq B^*$.
  By \cref{atmostonetri}, $M \ba a,b$ has at most one triangle, and if such a triangle $T$ exists, then $T \cup \{x,y\}$ is a $4$-element fan containing $u$.

  In what follows, we work in the matroid $M \ba a,b$ unless explicitly specified otherwise; for example, when we say $P_1$ is a triad, we mean it is a triad of $M \ba a,b$.

\begin{claim}
  \label{ndtends}
  Let $i \in \{1,m\}$.  Then $P_i$ is either a cosegment or a $5$-element fan whose ends are rim elements.
\end{claim}
\begin{subproof}
  The end $P_i$ contains either a triangle or a triad.
  If $P_i$ does not contain a triangle, then it is a cosegment.
  On the other hand, if $P_i$ contains a triangle, then, by \cref{atmostonetri}, $P_i$ is a fan of size at least~$4$.
  If $P_i$ is a fan of size at least~$6$, then it contains at least $2$ triangles, a contradiction.
  If the fan $P_i$ has a spoke end $d$, then $M \ba a,b,d$ is $3$-connected and has a $\utfutf$-minor, by \cref{fanends,pathdescends}, contradicting that $M$ has no delete triples.
  It follows that $P_i$ has size five and both its ends are rim elements, as required.
\end{subproof}

Suppose neither $P_1$ nor $P_m$ is a $5$-element fan.
Then both $P_1$ and $P_m$ are cosegments, by \cref{ndtends}.
By \cref{niceends}, $P_1$ and $P_m$ are coclosed.
Hence 
$P_2$ and $P_{m-1}$ are guts sets.
Moreover, $m$ is odd.

On the other hand, if $P_1$ is a $5$-element fan, then $P_2$ could be either a guts set (in which case $m$ is odd) or a coguts set (in which case $m$ is even).

For ease of notation, for any $i \in \{2,3,\dotsc,m-1\}$ we let $P_i^- = P_1 \cup \dotsm \cup P_{i-1}$ and $P_i^+ = P_{i+1} \cup \dotsm \cup P_{m}$.

\begin{claim}
  \label{ndttriadseq}
  Let $i \in \{2,3,\dotsc,m-1\}$ such that $P_i$ is a guts set.
  Then $|P_i| \le 2$ and for each $e \in P_i$, there is a triad of $M \ba a,b$ containing $e$ that meets $P_i^-$ and $P_i^+$.
\end{claim}
\begin{subproof}
  First, observe that if some $e \in P_i$ is in a triad~$T^*$, then it follows from orthogonality that $T^*$ meets both $P_i^-$ and $P_i^+$.

  If $|P_i| = 3$, then $P_i$ is a triangle, so $P_i$ is contained in a $4$-element fan by \cref{atmostonetri}.
  But then there is a triad containing two elements of $P_i$, so it does not meet both $P_i^-$ and $P_i^+$, a contradiction.
  So $|P_i| \le 2$.

  Now let $e \in P_i$.
  By \cref{pathdescprops}, $e$ is $\utfutf$-deletable in $M \ba a,b$, and $\co(M \ba a,b \ba e)$ is $3$-connected.  Thus, if $e$ is not in a triad, then $\{a,b,e\}$ is a delete triple, a contradiction.
  So $e$ is in a triad which, by the foregoing, meets both $P_i^-$ and $P_i^+$.
\end{subproof}

We say that $T^*$ is an \emph{internal triad} if $T^*$ is a triad that contains an element in some guts set $P_i$, for $i \in \{2,3,\dotsc,m-1\}$.

\begin{claim}
  \label{ndtdisjtriads}
  Let $\{\ell,e,r\}$ and $\{\ell',e',r'\}$ be distinct internal triads, where $e \in P_i$ and $e' \in P_{i'}$ for guts sets $P_i$ and $P_{i'}$, and
$\ell \in P_i^-$, $r \in P_i^+$, $\ell' \in P_{i'}^-$, and $r' \in P_{i'}^+$.
  Then $\ell \neq \ell'$ and $r \neq r'$.
  In particular, if for some guts set $P_j$ we have $|P_j|=2$, say $P_j = \{e,e'\}$, then the triads containing $e$ and $e'$ are disjoint.
\end{claim}
\begin{subproof}
  Suppose $\ell = \ell'$. Then, by cocircuit elimination, there is a cocircuit~$C^*$ contained in $\{e,e',r,r'\}$.
  First suppose that $i = i'$.
  Since $e$ and $e'$ are in the guts set $P_i$, and $P_i^+ \cap C^* = \{r,r'\}$,
  it follows from orthogonality that neither $e$ nor $e'$ is in the cocircuit~$C^*$.  But then $\{r,r'\}$ is a series pair, a contradiction.
  Now suppose that $i \neq i'$.  Without loss of generality, let $i < i'$.
  Then, it follows from orthogonality that $e \notin C^*$, so $\{r,e',r'\}$ is a triad, where $r \in P_i^+ \cap P_{i'}^-$.
  Now $\{\ell,r,e',r'\}$ is a $4$-cosegment.
  But then $\{\ell,r,e'\}$ is a triad that avoids $P_{i'}^+$, contradicting orthogonality.
  So $\ell \neq \ell'$ and, similarly, $r \neq r'$.
\end{subproof}

We now assume that $\mathbf{P} = (P_1, \dotsc, P_m)$ is a nice path description for $M \ba a,b$ such that $P_m$ is a triad, using \cref{oneendisatriad} and up to the reversal of $\mathbf{P}$.
Recall also that the reversal of $\mathbf{P}$ is, by definition, left-justified.

\begin{claim}
  \label{ndttriads}
  $|P_{m-1}| = 1$, and if an internal triad meets $P_m$, then this triad contains $P_{m-1}$.
  Moreover,
  \begin{enumerate}[label=\rm(\Roman*)]
    \item if $P_1$ is a triad, then $|P_2| = 1$, and each internal triad that meets $P_1$ contains $P_2$;\label{ndttriadsi}
    \item if $P_1$ is a $4$-cosegment or a $5$-element fan and $P_2$ is a guts set, and $|P_i|=2$ for some guts set $P_i$ with $i > 2$, then for some $e \in P_i$ each triad containing $e$ is disjoint from $P_1$.\label{ndttriadsii}
  \end{enumerate}
\end{claim}
\begin{subproof}
  Let $P_1$ be a cosegment.
  For each $p_2 \in P_2$, the set $P_1 \cup p_2$ contains a circuit.
  If this circuit is a triangle, then, by orthogonality, $|P_1|=3$ and $P_1$ is contained in a $4$-element fan, violating the definition of a nice path description.
  So the circuit contained in $P_1 \cup p_2$ has size at least~$4$.
  In particular, if $P_1$ is a triad, then $P_1 \cup p_2$ is a $4$-element circuit for each $p_2 \in P_2$.

  Suppose $P_1$ is a triad and $|P_2| \ge 2$.
  Let $P_2 = \{p_2,p_2'\}$.
  By \cref{ndttriadseq}, $p_2$ is in a triad $T^*$ that contains an element of $P_1$, and an element of $P_2^+$.
  But then $|T^* \cap (P_1 \cup p_2')| = 1$, contradicting orthogonality.
  So if $P_1$ is a triad, then $|P_2| = 1$.
  Similarly, since $P_m$ is a triad, $|P_{m-1}| = 1$.

  Let $P_1$ be a cosegment,
  let $p_2 \in P_2$, and let $T^*_2 = \{\ell_2,p_2,r_2\}$ be the internal triad containing $p_2$, with $\ell_2 \in P_1$.
  Suppose there is some internal triad $\{\ell_i,p_i,r_i\}$ with $\ell_i \in P_1$, and $p_i \in P_i$ for some guts set $P_i$ with $i > 2$.
  Then $r_i \in P_i^+$, so $|\{\ell_i,p_i,r_i\} \cap (P_1 \cup p_2)| = 1$.
  By orthogonality, $P_1 \cup p_2$ is not a circuit.
  In particular, we deduce that if $P_1$ is a triad, then no such internal triad $\{\ell_i,p_i,r_i\}$ exists, and \cref{ndttriads}\ref{ndttriadsi} follows.
  (By symmetry, if an internal triad meets $P_m$ then it contains $P_{m-1}$.)
  If $P_1$ is a $4$-cosegment, then we deduce that $(P_1 - \ell_i) \cup p_2$ is a $4$-element circuit.
  Now if there is some $e \in P_i-p_i$, then any internal triad containing $e$ does not contain $\ell_i$, by \cref{ndtdisjtriads}, and does not meet $P_1 - \ell_i$, by orthogonality.
  So \cref{ndttriads}\ref{ndttriadsii} holds in the case that $P_1$ is a $4$-cosegment.

  Finally, suppose $P_1$ is a $5$-element fan with ordering $(f_1,f_2,f_3,f_4,f_5)$ and $P_2$ is a guts set.
  Then $\{f_2,f_3,f_4\}$ is a triangle.
  For each internal triad $\{\ell_i,p_i,r_i\}$ with $p_i \in P_i$, $\ell_i \in P_i^-$ and $r_i \in P_i^+$, we have $|P_1 \cap \{\ell_i,p_i,r_i\}| \le 1$, so, by orthogonality, either $\ell_i \in \{f_1,f_5\}$ or $\ell_i \notin P_1$.
  By \cref{ndtdisjtriads}, at most two internal triads meet $P_1$.
  There is an internal triad containing $p_{2}$ that meets $P_1$.
  Thus for any guts set $P_i$ with $i > 2$ and $|P_i| =2$, there is some $e \in P_i$ such that any internal triad containing $e$ avoids $P_1$, as required.
\end{subproof}

Let $G$ and $Q$ be the guts and coguts elements in $P_2 \cup \dotsm \cup P_{m-1}$, respectively.

\begin{claim}
  \label{ndtcogutsupper}
  $|Q| \le |G|$.
  Moreover, 
  \begin{enumerate}[label=\rm(\Roman*)]
    \item if $P_1$ is a $4$-cosegment, or $M \ba a,b$ has a triangle, then $|Q| \le |G|-1$; and
    \item if $P_1$ is a $4$-cosegment and $M \ba a,b$ has a triangle, then $|Q| \le |G|-2$.
  \end{enumerate}
\end{claim}
\begin{subproof}
  Observe that for each coguts element $q \in P_i$, we have \[q \notin \cl(P_1 \cup \dotsm \cup P_{i-1} \cup (P_i - q)).\]
  It follows that $r(M \ba a,b) = r(P_1) + |Q| + r(P_m) - 2$.

  Suppose $P_1$ is a $5$-element fan.
  Then, by \cref{atmostonetri}, $M \ba a,b$ has an $(N,B)$-robust element outside of $\{x,y\}$, so $r(M \ba a,b) \le r^*(M \ba a,b)+1$ by \cref{rkcorkbounds}.
  By \cref{pathdescends}\ref{pde3}, an element of $P_1$ is $\utfutf$-deletable if and only if it is a spoke, so $P_1$ has precisely two elements that are $\utfutf$-deletable.
  On the other hand, the triad $P_m$ has precisely one $\utfutf$-deletable element, also by \cref{pathdescends}\ref{pde2}.
  So $M \ba a,b$ has precisely $|G|+3$ elements that are $\utfutf$-deletable, by \cref{pathdescprops}, and hence $r^*(M \ba a,b) = |G|+3$, by \cref{pathdescrank}.
  Since $r(P_1) = 4$ and $r(P_m) = 3$,
  \begin{align*}
    |Q| &= r(M \ba a,b) + 2 - r(P_1) - r(P_m) \\
    &\le r^*(M \ba a,b) + 3 - 4 - 3\\
    &= (|G|+3) -4 = |G|-1,
  \end{align*}
  as required.

  Now, by \cref{ndtends}, we may assume that $P_1$ is a cosegment.
  Then $M \ba a,b$ has $|G| + 2$ elements that are $\utfutf$-deletable, by \cref{pathdescprops,pathdescends}.
  If $M \ba a,b$ has a triangle, then $r(M \ba a,b) \le r^*(M \ba a,b)+1 = |G|+3$ by \cref{atmostonetri,rkcorkbounds}.
  Otherwise, by \cref{rkcorkbounds,pathdescrank}, $r(M \ba a,b) \le r^*(M \ba a,b)+2 = |G| + 4$.
  Observe that $r(P_1)\ge 3$ and $r(P_m)=3$.
  In the case that $M \ba a,b$ does not have a triangle,
  \begin{align*}
    |Q| &= r(M \ba a,b) - r(P_1) - r(P_m) + 2 \\
    &\le (|G|+4) -r(P_1) -3 + 2 \\
    &= |G| -r(P_1) +3 \le |G|, 
  \end{align*}
  and, if $P_1$ is a $4$-cosegment, then $r(P_1) = 4$, in which case $|Q| \le |G|-1$.
  Similarly, if $M \ba a,b$ has a triangle, then
  \begin{align*}
    |Q| &\le 
    |G| -r(P_1) +2 \le |G|-1,
  \end{align*}
  and, if $P_1$ is a $4$-cosegment, then $r(P_1) = 4$, in which case $|Q| \le |G|-2$.
\end{subproof}

\begin{claim}
  \label{ndt5eltfan}
  If $P_1$ is a $5$-element fan, then $P_2$ is a guts set and $|P_2|=1$.
\end{claim}
\begin{subproof}
Let $P_1$ be a $5$-element fan.
First, suppose $P_2$ is a coguts set.
Then $m$ is even. 
Let $P_i$ be a guts set of $\mathbf{P}$ with $i \neq m-1$.
If $|P_i| = 1$, then clearly $|P_i| \le |P_{i+1}|$.
Otherwise, $|P_i| = 2$, by \cref{ndttriadseq}, in which case, by \cref{ndtdisjtriads,ndttriads}, there are two disjoint internal triads that meet $P_i$ and avoid $P_m$.
Since $\mathbf{P}$ is left-justified, $|P_i| = 2 \le |P_{i+1}|$.
Finally, observe that $1=|P_{m-1}| \le |P_2|$, by \cref{ndttriads}.
Since 
$m$ is even, it follows that
$|G| \le |Q|$,
but this contradicts \cref{ndtcogutsupper}.
We deduce that $P_2$ is a guts set.

Now suppose $|P_2| \ge 2$.
Then $|P_2| = 2$, by \cref{ndttriadseq}, so let $P_2 = \{e,e'\}$.
By \cref{ndtdisjtriads} there are distinct elements $\ell$ and $\ell'$ such that $\{\ell,e\}$ and $\{\ell',e'\}$ are contained in triads where, for each triad, the final element is in $P_2^+$.
By orthogonality, $\ell$ and $\ell'$ are the rim ends of the fan $P_1$.
Let $(\ell,f_2,f_3,f_4,\ell')$ be an ordering of $P_1$.
As $r(\{\ell,f_2,f_4,\ell',e,e'\}) = 4$, and $\{\ell,f_2,f_3\}$ is a triad, $r(\{f_4,\ell',e,e'\}) \le 3$.
But $M \ba a,b$ has at most one triangle, $\{f_2,f_3,f_4\}$, so $\{f_4,\ell',e,e'\}$ is a circuit.
This circuit intersects the triad containing $\{\ell,e\}$ in a single element, contradicting orthogonality.
We deduce that $|P_2| = 1$.
\end{subproof}

By \cref{ndt5eltfan}, we may now assume that $P_2$ is a guts set, so $m$ is odd, and if $P_1$ is not a $4$-cosegment, then $|P_2|=1$.

We may also assume that $m \ge 5$, for otherwise $|E(M)| \le 11$.
Suppose $m = 5$.  As $|P_1|+|P_2| \le 6$ and $|P_{4}| + |P_5| = 4$, we may assume that $|P_3|=3$ and $|P_1|+|P_2| = 6$.  But the latter implies that $P_1$ is either a $5$-element fan or $4$-cosegment, so $|Q| \le |G|-1$ by \cref{ndtcogutsupper}, in which case $|P_3| \le 2$.
So $m \ge 7$.

\begin{claim}
  \label{ndtendcandidates}
  There exist elements $a' \in P_1$ and $b' \in P_m$ such that $M \ba a,b \ba a',b'$ is $\utfutf$-fragile, where each $\utfutf$-deletable (or $\utfutf$-contractible) element of $M \ba a,b$ not in $\{a',b'\}$ remains $\utfutf$-deletable (or $\utfutf$-contractible, respectively) in $M \ba a,b \ba a',b'$.
  Moreover,
  \begin{enumerate}[label=\rm(\Roman*)]
    \item neither $a'$ nor $b'$ is in an internal triad of $M \ba a,b$;
    \item $a'$ is in a circuit contained in $P_1 \cup p_2$ for each $p_2 \in P_2$, and $P_{m-1} \cup P_m$ is a $4$-element circuit containing $b'$; and
    \item if $P_1$ is a $5$-element fan, then $a'$ is a spoke of $P_1$.
  \end{enumerate}
\end{claim}
\begin{subproof}
  Let $i \in \{1,m\}$.
  When $P_i$ is not a $5$-element fan, then, using \cref{pathdescends}\ref{pde2}, we choose $e_i$ to be the unique element in $P_i$ that is $\utfutf$-deletable.
  When $P_1$ is a $5$-element fan, we choose $e_1$ to be a spoke of $P_1$;
  then $e_1$ is $\utfutf$-deletable, by \cref{pathdescends}\ref{pde3}. 
  Let $a' = e_1$ and $b'=e_m$.
  By \cref{keepfragilelabels}\ref{kfl1}, $M \ba a,b \ba a',b'$ is $\utfutf$-fragile and has no $\utfutf$-essential elements.  Since $M \ba a,b$ has no $\utfutf$-flexible elements, each element of $M \ba a,b \ba a',b'$ is $\utfutf$-deletable (or $\utfutf$-contractible) if and only if it was $\utfutf$-deletable (or $\utfutf$-contractible, respectively) in $M \ba a,b$.
  This proves the first part of \cref{ndtendcandidates}.

  Suppose $a'$ is in an internal triad $\{a',e,r\}$, where $e$ is the guts element.
  Then $M \ba a,b \ba a'$ has a $\utfutf$-minor, and $\{e,r\}$ is a series pair in this matroid, so $e$ is $\utfutf$-contractible in $M \ba a,b \ba a'$ and thus also in $M \ba a,b$.
  But this contradicts \cref{pathdescprops}.
  So $a'$ is not in an internal triad.
  Similarly, neither is $b'$.

  By \cref{ndttriads}, $|P_{m-1}| = 1$, so $P_m \cup P_{m-1}$ is a circuit containing $b'$.
  Similarly, if $P_1$ is a triad, then $|P_2| = 1$ and $P_1 \cup P_2$ is a circuit containing $a'$.
  If $P_1$ is a $5$-element fan, then there is a unique triangle $T \subseteq P_1$, and $a' \in T$, since $a'$ was chosen to be a spoke of $P_1$.

  Finally, suppose that $P_1$ is a $4$-cosegment.
  For each $p_2 \in P_2$, the set $P_1 \cup p_2$ contains a circuit.
  It remains to show that this circuit contains $a'$.
  Clearly this is the case if $P_1 \cup p_2$ is a circuit, so suppose otherwise.
  By orthogonality, the circuit is not a triangle, so $(P_1 - a') \cup p_2$ is a $4$-element circuit.
  Let $P_1 = \{p_1,p_1',p_1'',a'\}$, so $\{p_1,p_1',p_1'',p_2\}$ is a circuit.
  By \cref{pathdescends}\ref{pde3}, $p_1$, $p_1'$, and $p_1''$ are $\utfutf$-contractible.
By \cref{keepfragilelabels}\ref{kfl2}, $M \ba a,b / p_1$ is $\utfutf$-fragile and has no $\utfutf$-essential elements, so $p_1'$ is $\utfutf$-contractible in $M \ba a,b / p_1$.
  Thus $M \ba a,b / p_1,p_1'$ has a $\utfutf$-minor, and $\{p_1'',p_2\}$ is a parallel pair in this matroid.
  Thus $p_1''$ is $\utfutf$-deletable in $M \ba a,b / p_1,p_1'$, and hence in $M \ba a,b$, a contradiction.
  We deduce that, for each $p_2 \in P_2$, there is a circuit contained in $P_1 \cup p_2$ that contains $a'$, as required.
\end{subproof}

We work towards applying \cref{ndtcrux} using $a'$ and $b'$ as given in \cref{ndtendcandidates}.
First we require the following.

\begin{claim}
  \label{localcoconn}
  $\lc^*_{M \ba a,b}(P_1,P_m) = 0$.
\end{claim}
\begin{subproof}
  Suppose $\lc^*_{M \ba a,b}(P_1,P_m) \ge 1$.
  Let $P_i$ be a guts set for some $i \in \{2,3,\dotsc,m-1\}$.
  Then, by the duals of \cref{growpi,pflancoguts}, $\lc^*_{M \ba a,b}(P_i^-,P_i^+) = 1$, so $|P_i|=1$.
  So every guts set has size one.

  First, assume that $P_1$ is a cosegment.
  Then, by \cref{ndtcogutsupper}, $|Q| \le |G|$.
  Since each guts set has size one, and the number of guts sets is one more than the number of coguts sets, there is at most one coguts set of size two. 
  Recall that $m \ge 7$, so there exists a coguts set $P_j$ with $|P_j|=1$, for some $\{j,j'\} = \{3,5\}$.
  By the dual of \cref{pflancoguts}, $P_{j-1} \cup P_j \cup P_{j+1}$ is a triangle.
  But this implies that $|Q| \le |G|-1$ by \cref{ndtcogutsupper}, so every coguts set has size one.
  In particular, $|P_{j'}| = 1$, so $P_{j'-1} \cup P_{j'} \cup P_{j'+1}$ is also a triangle, contradicting that $M \ba a,b$ has at most one triangle.

  Now assume $P_1$ is a $5$-element fan.
  By \cref{ndtcogutsupper}, $|Q| \le |G|-1$, so every guts and coguts set has size one.
  By the dual of \cref{pflancoguts}, $P_{2} \cup P_3 \cup P_{4}$ is a triangle, so $|Q| \le |G|-2$, by \cref{ndtcogutsupper}, a contradiction.
\end{subproof}

\begin{claim}
  \label{ndtgoodtrip}
  Let $a'$ and $b'$ be as given in \cref{ndtendcandidates}, and let $c'$ be a guts element in $P_k$ for some $k \in \{4,6,\dotsc,m-3\}$.
  Suppose $C^*$ is a cocircuit of $M \ba a,b$.
  \begin{enumerate}
    \item $C^* \nsubseteq \{a',b',c'\}$.\label{ndtgt1}
    \item If $\{a',c',b'\} \subseteq C^*$ and $|C^*|=5$, then $C^* = \{a'',a',c',b',b''\}$ for some $a'' \in P_1-a'$ and $b'' \in P_m-b'$.\label{ndtgt2}
    \item If $\{a',c'\} \subseteq C^*$ and $|C^*|=4$, then $C^* = \{a'',a',c',r\}$ for some $a'' \in P_1-a'$ and $r \in P_k^+$.\label{ndtgt3}
    \item If $\{c',b'\} \subseteq C^*$ and $|C^*|=4$, then $C^* = \{\ell,c',b',b''\}$ for some $\ell \in P_k^-$ and $b'' \in P_m-b'$.\label{ndtgt4}
    \item If $\{a',b'\} \subseteq C^*$, then $|C^*| \neq 4$.\label{ndtgt5}
  \end{enumerate}
\end{claim}
\begin{subproof}
  Since $P_1 \cup P_2$ contains a circuit containing $a'$, by \cref{ndtendcandidates}, any cocircuit of $M \ba a,b$ containing $a'$ meets $(P_1 -a')\cup P_2$, by orthogonality.
  Similarly, $P_{m-1} \cup P_m$ contains a circuit containing $b'$, so any cocircuit of $M \ba a,b$ containing $b'$ meets $P_{m-1} \cup (P_m - b')$.
  It follows that $\{a',b',c'\}$ is coindependent in $M \ba a,b$, thus proving \ref{ndtgt1}.

  Observe that $M \ba \{a,b,a',b',c'\}$ has a $\utfutf$-minor, by \cref{ndtendcandidates} and since $c'$ is $\utfutf$-deletable in $M \ba a,b$.
  Suppose $\{a',b',c'\}$ is contained in a $5$-element cocircuit $\{a',b',c',a'',b''\}$.
  Since $M \ba \{a,b,a',b',c'\}$ has a $\utfutf$-minor and $\{a'',b''\}$ is a series pair in this matroid, $a''$ and $b''$ are $\utfutf$-contractible in $M \ba a,b$.
  By \cref{pathdescprops}, $a''$ and $b''$ are not guts elements, so $a'' \in P_1-a$ and $b'' \in P_m-b'$, thus proving \ref{ndtgt2}.
  Similarly, if some pair of elements in $\{a',b',c'\}$ is contained in a $4$-element cocircuit, then the other two elements in this cocircuit are $\utfutf$-contractible in $M \ba a,b$.
  Cases~\ref{ndtgt3} and~\ref{ndtgt4} of the claim then follow from orthogonality.
  For case~\ref{ndtgt5}, if $\{a',b'\}$ is contained in a $4$-element cocircuit $\{a'',a',b',b''\}$ say, 
  then $a'' \in P_1-a'$ and $b'' \in P_m-b'$, by orthogonality.
  But this contradicts \cref{localcoconn} and the dual of \cref{picircuits}.
%
\end{subproof}

\begin{claim}
  \label{ndtgoodtripsupplement}
  Let $a'$ and $b'$ be as given in \cref{ndtendcandidates}, and let $c'$ be a guts element in $P_i$, for some $i \in \{4,6,\dotsc,m-3\}$, such that
  \begin{enumerate}[label=\rm(\Roman*)]
    \item neither $\{a',c'\}$ nor $\{c',b'\}$ is contained in a $4$-element cocircuit of $M \ba a,b$, and\label{ndtgts1}
    \item there is a unique triad containing $c'$, and this triad avoids $P_1 \cup P_m$.\label{ndtgts2}
  \end{enumerate}
  Then $\{a',b',c'\}$ is a delete triple.
\end{claim}
\begin{subproof}
  Observe that $M \ba \{a,b,a',b',c'\}$ has a $\utfutf$-minor, by \cref{ndtendcandidates} and since $c'$ is $\utfutf$-deletable in $M \ba a,b$; hence this matroid is $\utfutf$-fragile.
  By \cref{genfragileconn}, it follows that $M\ba \{a,b,a',b',c'\}$ is $3$-connected up to series classes.

  We work towards an application of \cref{ndtcrux}.
  First, observe that in $M \ba a,b$, the set $\{a',b',c'\}$ is coindependent by \cref{ndtgoodtrip}\ref{ndtgt1}, and $\{a',b'\}$ is not contained in a $4$-element cocircuit by \cref{ndtgoodtrip}\ref{ndtgt5}.

  Suppose $\{a',b',c'\}$ is contained in a $5$-element cocircuit of $M \ba a,b$.
  Then, by \cref{ndtgoodtrip}\ref{ndtgt2}, this cocircuit is $\{a'',a',c',b',b''\}$ for some $a'' \in P_1 - a'$ and $b'' \in P_m$.
  Let $\{\ell,c',r\}$ be an internal triad containing $c'$, with $\ell \in P_i^-$ and $r \in P_i^+$.
  Then, by \ref{ndtgts2}, $\ell \notin P_1$, and $r \notin P_m$. 
  By cocircuit elimination, there is a cocircuit~$C^*$ contained in $\{a'',a',\ell,r,b',b''\}$.
  By the dual of \cref{pflancoguts}, $\lc^*_{M \ba a,b}(P_i^-,P_i^+)=0$.
  But $C^* \subseteq P_i^- \cup P_i^+$, so, by the dual of \cref{picircuits}, either $C^* \subseteq P_i^-$ or $C^* \subseteq P_i^+$.
  Thus either $\{a'',a',\ell\}$ or $\{r,b',b''\}$ is a triad.
  But $\ell \notin P_1$ and $r \notin P_m$, so this contradicts that $P_1$ and $P_m$ are ends of a nice path description.
  So $\{a',b',c'\}$ is not contained in a $5$-element cocircuit.

  It remains only to show that $M \ba \{a,b,a',b',c'\}$ has at least three non-trivial series classes and
  is $3$-connected up to series classes of size at most three.
  Suppose $S'$ is a series pair of $M \ba \{a,b,a',b',c'\}$.
  Then $S' \cup \{a',b',c'\}$ contains a cocircuit $C^*$ in $M \ba a,b$.
  By \cref{ndtgoodtrip}, either $C^*=5$ and $\{a',b',c'\} \subseteq C^*$, or $C^*=4$ with $c' \in C^*$ and $C^* \cap \{a',b'\} = 1$, or $|C^*|=3$ and $|C^* \cap \{a',b',c'\}| = 1$.
  But by the foregoing, and \ref{ndtgts2}, only the latter is possible; that is, every element in a non-trivial series class of $M \ba \{a,b,a',b',c'\}$ is in a triad of $M \ba a,b$ that contains one of $a'$, $b'$, or $c'$.
  Let $S_a$, $S_b$ and $S_c$ be the set of elements in $M\ba a,b$ that are in a triad with $a'$, $b'$, and $c'$, respectively.
  Then, each of $S_a$, $S_b$ and $S_c$ is contained in a series class of $M \ba \{a,b,a',b',c'\}$, and each element in a non-trivial series class of $M \ba \{a,b,a',b',c'\}$ is in $S_a \cup S_b \cup S_c$.

  Observe that $a'$, $b'$, and $c'$ are each in at least one triad,
  so the sets $S_a$, $S_b$, and $S_c$ are non-empty.
  We claim that these three sets have size at most three, and are pairwise disjoint.
  Suppose $P_1$ is a cosegment, so $P_1-a' \subseteq S_a$.
  Since $a'$ is in a circuit contained in $P_1 \cup p_2$ for any $p_2 \in P_2$, any triad containing $a'$ is either contained in $P_1$, or is an internal triad containing a guts element $p_2 \in P_2$, by orthogonality.
  But $a'$ is not in an internal triad, so $S_a = P_1-a'$ when $P_1$ is a cosegment.
  Similarly, $S_b = P_m-b'$.
  Now suppose $P_1$ is a $5$-element fan.
  We may assume that $(f_1,f_2,f_3,a',f_5)$ is a fan ordering of $P_1$, where $\{f_2,f_3,a'\}$ is a triangle.
  Suppose $a'$ is in a triad that also contains some $z \notin P_1$.
  Then, by orthogonality, this triad is $\{a',z,f_2\}$.  But then $(M \ba a,b)^*|(P_1 \cup z) \cong M(K_4)$, contradicting \cref{noMK4}.
  So $S_a = \{f_3,f_5\} \subseteq P_1$ when $P_1$ is a $5$-element fan.
  Now $|S_a| \le 3$ and $|S_b| = 2$.
  By \ref{ndtgts2}, $|S_c| = 2$ and $S_c \cap (P_1 \cap P_m) = \emptyset$, so the sets $S_a$, $S_b$, and $S_c$ are pairwise disjoint.
  It remains to show that the series classes of $M \ba \{a,b,a',b',c'\}$ containing $S_a$, $S_b$, and $S_c$ are distinct.

  We first show that $c' \notin \cocl_{M \ba a,b,a',b'}(S_a \cup S_b)$.
  Suppose that $c' \in \cocl_{M \ba a,b,a',b'}(S_a \cup S_b)$.
  Then $c'$ is in a cocircuit~$D_1$ of $M \ba a,b$ contained in $S_a \cup S_b \cup \{a',b',c'\}$.
  Note that $r^*_{M \ba a,b}(S_a \cup S_b \cup \{a',b'\}) \le 4$, so $|D_1| \le 5$.
  If $D_1$ contains at most two elements in $S_a \cup S_b$, then $S'=D_1-\{a',b',c'\}$ is a series pair in $M \ba \{a,b,a',b',c'\}$, in which case $S' \cup c'$ is a triad.
  But then $S' \subseteq S_c$, a contradiction.
  So $|D_1 \cap (S_a \cup S_b)| \in \{3,4\}$. 
  Since $S_a \cup a' \subseteq P_1$ and $S_b \cup b' \subseteq P_m$, and $c'$ is a guts element, 
  $D_1 \cap P_m \neq \emptyset$.
  By orthogonality, $|D_1 \cap P_m| \neq 1$, so $|D_1 \cap P_m| = 2$.
  We claim that there is a cocircuit $D_2$ with $|D_2| \in \{4,5\}$ and $\{c',b'\} \subseteq D_2 \subseteq P_1 \cup \{a',c',b'\} \cup P_m$.
  If $b' \in D_1$, then we can just let $D_2 = D_1$; so suppose that $b' \notin D_1$.
  Let $s_b \in S_b$.
  By cocircuit elimination, there is a cocircuit $D_2$ contained in $(D_1 \cup b') - s_b$.
  By \cref{localcoconn}, this cocircuit contains $c'$.
  Thus, arguing as for $D_1$, we have that $|D_2 \cap P_m| \ge 2$.
  As $c'$ is a guts element, $D_2 \cap P_1 \neq \emptyset$.
  Now $D_2$ has the claimed properties; in particular, $|D_2| \in \{4,5\}$.
  By \ref{ndtgts1}, $|D_2| = 5$.
  Then $a' \notin D_2$, since no $5$-element cocircuit contains $\{a',c',b'\}$.
  Let $s_a \in S_a$.
  By cocircuit elimination, there is a cocircuit $D_3$ contained in $(D_2 - a') \cup s_a$.
  Arguing as before, $|D_3 \cap P_1|=|D_3 \cap P_m| \ge 2$, so $D_3 = D_2 \triangle \{a',s_1\}$.
  Thus $\{a',b',c'\}$ is contained in a $5$-element cocircuit, a contradiction.
  This proves that $c' \notin \cocl_{M \ba a,b,a',b'}(S_a \cup S_b)$.

  By \cref{localcoconn}, $r^*_{M \ba a,b}(S_a \cup S_b \cup \{a',b'\})=4$, so $r^*_{M \ba a,b,a',b'}(S_a \cup S_b)=2$.
  As $c' \notin \cocl_{M \ba a,b,a',b'}(S_a \cup S_b)$, we have $r^*_{M \ba a,b,a',b',c'}(S_a \cup S_b) = 2$, so the series classes containing $S_a$ and $S_b$ are distinct.

  Now we claim that the series classes containing $S_a$ and $S_c$ are distinct.
  Pick $s_a \in S_a$ and $s_c \in S_c$, and observe that $r^*_{M \ba a,b}(S_a \cup S_c \cup \{a',c'\}) = r^*_{M \ba a,b}(\{a',s_a,c',s_c\})$.
  Suppose $r^*_{M \ba a,b}(S_a \cup S_c \cup \{a',c'\}) \le 3$.
  Then $\{a',s_a,c',s_c\}$ is dependent in $(M \ba a,b)^*$.
  But $\{a',s_a,c',s_c\}$ does not contain a triad of $M \ba a,b$, so $\{a',s_a,c',s_c\}$ is a cocircuit, contradicting \ref{ndtgts1}.
  Thus $r^*_{M \ba a,b}(S_a \cup S_c \cup \{a',c'\}) = 4$.
Suppose $b' \in \cocl_{M \ba a,b,a',c'}(S_a \cup S_c)$.
  Then there is a cocircuit contained in $S_a \cup S_c \cup \{a',b',c'\}$ and containing $b'$.
  But this cocircuit intersects the circuit $P_{m-1} \cup P_m$ in a single element, a contradiction.
  So $b' \notin \cocl_{M \ba a,b,a',c'}(S_a \cup S_c)$.
  Now $r^*_{M \ba a,b,a',b',c'}(S_a \cup S_c) = 2$, so the series classes of $M \ba \{a,b,a',b',c'\}$ containing $S_a$ and $S_c$ are distinct.

  By a similar argument, the series classes containing $S_c$ and $S_b$ are distinct.  This completes the proof.
\end{subproof}

Next, we argue that each guts set of $\mathbf{P}=(P_1,P_2,\dotsc,P_m)$ has size one, except perhaps $P_2$ when $P_1$ is a $4$-cosegment.

\begin{claim}
  \label{nodoubleguts}
  Let $P_i$ be a guts set for some $i \in \{2, \dotsc, m-1\}$ such that if $P_1$ is a $4$-cosegment then $i \neq 2$. 
  Then $|P_i| = 1$.
\end{claim}
\begin{subproof}
  Recall that $|P_{m-1}| = 1$ and if $P_1$ is not a $4$-cosegment then $|P_2| = 1$, so \cref{nodoubleguts} holds when $i = m-1$ or $i=2$. 
  So we may assume that $3 \le i < m-2$.
  Let $a'$ and $b'$ be as given in \cref{ndtendcandidates}.
  Let $P_i = \{c',c''\}$ such that each triad containing $c'$ is disjoint from $P_1$, where such a $c' \in P_i$ exists by \cref{ndttriads}\ref{ndttriadsii}.
  Towards an application of \cref{ndtgoodtripsupplement}, it remains to show that there is a unique triad containing $c'$, which avoids $P_1 \cup P_m$, and neither $\{a',c'\}$ nor $\{c',b'\}$ is contained in a $4$-element cocircuit of $M \ba a,b$.

  Suppose $T_1^*$ and $T_2^*$ are distinct triads containing $c'$.
  Then, by \cref{ndttriadseq,ndtdisjtriads}, $T_1^* = \{\ell_1,c',r_1\}$ and $T_2^* = \{\ell_2,c',r_2\}$ for distinct $\ell_1,\ell_2 \in P_i^-$ and distinct $r_1,r_2 \in P_i^+$.
  By cocircuit elimination, there is a cocircuit contained in $\{\ell_1,\ell_2,r_1,r_2\}$, which has size at least~$3$, since $M \ba a,b$ is $3$-connected.  But this contradicts the dual of \cref{picircuits}.
  We deduce that there is a unique triad containing $c'$.

  Suppose $\{a',c'\}$ is contained in a $4$-element cocircuit.
  Then this cocircuit is $\{a',a'',c',r'\}$ where $a'' \in P_1 - a'$ and $r' \in P_i^+$, by \cref{ndtgoodtrip}\ref{ndtgt3}. 
  Let $\{\ell,c',r\}$ be the internal triad containing $c'$, with $\ell \in P_i^- - P_1$ and $r \in P_i^+ - P_m$.
  Then, by cocircuit elimination, there is a cocircuit~$C^*$ contained in $\{a',a'',\ell,r,r'\}$.
  By the dual of \cref{pflancoguts}, $\lc^*_{M \ba a,b}(P_i^-,P_i^+)=0$, so, by the dual of \cref{picircuits}, either $C^* \subseteq P_i^-$ or $C^* \subseteq P_i^+$.
  Thus $\{a'',a',\ell\}$ is a triad.
  But then $P_1 \cup \ell$ is a $4$-cosegment, contradicting that $P_1$ is an end of a nice path description.
  By a symmetric argument, $\{c',b'\}$ is not contained in a $4$-element cocircuit.

  Now, by \cref{ndtgoodtripsupplement}, $M$ has a delete triple, a contradiction.
  This proves \cref{nodoubleguts}.
\end{subproof}

\begin{claim}
  \label{onedoublecoguts}
  At most one coguts set of $\mathbf{P}$ has size more than one, and if a coguts set of size more than one exists, then it has size two.
  Moreover,
  \begin{enumerate}
    \item if $P_1$ is a $4$-cosegment, then $|P_2|=|P_3|$ and $M \ba a,b$ has no triangles; and\label{odcg1}
    \item if $M \ba a,b$ has a triangle, then $|P_i|=1$ for each $i \in \{2,3,\dotsc,m-1\}$.\label{odcg2}
  \end{enumerate}
\end{claim}
\begin{subproof}
%
  Suppose $P_1$ is not a $4$-cosegment.
  By \cref{nodoubleguts}, every guts set has size one.
  By \cref{ndtcogutsupper}, $|Q| \le |G|$.
  So either every coguts set has size one, in which case $|Q| =|G|-1$,
  or all but one coguts set has size one, and this coguts set has size two, in which case $|Q| =|G|$.
  If $M \ba a,b$ has a triangle, 
  then $|Q| \le|G|-1$ by \cref{ndtcogutsupper}, so every coguts set
  also
  has size one.

  Now suppose $P_1$ is a $4$-cosegment.
  Then $|Q| \le|G|-1$, by \cref{ndtcogutsupper}.
  Again by \cref{nodoubleguts}, every guts set except perhaps $P_2$ has size one.
  Thus, if $|P_2| = 1$, then every coguts set has size one and $|Q| = |G|-1$; in particular, $|P_2|=|P_3|=1$ and $M \ba a,b$ has no triangles, the latter by \cref{ndtcogutsupper}.
  On the other hand, if $|P_2| = 2$, then at most one coguts set has size two; by \cref{ndttriadseq,ndtdisjtriads} and since $\mathbf{P}$ is left-justified, 
  we have $|P_3|=2$,
  so again $|Q| = |G|-1$, and $M \ba a,b$ has no triangles by \cref{ndtcogutsupper}.
\end{subproof}

By \cref{onedoublecoguts}, there is at most one coguts set with size two.
If such a coguts set exists,
let $j \in \{3,5,\dotsc,m-2\}$ such that $|P_j|=2$; otherwise, let $j=0$.
We work towards applying \cref{ndtgoodtripsupplement}, first when $P_1$ is a triad, and then when it is not.  First we prove one more claim that holds in either case.

\begin{claim}
  \label{ndtinternaltris}
  Let $p_k \in P_k$ be a guts element, for some $k \in \{2,4,\dotsc,m-1\}$, and let $T^*$ be an internal triad containing $p_k$.
  Then $T^* = \{\ell_k,p_k,r_k\}$ for some $\ell_k \in P_k^-$ and $r_k \in P_{k+1}$.
  Moreover, if
  \begin{enumerate}[label=\rm(\Roman*)]
    \item $P_1$ is a triad, and $j=0$ or $k \le j+1$;\label{ndtit1}
    \item $P_1$ is a $4$-cosegment, $|P_2|=|P_3|=2$, and $k \in \{2,4\}$; or\label{ndtit2}
    \item $P_1$ is a $5$-element fan;\label{ndtit3}
  \end{enumerate}
  then $\ell_k \in P_{k-1}$.
\end{claim}
\begin{subproof}
  By \cref{ndttriadseq}, we may assume that $T^* = \{\ell_k,p_k,r_k\}$ for some $\ell_k \in P_k^-$ and $r_k \in P_k^+$.
  Since $\mathbf{P}$ is left-justified, either $r_k \in P_{k+1}$ or $r_k \in P_m$.
  If $k < m-1$, then $r_k \notin P_m$ by \cref{ndttriads}.
  Thus $r_k \in P_{k+1}$ for each even $k$.
  
  We first consider when \ref{ndtit3} holds.
  Let $P_1$ be a $5$-element fan $(f_1,f_2,f_3,f_4,f_5)$ and let $\{\ell_{i},p_{i},r_{i}\}$ is an internal triad for each guts element~$p_{i}$.
  Suppose $\ell_t \in P_1$ for some even $t \ge 4$.
  By orthogonality and \cref{ndtdisjtriads}, we may assume that $f_1 = \ell_2$ and $f_5 = \ell_t$.
  By \cref{onedoublecoguts}\ref{odcg2}, we have $P_{i} = \{p_{i}\}$ for each $i \in \{2,3,\dotsc,m-1\}$.
  Observe that $r^*_{M \ba a,b}(P_1 \cup p_2) = 4$ and $r^*_{M \ba a,b}(P_2^+) = r(M \ba a,b)-2$.
  Now $f_5 \in \cocl_{M \ba a,b}(P_2^+)$, so $\cocl_{M \ba a,b}(P_2^+ \cup \{f_4,f_5\})$ is contained in a cohyperplane.  As $f_3 \in \cocl_{M \ba a,b}(P_2^+ \cup \{f_4,f_5\})$, the set $\{f_1,f_2,p_2\}$ is a triangle.
  But then $P_1 \cup p_2$ is a $6$-element fan, contradicting that $P_1$ is an end of a nice path description.

  Now suppose \ref{ndtit1} or \ref{ndtit2} holds.  It remains to show that $\ell_k \in P_{k-1}$.
  This is clear if $k = 2$.
  Suppose $P_1$ is a $4$-cosegment, $|P_2|=2$, and $k=4$.
  By the dual of \cref{pflancoguts}, $\lc^*_{M \ba a,b}(P_1,P_2^+) = 0$.
  Thus, by the dual of \cref{picircuits}, $\ell_4 \in P_2^+$, so $\ell_4 \in P_3$.
  Now we may assume that $P_1$ is a triad and $j=0$ or $k \le j+1$.
  Let $i$ be even, with $2 < i \le k$, and suppose for all $i'$ such that $2 \le i' < i$, we have $\ell_{i'} \in P_{i'-1}$.
  Now $\ell_{i} \notin P_1$, by \cref{ndttriads}\ref{ndttriadsi}.
  Observe, for each even $i'$ with $2 < i' < i$, we have $i' \le j-1$, since $i' < i \le k \le j+1$ where $i'$ and $k$ are even.
  So, for such an $i'$, we have $P_{i'-1}=\{\ell_{i'}\}$, and thus $\ell_{i} \notin P_{i'-1}$, by \cref{ndtdisjtriads}.
  By orthogonality, $\ell_{i}$ is not a guts element, so $\ell_{i} \in P_{i-1}$.
  The claim follows by induction.
\end{subproof}

\begin{claim}
  \label{ndtendgame}
  Suppose $P_1$ is a triad. 
  Then \labelcref{ndtendgameoutcomes16}\labelcref{ndtendgameoutcome16a} holds.
\end{claim}
\begin{subproof}
  Recall that $\mathbf{P}=(P_1,P_2,\dotsc,P_m)$ is a nice path description of $M \ba a,b$ with $m$ odd, where $P_2$ and $P_{m-1}$ are guts sets,
  and $|P_i|=1$ for every $i \in \{2,3,\dotsc,m-1\} - j$.
  Let $P_i = \{p_i\}$ for all $i \in \{2,3,\dotsc,m-1\} - j$ and, if $j \neq 0$, let $P_j = \{p_j,p_j'\}$.

  Recall also that $M \ba a,b$ has at most one triangle.
  If $M \ba a,b$ has a triangle, then, by \cref{onedoublecoguts}\ref{odcg2}, $j=0$, and, up to replacing $(P_1,P_2,\dotsc,P_m)$ with its reversal, $\{p_{m-3},p_{m-2},p_{m-1}\}$ is not a triangle.  
  So we may assume that $\{p_{m-3},p_{m-2},p_{m-1}\}$ is independent.
  Since $13 \le |E(M \ba a,b)| \le m+5$, and $m$ is odd, we have $m \ge 9$.
  We distinguish the following cases:
  \begin{enumerate}[label=\rm(\Roman*)]
    \item $m=9$, and $j=5$.\label{ndteg1}
    \item $m \ge 11$ and $j=5$.\label{ndteg2}
    \item $m=9$ and $j=7$.\label{ndteg3}
    \item None of \ref{ndteg1} to \ref{ndteg3} holds; that is, $j \notin \{5,7\}$, or $m \ge 11$ and $j =7$.\label{ndteg4}
  \end{enumerate}
  Note that $M \ba a,b$ has no triangles in cases~\ref{ndteg1} to \ref{ndteg3}; in the case that $M \ba a,b$ has a triangle, case~\ref{ndteg4} holds.

  We first handle cases~\ref{ndteg2} to \ref{ndteg4}, before returning to case~\ref{ndteg1}.
  Let $a'$ and $b'$ be as given in \cref{ndtendcandidates}.
  In cases~\ref{ndteg2} and \ref{ndteg3} we let $c' = p_6$; in case~\ref{ndteg4} we let $c'=p_4$; whereas in case~\ref{ndteg1}, $c' \in \{p_4,p_6\}$ as appropriate.
  Choose $k \in \{4,6\}$ so that $c' = p_k$.
  We work towards an application of \cref{ndtgoodtripsupplement} with the elements $a',b',c'$; it remains to show that neither $\{a',c'\}$ nor $\{c',b'\}$ is contained in a $4$-element cocircuit of $M \ba a,b$, and there is a unique triad containing $c'$, which avoids $P_1 \cup P_m$.

  \smallskip

  Suppose case~\ref{ndteg3} holds, so $m=9$, $j=7$, and $k=6$.
  Note that, by \cref{ndtdisjtriads,ndtinternaltris}, $\{p_5,c',p_7\}$ is the unique triad containing $c'$, up to swapping the labels on $p_7$ and $p_7'$.
  As $M \ba a,b$ has no triangles, $\{p_2,p_3,p_4\}$ is independent, so $\lc^*_{M \ba a,b}(P_1,P_4^+) = 0$ by the dual of \cref{pflantriad}.
  By \cref{ndtgoodtrip}\ref{ndtgt3}, if there is a $4$-element cocircuit containing $\{a',c'\}$, then it avoids $\{p_2,p_3,p_4\}$; hence, by the dual of \cref{picircuits}, no such cocircuit exists.
  Suppose $\{c',b'\}$ is contained in a $4$-element cocircuit.
  By \cref{ndtgoodtrip}\ref{ndtgt4} and orthogonality, this cocircuit is $\{\ell,c',b',b''\}$ for some $b'' \in P_m -b'$ and $\ell \in \{p_3,p_5\}$.
  If $\ell = p_3$, then, by cocircuit elimination with the triad $\{p_3,p_4,p_5\}$ there is also a cocircuit contained in $\{p_4,p_5,c',b',b''\}$, which (again by orthogonality) does not contain $p_4$.
  So we may assume that $\ell=p_5$.
  Recall that $\{p_5,c',p_7\}$ is a triad.
  By cocircuit elimination, $\{c',p_7,b',b''\}$ contains a cocircuit.
  As $c'$ is a guts element, $c' \notin \cocl_{M \ba a,b}(\{p_7,b',b''\})$.
  Thus $\{p_7,b',b''\}$ is a triad.
  But $p_7 \notin \cocl_{M \ba a,b}(P_m)$, since $P_m$ is an end of a nice path description, so this is contradictory.
  So $\{c',b'\}$ is not contained in a $4$-element cocircuit.
  Thus, by \cref{ndtgoodtripsupplement}, $M$ has a delete triple, a contradiction.

  \smallskip

  Now assume we are in case~\ref{ndteg2} or~\ref{ndteg4}.
  Observe that $j=0$ or $k \le j+1$ in either case, so, by \cref{ndtdisjtriads,ndtinternaltris}, there is a unique triad containing $c'$, which we may assume is $\{p_{k-1},c',p_{k+1}\}$, up to switching the labels on $p_{k-1}$ and $p_{k-1}'$ when $j=k-1$.
  We claim that $\{a',c'\}$ is not contained in a $4$-element cocircuit.
  Towards a contradiction, suppose $\{a',c'\}$ is contained in a $4$-element cocircuit~$C^*$.
  Then $C^* = \{a'',a',c',r\}$ with $a'' \in P_1 - a'$ and $r \in P_k^+$, by \cref{ndtgoodtrip}\ref{ndtgt3}.
  Since $\mathbf{P}$ is left-justified, either $r \in P_{k+1}$ or $r \in P_m$.
  But if $r \in P_m$, then $C^*$ intersects the circuit $P_m \cup p_{m-1}$ in a single element, contradicting orthogonality.
  So $r \in P_{k+1}$.
  Note that, in either case, $j \neq k+1$, so $|P_{k+1}|=1$ and $p_{k+1} = r$.
  Recall that $\{p_{k-1},c',p_{k+1}\}$ is a triad.
  By cocircuit elimination with $C^*$, there is a cocircuit contained in $\{a'',a',p_{k-1},c'\}$.
  But $c' \notin \cocl_{M \ba a,b}(\{a'',a',p_{k-1}\})$, since $\{a'',a',p_{k-1}\} \subseteq P_k^-$ and $P_k=\{c'\}$ is a guts set, so $\{a'',a',p_{k-1}\}$ is a triad of $M \ba a,b$.
  Since $a'' \in P_1 - a'$,
  the set $P_1 \cup p_{k-1}$ is a $4$-cosegment, contradicting that $P_1$ is an end of a nice path description.
  We deduce that $\{a',c'\}$ is not contained in a $4$-element cocircuit.

  Suppose $\{c',b'\}$ is contained in a $4$-element cocircuit~$C^*$. Then $C^*=\{c'',c',b',b''\}$ for some $b'' \in P_m -b'$ and $c'' \in P_k^-$, by \cref{ndtgoodtrip}\ref{ndtgt4}.
  By the dual of \cref{picircuits}, the existence of $C^*$ implies that $\lc^*_{M \ba a,b}(P_{k+1}^-,P_m) \ge 1$.
  Recall that $\{p_{m-3},p_{m-2},p_{m-1}\}$ is not a triangle, so $\lc^*_{M \ba a,b}(P_{m-3}^-,P_m) =0$, by the dual of \cref{pflantriad}.
  But then, as $k \le m-5$ and by the dual of \cref{growpi}, $\lc^*_{M \ba a,b}(P_{k+1}^-,P_m) \le \lc^*_{M \ba a,b}(P_{m-3}^-,P_m) =0$, a contradiction.
  We deduce that $\{c',b'\}$ is not contained in a $4$-element cocircuit.
  Now, by \cref{ndtgoodtripsupplement}, $M$ has a delete triple, a contradiction.

  \smallskip

  It remains only to consider case~\ref{ndteg1}, where $m=9$, $j=5$ and $M \ba a,b$ has no triangles.
  Let $\mathbf{P}' = (P_1',P_2'\dotsc,P_m')$ be the (left-justified) reversal of $\mathbf{P} = (P_1,P_2,\dotsc,P_m)$, where $P_1' = P_m$ and $P_1 = P_m'$.
  We may assume $|P_5'|=2$, for otherwise case~\ref{ndteg4} applies for $\mathbf{P}'$.
  Thus $\mathbf{P}$ is both left- and right-justified; in particular, $\{p_1,p_2,p_3\}$ and $\{p_7,p_8,p_9\}$ are triads for some $p_1 \in P_1$ and $p_9 \in P_9$.
  Up to swapping the labels on $p_5$ and $p_5'$, we may assume that $\{p_3,p_4,p_5\}$ is a triad, and this is the unique triad containing $p_4$, by \cref{ndtdisjtriads,ndtinternaltris}.
  Let $q \in \{p_5,p_5'\}$ such that $\{q,p_6,p_7\}$ is a triad, and note that this is the unique triad containing $p_6$, by \cref{ndtdisjtriads,ndtinternaltris}.
  Since $\lc^*_{M \ba a,b}(P_1,P_4^+) = 0$ and $\lc^*_{M \ba a,b}(P_6^-,P_9) = 0$, by the dual of \cref{pflantriad}, it follows from \cref{picircuits,ndtgoodtrip} that there are no $4$-element cocircuits containing $\{a',p_6\}$ or $\{p_4,b'\}$.
  Thus, if $\{a',p_4\}$ is not contained in a $4$-element cocircuit, or $\{p_6,b'\}$ is not contained in a $4$-element cocircuit, then we can apply \cref{ndtgoodtripsupplement}, with $c' = p_4$ or $c' = p_6$ respectively, to deduce that $M$ has a contradictory delete triple.
  So we may assume that $\{a',p_4\}$ and $\{p_6,b'\}$ are contained in $4$-element cocircuits.
  Let the former cocircuit be $\{a'',a',p_4,r\}$.
  Then $a'' \in P_1-a'$ and $r \in \{p_5,p_5'\}$, due to the left-justification of $\mathbf{P}$ and \cref{ndtgoodtrip}\ref{ndtgt3}.
  If $r = p_5$, then, by cocircuit elimination with the triad $\{p_3,p_4,p_5\}$, the set $\{a'',a',p_3,p_4\}$ contains a cocircuit.
  Since $p_4$ is a guts element, $\{a'',a',p_3\}$ is a triad, a contradiction.
  So $r = p_5'$.
  Now $\{a'',a',p_4,p_5'\}$ is a cocircuit.
  By cocircuit elimination with the triad $P_1$, the set $\{p_1',p_1'',p_4,p_5'\}$ is a cocircuit for any pair $\{p_1',p_1''\} \subseteq P_1$.
  By a symmetric argument, after letting $P_5 = \{q,q'\}$ such that the internal triad containing $p_6$ is $\{q,p_6,p_7\}$, the set $\{q',p_6,p_9',p_9''\}$ is a cocircuit for any pair $\{p_9',p_9''\} \subseteq P_9$.
  So \labelcref{ndtendgameoutcomes16}\labelcref{ndtendgameoutcome16a} holds, thus completing the proof of \cref{ndtendgame}.
\end{subproof}

Finally, we handle the case where one end of $\mathbf{P}$ is a $4$-cosegment or a $5$-element fan.

\begin{claim}
  \label{ndtendgame2}
  If $P_1$ is not a triad, then \labelcref{ndtendgameoutcomes16}\labelcref{ndtendgameoutcome16b} holds.
\end{claim}
\begin{subproof}
  Suppose $P_1$ is not a triad, so $P_1$ is a $5$-element fan or a $4$-cosegment.
  Recall that $\mathbf{P}=(P_1,P_2,\dotsc,P_m)$ is a nice path description of $M \ba a,b$ with $m$ odd, where $P_2$ and $P_{m-1}$ are guts sets, and $m \ge 7$.
  By \cref{nodoubleguts,onedoublecoguts}, if $P_1$ is a $5$-element fan, then every guts and coguts set has size one; whereas if $P_1$ is a $4$-cosegment, then every guts and coguts set except perhaps $P_2$ and $P_3$ has size one, and $|P_2|=|P_3| \in \{1,2\}$.
  For all $i \in \{2,3,\dotsc,m-1\}$, let $P_i = \{p_i\}$ if $|P_i|=1$, otherwise let $P_i = \{p_i,p_i'\}$.

  We distinguish the following cases:
  \begin{enumerate}[label=\rm(\Roman*)]
    \item $|P_2| = |P_3| = 2$, and $m \ge 9$.\label{ndtegg1}
    \item $|P_2| = |P_3| = 2$, and $m = 7$.\label{ndtegg2}
    \item $|P_2| = |P_3| = 1$ and $P_1$ is a $5$-element fan.\label{ndtegg3}
    \item $|P_2| = |P_3| = 1$, $P_1$ is a $4$-cosegment, and $m \ge 11$.\label{ndtegg4}
    \item $|P_2| = |P_3| = 1$, $P_1$ is a $4$-cosegment, and $m = 9$.\label{ndtegg5}
  \end{enumerate}
  Note that if $P_1$ is a $5$-element fan, then $|P_2| = |P_3| = 1$, by \cref{ndt5eltfan}. 
  So only in case~\ref{ndtegg3} is $P_1$ a $5$-element fan.
  Moreover, by \cref{onedoublecoguts}\ref{odcg1}, if $M \ba a,b$ has a triangle, then it is a triangle of $P_1$, where $P_1$ is a $5$-element fan; so only in case~\ref{ndtegg3} does $M \ba a,b$ have any triangles.
  Observe also that when cases~\ref{ndtegg1} to \ref{ndtegg3} do not hold, then, since $13 \le |E(M \ba a,b)| \le m+5$ and $m$ is odd, we have $m \ge 9$.
  So these five cases are exhaustive.

  Let $a'$ and $b'$ be as given in \cref{ndtendcandidates}.
  In case~\ref{ndtegg1} we let $c'=p_{m-5}$;
  in case~\ref{ndtegg2} and~\ref{ndtegg3} we let $c'=p_4$;
  in case~\ref{ndtegg4} we let $c'=p_6$;
  while in case~\ref{ndtegg5}, $c' \in \{p_4,p_6\}$ as appropriate.
  We work towards an application of \cref{ndtgoodtripsupplement} with the elements $a',b',c'$; it remains to show that neither $\{a',c'\}$ nor $\{c',b'\}$ is contained in a $4$-element cocircuit of $M \ba a,b$, and there is a unique triad containing $c'$, which avoids $P_1 \cup P_m$.

  \smallskip

  Firstly we address cases~\ref{ndtegg1} and~\ref{ndtegg2}.
  Recall that $c' = p_{m-5}$ in case~\ref{ndtegg1}, and $c'=p_4$ in case~\ref{ndtegg2}.
  In either case, $c' \in P_3^+$.
  Since $|P_2|=2$, we have $\lc^*_{M \ba a,b}(P_1,P_2^+)=0$, by the dual of \cref{pflancoguts}.
  By \cref{ndtgoodtrip}\ref{ndtgt3}, any $4$-element cocircuit containing $\{a',c'\}$ avoids $P_2$; so, by the dual of \cref{picircuits}, no such cocircuit exists.
  Note also that, by \cref{ndtdisjtriads,ndtinternaltris}, there is a unique triad $\{\ell_c,c',r_c\}$ containing $c'$ where $\ell_c \in P_{m-5}^-$ and $r_c=p_{m-4}$ in case~\ref{ndtegg1}, and $\ell_c \in P_4^-$ and $r_c=p_5$ in case~\ref{ndtegg2}. 
  In either case, $\ell_c \notin P_1$, by \cref{picircuits}, since $\lc^*_{M \ba a,b}(P_1,P_2^+)=0$.
  Consider case~\ref{ndtegg1}.
  It remains only to show that $\{c',b'\}$ is not contained in a $4$-element cocircuit.
  As $\{p_{m-3},p_{m-2},p_{m-1}\}$ is independent, $\lc^*_{M \ba a,b}(P_{m-3}^-,P_m) = 0$ by the dual of \cref{pflantriad}.
  As $c' \in P_{m-3}^-$ and $b' \in P_m$, and by \cref{ndtgoodtrip}\ref{ndtgt4}, there is no $4$-element cocircuit containing $\{c',b'\}$, and hence $M$ has a delete triple, by \cref{ndtgoodtripsupplement}, a contradiction.

  Now consider case~\ref{ndtegg2}, where $c'=p_4$ and $m=7$.
  We will show that \labelcref{ndtendgameoutcomes16}\labelcref{ndtendgameoutcome16b} holds.
  If $\{c',b'\}$ is not contained in a $4$-element cocircuit, then $M$ has a contradictory delete triple, by the foregoing and \cref{ndtgoodtripsupplement}.
  So let $C^*$ be a $4$-element cocircuit containing $\{c',b'\}$.
  Then $C^* = \{\ell,c',b',b''\}$ with $b'' \in P_m - b'$ and $\ell \in P_{4}^-$, by \cref{ndtgoodtrip}\ref{ndtgt4}.
%
  Recall that there is a triad $\{\ell_c,p_4,p_5\}$ with $\ell_c \in P_4^- -P_1$.
  By orthogonality, $\ell_c \notin P_2$, so we may assume that $\{p_3,p_4,p_5\}$ is a triad, up to swapping $p_3$ and $p'_3$.
  By \cref{ndtinternaltris,ndtdisjtriads,ndtendcandidates}, we may assume that $\{p_1,p_2,p_3\}$ and $\{p_1',p_2',p_3'\}$ are internal triads, where $P_1 = \{a',a'',p_1,p_1'\}$, up to swapping the labels on $p_2$ and $p_2'$.
  Observe that $p_2$ is in a circuit contained in $P_1 \cup p_2$, and $p_2'$ is in a circuit contained in $P_1 \cup p_2'$, where neither of these circuits is a triangle.
  If there is some element of $P_1$ that both of these circuits avoid, then, by circuit elimination, there is a circuit contained in $P_1 \cup P_2$ that avoids two elements of $P_1$, contradicting orthogonality.
  So $P_1 \cup P_2$ is the union of two circuits.
  Now, by orthogonality, $\ell \notin P_1 \cup P_2$.
  So $\ell \in \{p_3,p_3'\}$.
  If $\ell = p_3$, then, by cocircuit elimination with the triad $\{p_3,p_4,p_5\}$, there is a cocircuit contained in $\{p_4,p_5,b',b''\}$.
  But $p_4 \notin \cocl_{M \ba a,b}(\{p_5,b',b''\})$, since $\{p_5,b',b''\} \subseteq P_4^+$ and $p_4$ is a guts element, so $\{p_5,b',b''\}$ is a triad of $M \ba a,b$, in which case $P_m \cup p_5$ is a $4$-cosegment, contradicting that $P_m$ is an end of a nice path description.
  So $\ell = p_3'$ and $\{p_3',p_4,b',b''\}$ is a cocircuit.
  By cocircuit elimination with the triad $P_7$, orthogonality, and the fact that $P_7$ is an end of a nice path description, we deduce $\{p_3',p_4,p_7,p_7'\}$ is a cocircuit for any pair $\{p_7,p_7'\} \subseteq P_7$.
  It remains to show that $\{p_5,p_6,p_7\}$ is a triad. 
  If $\{p_3',p_6,p_7\}$ is a triad, then the corank-$5$ set $P_m \cup \{p_6, p_2',a'\}$ cospans $H^* = E(M \ba a,b) - \{p_2,p_3,p_4\}$, so $H^*$ is contained in a cohyperplane.  But then $\{p_2,p_3,p_4\}$ is a circuit, 
  a contradiction.
  By \cref{ndtdisjtriads} and orthogonality, $\{p_5,p_6,p_7\}$ is a triad.
  So \labelcref{ndtendgameoutcomes16}\labelcref{ndtendgameoutcome16b} holds.

  \smallskip

  Now consider case~\ref{ndtegg3}, where $P_1$ is a $5$-element fan.
  Recall that $c' = p_4$, and the only triangle of $M \ba a,b$ is contained in $P_1$.
  Let $(f_1,f_2,f_3,f_4,f_5)$ be a fan ordering of $P_1$.
  First, observe that $\{p_3,c',p_5\}$ is the unique triad containing $c'$, by \cref{ndtinternaltris}.
  Next we show that $\{a',c'\}$ is not contained in a $4$-element cocircuit.
  Towards a contradiction, let $C^*$ be a $4$-element cocircuit containing $\{a',c'\}$.
  Then $C^* = \{a'',a',c',p_5\}$ with $a'' \in P_1 - a'$, by \cref{ndtgoodtrip}\ref{ndtgt3} and since $\mathbf{P}$ is left-justified.
  By cocircuit elimination with the triad $\{p_3,c',p_5\}$, the set $\{a'',a',p_3,c'\}$ contains a cocircuit.
  By orthogonality, this cocircuit is the triad $\{a'',a',p_3\}$.
  Since $a'$ is a spoke of the fan $P_1$, by \cref{ndtendcandidates}, we may assume that $a' =f_2$.
  Note that $a'' \neq f_3$, for otherwise $\{a'',a',p_3,f_1\}$ is a cosegment and $\{a'',a',f_4\}$ is a triangle, a contradiction.
  So, by orthogonality, $a'' = f_4$.
  Now $(M \ba a,b)^*|(P_1 \cup p_3) \cong M(K_4)$, contradicting \cref{noMK4}.
  We deduce that $\{a',c'\}$ is not contained in a $4$-element cocircuit.

  We next show that $\{c',b'\}$ is not contained in a $4$-element cocircuit.
  Assume that $m \ge 9$.
  Then $c' \in P_{m-3}^-$.
  Since $\{p_{m-3},p_{m-2},p_{m-1}\}$ is independent, $\lc^*_{M \ba a,b}(P_{m-3}^-,P_m) = 0$, by the dual of \cref{pflantriad}.
  Since $b' \in P_m$, it follows, by \cref{ndtgoodtrip}\ref{ndtgt4} and the dual of \cref{picircuits}, that $\{c',b'\}$ is not contained in a $4$-element cocircuit.
  Now assume that $m=7$ and suppose that $\{c',b'\}$ is contained in a $4$-element cocircuit.
  Then this cocircuit is $\{\ell,c',b',b''\}$ with $b'' \in P_m - b'$ and $\ell \in P_4^-$, by \cref{ndtgoodtrip}\ref{ndtgt4}.
  Recall that $\{p_3,c',p_5\}$ is a triad.
  If $\ell = p_3$, then, by cocircuit elimination and orthogonality, $\{p_5,b',b''\}$ is a triad, so $P_m$ is not coclosed, a contradiction.
  So $\ell \neq p_3$.
  Without loss of generality,
  $\{f_1,p_2,p_3\}$ is a triad.
  By \cref{ndtendcandidates}, $a' \in \{f_2,f_4\}$.
  Now, by orthogonality, \cref{ndtdisjtriads}, and the foregoing, $\ell =f_5$.
  Then $f_5 \in \cocl_{M \ba a,b}(P_2^+)$, where $r^*_{M \ba a,b}(P_2^+) = r^*(M \ba a,b)-2$.
  Thus, $\cocl_{M \ba a,b}(P_2^+ \cup \{f_4,f_5\})$ is contained in a cohyperplane.  As $f_3 \in \cocl_{M \ba a,b}(P_2^+ \cup \{f_4,f_5\})$, the set $\{f_1,f_2,p_2\}$ is a triangle, a contradiction.
  So $\{c',b'\}$ is not contained in a $4$-element cocircuit.
  By \cref{ndtgoodtripsupplement}, $M$ has a delete triple, a contradiction.

  \smallskip

  Next we assume that case~\ref{ndtegg4} or \ref{ndtegg5} holds, so $m \ge 9$.
  In case~\ref{ndtegg5}, for now let $c' = p_6$.
  Since $\{p_2,p_3,p_4\}$ is independent, $\lc^*_{M \ba a,b}(P_1,P_{4}^+) =0$ by the dual of \cref{pflantriad}.
  Observe that, by \cref{ndtdisjtriads,ndtinternaltris}, there is a unique triad containing $p_6$, which also contains $p_7$, and avoids $P_1$, by the dual of \cref{picircuits}.
  As $a' \in P_1$ and $c' \in P_4^+$, the dual of \cref{picircuits} and \cref{ndtgoodtrip}\ref{ndtgt3} imply that $\{a',c'\}$ is not contained in a $4$-element cocircuit.
  We first consider case~\ref{ndtegg4}; it remains to show that $\{c',b'\}$ is not contained in a $4$-element cocircuit.
  Since $\{p_{m-3},p_{m-2},p_{m-1}\}$ is independent, $\lc^*_{M \ba a,b}(P_{m-3}^-,P_{m}) =0$ by the dual of \cref{pflantriad}, with $b' \in P_m$.
  As $c' \in P_{m-3}^-$, we have that $\{c',b'\}$ is not contained in a $4$-element cocircuit, by \cref{ndtgoodtrip}\ref{ndtgt4} and the dual of \cref{picircuits}.
  By \cref{ndtgoodtripsupplement}, $M$ has a delete triple, a contradiction.

  Now consider case~\ref{ndtegg5}, where $m=9$.
  We may assume that $\{p_6,b'\}$ is in a $4$-element cocircuit, for otherwise, by the foregoing, we can apply \cref{ndtgoodtripsupplement} with $c' = p_6$ to obtain a contradictory delete triple.
  By the dual of \cref{picircuits} and orthogonality, $C_1^* = \{\ell,p_6,b',b''\}$ is a cocircuit for $\ell \in \{p_3,p_5\}$ and $b'' \in P_m-b'$.
  Observe also that, by \cref{ndtdisjtriads,ndtinternaltris}, there is a unique triad containing $p_4$, which also contains $p_5$.
  It follows that this triad is either $\{p_3,p_4,p_5\}$, or $\{p_1,p_4,p_5\}$ for some $p_1 \in P_1-a'$. 

  Suppose $\{p_3,p_4,p_5\}$ is a triad.
  As $\lc^*_{M \ba a,b}(P_{6}^-,P_{m}) = 0$, by the dual of \cref{pflantriad}, with $p_4 \in P_{6}^-$ and $b' \in P_m$, the pair $\{p_4,b'\}$ is not contained in a $4$-element cocircuit, by \cref{ndtgoodtrip}\ref{ndtgt4} and the dual of \cref{picircuits}.
  So we may assume that $\{a',p_4\}$ is in a $4$-element cocircuit, for otherwise we can apply \cref{ndtgoodtripsupplement} with $c'=p_4$.
  Let $C_2^*$ be the $4$-element cocircuit containing $\{a',p_4\}$.
  Then, by \cref{ndtgoodtrip} and since $\mathbf{P}$ is left-justified, $C_2^*=\{a'',a',p_4,p_5\}$ for $a'' \in P_1 - a'$.
  By cocircuit elimination with $\{p_3,p_4,p_5\}$, there is a cocircuit contained in $\{a'',a',p_3,p_4\}$.
  But then, since $p_4$ is a guts element, $\{a'',a',p_3\}$ is a triad, contradicting that $P_1$ is an end of a nice path description.

  So we may assume that $\{p_1,p_4,p_5\}$ is a triad, where $P_1 = \{a',a'',p_1,p_1'\}$. 
  By \cref{ndtdisjtriads} and the left-justification of $\mathbf{P}$, the internal triad containing $p_2$ is, without loss of generality, $\{p_1',p_2,p_3\}$.
  Now $\{p_1',p_2,p_3\}$ and $\{p_1,p_4,p_5\}$ are triads, and $C_1^*=\{\ell,p_6,b',b''\}$ is a cocircuit.
  Since $\mathbf{P}$ is left-justified, $\{p_3,p_4\}$ is contained in a circuit which is, in turn, contained in $P_3^- \cup \{p_3,p_4\}$. 
  Hence, by orthogonality, $\ell = p_5$, so $C_1^*=\{p_5,p_6,b',b''\}$.
  If $\{p_5,p_6,p_7\}$ is a triad, then, by cocircuit elimination, there is a cocircuit contained in $\{p_6,p_7,b',b''\}$.
  Then, by orthogonality, $\{p_7,b',b''\}$ is a triad, so $P_m$ is not an end of a nice path description, a contradiction.
  So there is an internal triad $\{\ell',p_6,p_7\}$ with $\ell' \in P_5^-$.
  Since $\{p_2,p_3,p_4\}$ is independent, the dual of \cref{picircuits,pflantriad}, and orthogonality, imply that $\ell' =p_3$, so $\{p_3,p_6,p_7\}$ is a triad.
  But this contradicts orthogonality, since $p_3$ is in a circuit contained in $P_5^-$.
\end{subproof}

The \lcnamecref{nodeltripleprop} now follows from \cref{ndtendgame,ndtendgame2}.
\end{proof}

\begin{corollary}
  \label{nodeltriplematrices}
  Suppose there is a pair $\{a,b\} \subseteq E(M)$ such that $M \ba a,b$ is $3$-connected with a $\utfutf$-minor,
  $M$ has no delete triples, and $|E(M)| \ge 16$.
  Then $M \ba a,b \cong M[I|A_i]$, for some $i \in \{1,2,3\}$, where each $A_i$ is a $\mathbb{U}_2$-matrix as follows:

$$A_1= \kbordermatrix{
      & a' &  p_1  & p_5 & p_5' &    p_9 & b' \\
 p_1' &  1 & \alpha_1 &   0 &    0 &      0 &  0 \\
 p_2  &  1 & \alpha_1 &   1 &    1 &      0 &  0 \\
 p_3  &  1 &     1 &   1 &    1 &      0 &  0 \\
 p_4  &  1 &     1 &   0 &    1 &      0 &  0 \\
 p_6  &  0 &     0 &   0 &    1 &      1 &  1 \\
 p_7  &  0 &    0  &   1 &    1 &      1 &  1 \\
 p_8  &  0 &     0 &   1 &    1 & \alpha_2 &  1 \\
 p_9' &  0 &     0 &   0 &    0 & \alpha_2 &  1 \\
},$$

$$A_2= \kbordermatrix{
      & a' &  p_1  & p_5 & p_5' &    p_9 & b' \\
 p_1' &  1 & \alpha_1 &   0 &    0 &      0 &  0 \\
 p_2  &  1 & \alpha_1 &   1 &    1 &      0 &  0 \\
 p_3  &  1 &     1 &   1 &    1 &      0 &  0 \\
 p_4  &  1 &     1 &   0 &    1 &      0 &  0 \\
 p_6  &  0 &     0 &   1 &    0 &      1 &  1 \\
 p_7  &  0 &    0  &   1 &    1 &      1 &  1 \\
 p_8  &  0 &     0 &   1 &    1 & \alpha_2 &  1 \\
 p_9' &  0 &     0 &   0 &    0 & \alpha_2 &  1 \\
},$$

$$A_3= \kbordermatrix{
     & p_1'' &     p_2 & p_2' & p_3    &     p_7 &    b' \\
 a'  &     1 &  \alpha_1 &  0   & \alpha_1 &       0 &     0 \\
p_1  &     0 &       1 &  0   &  1     &       0 &     0 \\
p_1' &     1 &      1  &  0   & 1      &       0 &     0 \\
p_3' &     1 &      1  &  1   & 1      &       0 &     0 \\
p_4  &     1 &       1 &  1   & 1      & \alpha_2-1 &     1 \\
p_5  &     1 &       1 &  1   & 0      & \alpha_2-1 &     1 \\
p_6  &     1 &       1 &  1   & 0      &   \alpha_2 &     1 \\
p_7' &     0 &       0 &  0   & 0      &   \alpha_2 &     1 \\
 }.$$
\end{corollary}


\begin{proof}
  We apply \cref{nodeltripleprop} and observe that, since $|E(M)| \ge 16$, \cref{nodeltripleprop}\ref{ndtendgameoutcomes16} holds.
  It remains to find the matroids satisfying \ref{ndtendgameoutcomes16}\ref{ndtendgameoutcome16a} or \ref{ndtendgameoutcomes16}\ref{ndtendgameoutcome16b}, and $\mathbb{P}$-representations for these matroids. 

  These can be found by hand; here we do not give all the details, but observe a few key points.
  Let $M' = M \ba a,b$.
  Observe that both ends of the nice path descriptions for $M'$ are cosegments, in either case~\ref{ndtendgameoutcomes16}\ref{ndtendgameoutcome16a} or case~\ref{ndtendgameoutcomes16}\ref{ndtendgameoutcome16b}.
  By \cref{pathdescends}\ref{pde2}, each of these ends has a unique $\utfutf$-deletable element. 
  Up to labels we assume that $a'$ is the unique $\utfutf$-deletable at one end; similarly $b'$ is the unique $\utfutf$-deletable element at the other end.
  Consider when \ref{ndtendgameoutcomes16}\ref{ndtendgameoutcome16a} holds.
  Recall that $M'$ has no triangles, and thus $\{a',p_1',p_1,p_2\}$, and $\{p_8,p_9,p_9',b'\}$ are circuits. 
  Since $M'$ is $\utfutf$-fragile, it can be argued that $\{a',p_1,p_3,p_4\}$, $\{p_6,p_7,p_9,b'\}$, and $\{p_4,p_5,p_5',p_6\}$ are circuits,
  and $\{p_1',p_2,p_8,p_{9}'\}$ is a cocircuit.
  We obtain $M_1$ if $(q,q')=(p_5,p_5')$ and $M_2$ if $(q,q')=(p_5',p_5)$.
  When \ref{ndtendgameoutcomes16}\ref{ndtendgameoutcome16b} holds, it can be argued that $\{p_2,p_3,p_5,p_6\}$, $\{p_4,p_5,p_7,b'\}$, and $\{p_6,p_7,p_7',b'\}$ are circuits and $\{p_1'',p_2,p_2',p_6,p_7'\}$ is a cocircuit; we obtain $M_3$ in this case.

  Alternatively, these matroids and representations can be found by a computer search on all $\mathbb{P}$-representable matroids on 14 elements, for $\mathbb{P} \in \{\mathbb{U}_2, \mathbb{H}_5\}$ (recall that all $3$-connected $\mathbb{U}_2$-representable matroids with a $\utfutf$-minor, and at most 15 elements, were enumerated in \cite{BP20}).
  This approach was used to verify the correctness of the representations found by hand.
\end{proof}

\begin{proof}[Proof of \cref{nodeltriplethm}]
  By \cref{nodeltriplematrices}, if $M$ has no delete triples and $|E(M)| \ge 16$, then $|E(M)|=16$ and $M \ba a,b$ is $2$-regular and isomorphic to $M[I|A_i]$, for some $i \in \{1,2,3\}$, with $A_i$ as described in \cref{nodeltriplematrices}.
  By a computer search, we found all $\mathbb{H}_5$-representable matroids that are single-element extensions of these three matroids. 
  Fix some $i \in \{1,2,3\}$.
  For each (not necessarily distinct) pair of extensions of $M[I|A_i]$, say $N_1$ and $N_2$, we found each matroid~$M$ with a pair $\{a,b\}$ such that $M \ba a \cong N_1$ and $M \ba b \cong N_2$, using the splicing techniques described in \cite{BP20}.
  We then discarded any such matroid $M$ with at least one triad.
  For each of the 
  matroids, we verified the matroid indeed has a delete triple.
  For example, for $i=1$ there were $56$ pairwise non-isomorphic single-element extensions that were $2$-regular, and a further $7$ pairwise non-isomorphic single-element extensions that were only $\mathbb{H}_5$-representable; after splicing a pair of these matroids, $368$ matroids were obtained that had no triads.
\end{proof}

\section{Proof of \texorpdfstring{\cref{intro1}}{Theorem 1.1}}
\label{concsec}

Combining \cref{deltripleprop,nodeltriplethm}, we prove our main result.

\begin{proof}[Proof of \cref{intro1}]
  Observe that $U_{2,5}$ is a non-binary $3$-connected strong stabilizer for the class of $\mathbb{P}$-representable matroids, by \cref{u2stabs}.
  We may assume that $M$ has a $U_{2,5}$-minor, for otherwise $M$ is an excluded minor
  for the class of near-regular matroids, in which case $|E(M)| \le 8$ \cite{HMvZ11}.
  Assume that $|E(M)| \ge |E(U_{2,5})| + 11 = 16$.
  By \cref{notriads2}, there exists a matroid~$M_1 \in \Delta^*(M)$ such that $M_1$ has a pair of elements $\{a,b\}$ for which $M_1 \ba a,b$ is $3$-connected and has a $\utfutf$-minor, and $M_1$ has no triads.
  By \cref{osvdelta}, $M_1$ is an excluded minor for the class of $\mathbb{P}$-representable matroids.
  By \cref{utfutffragile}, $M_1 \ba a,b$ is $\utfutf$-fragile.
  If $M_1$ has a delete triple, then, by \cref{deltripleprop}, $|E(M_1)| \le 15$; whereas if $M_1$ has no delete triples, then, by \cref{nodeltriplethm}, $|E(M_1)| \le 15$.
  But $|E(M)| = |E(M_1)|$, so this is contradictory.
  We deduce that $|E(M)| \le 15$, as required.
\end{proof}

\bibliographystyle{acm}
\bibliography{lib}

\appendix

\section*{Appendix: matroids appearing as excluded minors}

The matroids $P_6$, $F_7$, $F_7^-$, $P_8$, and $P_8^=$ are well known (see Oxley~\cite[Appendix]{Oxley11}, for example), as is the rank-$r$ uniform matroid on $n$ elements, $U_{r,n}$.
We now provide representations for other matroids appearing as excluded minors in this paper.  Note that we provide \emph{reduced} representations: that is, we provide a matrix $A$ such that $M \cong M[I|A]$.
For maximum generality, we provide a representation for a matroid $M$ over $\mathbb{P}_M$, the universal partial field of $M$, but in each case, we describe how one can obtain a finite field representation of $M$.

The following are reduced
$\mathbb{K}_2$-representations for $F_7^=$, $\TQ_8$, and $P_8^-$, respectively. 
The partial field $\mathbb{K}_2$ is formally defined in \cite{PvZ10b}, but note that a $\GF(4)$-representation can be obtained by setting $\alpha = \omega$, where $\omega$ is a generator of $\GF(4)$.
Alternatively, two inequivalent $\GF(5)$-representations can be obtained by setting $\alpha \in \{2,3\}$.

The matroid $F_7^=$ can be obtained by relaxing a circuit-hyperplane of $F_7^-$, and it has the following reduced representation:

$$\begin{bmatrix}
    1 & 1 & 0 & 1 \\
    1 & 0 & 1 & 1 \\
    0 & 1 & 1 & \alpha
\end{bmatrix}$$

The matroid $\TQ_8$ was introduced in \cite{BP20}, and it has the following reduced representation:

$$\begin{bmatrix}
    0 & 1 & 1 & 1\\
    1 & \alpha & 1 & \alpha + 1\\
    1 & 1 & 1 & \alpha \\
    1 & \alpha + 1 & \alpha & \alpha
\end{bmatrix}$$

Finally, $P_8^-$ can be obtained by relaxing one of the pair of disjoint circuit-hyperplanes of $P_8$, and it has the following reduced representation: 

$$\begin{bmatrix}
           1 & \alpha+1 &      0 & 1\\
    \alpha+1 & \alpha+1 & \alpha & 1\\
           0 & \alpha   & \alpha & 1\\
           1 & 1        &      1 & 0
\end{bmatrix}$$

\medskip

There is, up to isomorphism, a unique matroid that can be obtained by deleting an element from the affine geometry $\AG(2,3)$; following \cite{HMvZ11}, we denote this matroid $\AG(2,3)\ba e$.
We use $(\AG(2,3)\ba e)^{\Delta Y}$ to denote the self-dual matroid that can be obtained by performing a single \dY\ exchange on a triangle of $\AG(2,3)\ba e$.

\medskip

For a matroid $M$, let $M+e$ denote the free single-element extension of $M$.
Consider the matroids that can be obtained by relaxing zero or more circuit-hyperplanes starting from the Fano matroid $F_7$.
We obtain the sequence $$F_7, F_7^-, F_7^=, \{H_7,M(K_4)+e\}, \{\mathcal{W}^3+e,\Lambda_3\}, Q_6+e, P_6+e, U_{3,7}$$
where $\Lambda_3$ denotes the rank-$3$ tipped free spike.
Note that $H_7$ is the dual of the matroid (unique up to isomorphism) that can be obtained by performing a \dY\ exchange on a triangle of $M(K_4)+e$.

We first provide reduced $\mathbb{H}_4$-representations of $M(K_4)+e$ and $\Lambda_3$ (see \cite{PvZ10b} for a definition of the partial field $\mathbb{H}_4$).
Note that four inequivalent $\GF(5)$-representations can be obtained by substituting $(\alpha,\beta) \in \{(2,2), (3,3), (3,4), (4,3)\}$.

\noindent
\begin{tabular}{p{.45\textwidth} p{0.45\textwidth}}
$M(K_4)+e$:
$\begin{bmatrix}
 1 &  \alpha & \alpha & 1 \\
 0 &       1 &      1 & 1 \\
 1 &       0 & \alpha & \frac{\beta(\alpha-1)}{1-\beta} \\
\end{bmatrix}$ & %
$\Lambda_3$:
$\begin{bmatrix}
 1 &      1 & \alpha+\beta-2\alpha\beta & \alpha\beta-1 \\
 1 & \alpha &                         0 & \alpha(\beta-1) \\
 1 &      0 &           \alpha(1-\beta) & \alpha(\beta-1) \\
\end{bmatrix}$
\end{tabular}

Finally, we provide reduced $\mathbb{H}_2$-representations of $\mathcal{W}^3+e$ and $Q_6+e$ (see \cite{PvZ10b} for a definition of the partial field $\mathbb{H}_2$).
Note that two inequivalent $\GF(5)$-representations can be obtained by substituting $i \in \{2,3\}$.

\noindent
\begin{tabular}{p{.45\textwidth} p{0.45\textwidth}}
$\mathcal{W}^3+e$:
$\begin{bmatrix}
  1 & 0 & i & 1 \\
  i & 1 & 0 & 1 \\
  0 & i & 1 & 1
\end{bmatrix}$ & 
$Q_6+e$:
$\begin{bmatrix}
  \frac{i+1}{2} & 0 & i & 1 \\
  1 & 1 & 1 & 1 \\
  0 & \frac{1-i}{2} & -i & 1
\end{bmatrix}$
\end{tabular}

\end{document}